\documentclass{amsart}

\usepackage[mathscr]{eucal}
\usepackage{amssymb}
\usepackage[usenames,dvipsnames]{color}
\usepackage[normalem]{ulem}
\usepackage{amsthm}
\usepackage{bbold}
\usepackage{enumitem}
\usepackage{array}
\usepackage{amsmath}
\usepackage{wasysym}

\usepackage{hyperref}

\hypersetup{colorlinks=true,citecolor=Brown,urlcolor=Brown,linkcolor=Brown}

\numberwithin{equation}{section}
\usepackage[all]{xy}

\newdir{ >}{{}*!/-10pt/\dir{>}}

\swapnumbers 

\newtheorem{Thm}[equation]{Theorem}
\newtheorem*{Thm*}{Theorem}
\newtheorem{Prop}[equation]{Proposition}
\newtheorem{Lem}[equation]{Lemma}
\newtheorem{Cor}[equation]{Corollary}

\newtheorem{Conj}[equation]{Conjecture}

\theoremstyle{remark}
\newtheorem{Def}[equation]{Definition}
\newtheorem{Ter}[equation]{Terminology}
\newtheorem{Not}[equation]{Notation}
\newtheorem{Exa}[equation]{Example}

\newtheorem{Rem}[equation]{Remark}

\newcommand{\nc}{\newcommand}
\nc{\dmo}{\DeclareMathOperator}

\nc{\Beren}[1]{{\color{Red}#1}}
\nc{\Paul}[1]{{\color{BlueViolet}#1}}

\dmo{\Id}{Id}
\dmo{\Loc}{Loc}
\dmo{\rmK}{\textrm{\rm K}}
\dmo{\Spc}{Spc}
\dmo{\thick}{thick}
\nc{\thickt}[1]{\thick_\otimes(#1)}
\dmo{\cone}{cone}
\dmo{\End}{End}
\dmo{\rmH}{H}
\dmo{\Hom}{Hom}
\dmo{\id}{id}
\dmo{\incl}{incl}
\dmo{\Img}{Im}
\dmo{\Ker}{Ker}
\dmo{\Ind}{ind}
\dmo{\CoInd}{coind}
\dmo{\Res}{res}
\dmo{\Infl}{infl}
\dmo{\triv}{triv}
\dmo{\Tel}{Tel} 
\dmo{\Mod}{Mod}%
\dmo{\opname}{op}
\dmo{\SH}{SH}
\dmo{\smallb}{b}
\dmo{\Spec}{Spec}
\dmo{\supp}{supp}
\nc{\SHc}{{\SH^c}}
\nc{\SHp}{{\SH_{(p)}}}
\nc{\SHcp}{{\SH^c_{(p)}}}
\nc{\SHG}{\SH(G)}
\nc{\SHGp}{\SH(G)_{(p)}}
\nc{\SHGc}{\SHG^c}
\nc{\SHGcp}{\SHG^c_{(p)}}
\nc{\quadtext}[1]{\quad\textrm{#1}\quad}
\nc{\qquadtext}[1]{\qquad\textrm{#1}\qquad}
\nc{\adj}{\dashv}
\nc{\adjto}{\rightleftarrows}
\nc{\bbN}{\mathbb{N}}
\nc{\bbQ}{\mathbb{Q}}
\nc{\bbZ}{\mathbb{Z}}
\nc{\cat}[1]{\mathscr{#1}}
\nc{\ie}{{\sl i.e.}, }
\nc{\into}{\mathop{\rightarrowtail}}
\nc{\inv}{^{-1}}
\nc{\isoto}{\mathop{\overset{\sim}\to}}
\nc{\isotoo}{\mathop{\overset{\sim}\too}}
\nc{\onto}{\mathop{\twoheadrightarrow}}
\nc{\too}{\mathop{\longrightarrow}\limits}
\nc{\mapstoo}{\longmapsto}
\nc{\adh}[1]{\overline{#1}}
\nc{\adhpt}[1]{\adh{\{#1\}}}
\nc{\aka}{{a.\,k.\,a.}\ }
\nc{\calF}{\mathcal{F}}
\nc{\eg}{{\sl e.\,g.}}
\nc{\Homcat}[1]{\Hom_{\cat #1}}
\nc{\hook}{\hookrightarrow}
\nc{\ideal}[1]{\langle #1\rangle}
\nc{\ihom}{{\underline{\hom}}}
\nc{\Mid}{\,\big|\,}
\nc{\MMod}{\,\text{-}\Mod}%
\nc{\op}{^{\opname}}
\nc{\oto}[1]{\overset{#1}\to}
\nc{\otoo}[1]{\overset{#1}{\,\too\,}}
\nc{\sminus}{\!\smallsetminus\!}
\nc{\poplus}[1]{^{\oplus #1}}%
\nc{\potimes}[1]{^{\otimes #1}}
\nc{\sbull}{{\scriptscriptstyle\bullet}}
\nc{\SET}[2]{\big\{\,#1\Mid#2\,\big\}}
\nc{\SpcK}{\Spc(\cat K)}
\nc{\then}{\Rightarrow}
\nc{\unit}{\mathbb{1}}
\nc{\xra}{\xrightarrow}
\nc{\phigeom}[1]{\widetilde{\Phi}^{#1}}
\nc{\phigeomb}[1]{\Phi^{#1}}
\dmo{\Oname}{O}
\dmo{\proper}{proper}
\dmo{\lenormal}{\unlhd}
\dmo{\lnormal}{\lhd}
\nc{\normal}{\trianglelefteq}
\nc{\Op}{\Oname^p}
\nc{\Oq}{\Oname^q}

\newcounter{enum-resume-hack}

\begin{document}


\title[The spectrum of the equivariant stable homotopy category]{The spectrum of the equivariant stable homotopy category of a finite group}
\author{Paul Balmer}
\author{Beren Sanders}
\date{August 1, 2016}

\address{Paul Balmer, Mathematics Department, UCLA, Los Angeles, CA 90095-1555, USA}
\email{balmer@math.ucla.edu}
\urladdr{http://www.math.ucla.edu/$\sim$balmer}

\address{Beren Sanders, Department of Mathematical Sciences, University of Copenhagen, Universitetsparken 5, 2100 Copenhagen {\O}, Denmark}
\email{sanders@math.ku.dk}
\urladdr{http://beren.sites.ku.dk}

\begin{abstract}
We study the spectrum of prime ideals in the tensor-triangulated category of compact equivariant spectra over a finite group. We completely describe this spectrum \emph{as a set} for all finite groups.  We also make significant progress in determining its topology and obtain a complete answer for groups of square-free order. For general finite groups, we describe the topology up to an unresolved indeterminacy, which we reduce to the case of \mbox{$p$-groups.} We then translate the remaining unresolved question into a new chromatic blue-shift phenomenon for Tate cohomology. Finally, we draw conclusions on the classification of thick tensor ideals.
\end{abstract}


\thanks{First-named author supported by NSF grant~DMS-1303073.}
\thanks{Second-named author supported by the Danish National Research Foundation through the Centre for Symmetry and Deformation (DNRF92).}

\maketitle

\tableofcontents

\section{Introduction}
\medskip

Let $G$ be a finite group and let $\SHG$ denote the $G$-equivariant stable homotopy category, with respect to a complete $G$-universe. It is a tensor-triangulated category. We want to determine the \mbox{tensor-triangular} spectrum $\Spc(\SHGc)$ of its subcategory of compact objects~$\SHGc$.

As we explain below, this question is the equivariant analogue of the celebrated chromatic filtration in classical stable homotopy theory, due to Devinatz-Hopkins-Smith~\cite{DevinatzHopkinsSmith88,HopkinsSmith98}. Important work on this topic has been done by Strickland~\cite{Strickland12up}; see also Joachimi~\cite[Chap.\,3]{Joachimi15pp}. However, the question remained open even for the group of order two. We return to a discussion of earlier literature at the end of the Introduction, after stating our results.

\smallbreak

For a tensor-triangulated category~$\cat K$, like the above $\SHGc$, the topological space $\Spc(\cat K):=\SET{\cat P\subsetneq \cat K}{\cat P\textrm{ is prime}}$ consists of all proper \emph{tt-ideals} $\cat P$ (thick triangulated $\otimes$-ideal subcategories of~$\cat K$) which are \emph{prime} (\mbox{$x\otimes y \in \cat P$} implies $x\in \cat P$ or $y\in\cat P$); the \emph{support} of an object $x \in \cat K$ is the closed subset $\supp(x):=\SET{\cat P\in\Spc(\cat K)}{x \notin\cat P}$ and these closed subsets form a closed basis for the topology of $\Spc(\cat K)$. Conceptually, this spectrum is the universal space which carries a decent notion of support for objects of~$\cat K$. (See~\cite{Balmer05a,BalmerICM} for more details.)

Computing the topological space $\Spc(\cat K)$ is essentially equivalent to classifying all tt-ideals of~$\cat K$, and the latter may be thought of as classifying the objects of $\cat K$ up to the naturally available tensor-triangulated structure.  More precisely, a classification of tt-ideals describes when two objects can be built from each other by taking suspensions, direct sums, direct summands, cofiber sequences, and tensor products.  It is a major development of the last twenty years that such classification theorems can be obtained for highly non-trivial categories arising in stable homotopy theory, algebraic geometry, and modular representation theory \cite{HopkinsSmith98,Thomason97,BensonCarlsonRickard97,FriedlanderPevtsova07}.

Up until now, such classification theorems were obtained before the determination of the corresponding tensor-triangular spectrum. The category~$\cat K=\SHGc$ is a first case where the computation of~$\SpcK$ can be achieved beforehand, and the classification of tt-ideals deduced as a corollary.  Indeed, the fact that we attack the spectrum directly via methods of tensor-triangular geometry is one of the novel aspects of our approach.

The subject of classifying tt-ideals was born in non-equivariant stable homotopy theory. In that case, the $\otimes$-unit (the sphere spectrum) generates the category, hence every thick subcategory is automatically a $\otimes$-ideal. It then follows from the Thick Subcategory Theorem of Hopkins and Smith~\cite{HopkinsSmith98} that the corresponding $\Spc(\SHc)$ is the space depicted at the top of the following picture (see~\cite[\S 9]{Balmer10b}):
\begin{equation}
\label{eq:Spc(SH)}%
\vcenter{\xymatrix@C=.8em @R=.4em{
&&\cat C_{2,\infty} \ar@{-}[d]
&\cat C_{3,\infty} \ar@{-}[d]
&& \kern-2em{\cdots}
&\cat C_{p,\infty} \ar@{-}[d]
& {\cdots}
\\
\Spc(\SHc)= \ar[ddddddd]_-{\displaystyle\rho_\SHc}
&&{\vdots} \ar@{-}[d]
& {\vdots} \ar@{-}[d]
&&& {\vdots} \ar@{-}[d]
\\
&&\cat C_{2,n+1} \ar@{-}[d]
& \cat C_{3,n+1} \ar@{-}[d]
&& \kern-2em{\cdots}
& \cat C_{p,n+1} \ar@{-}[d]
& {\cdots}
\\
&&\cat C_{2,n} \ar@{-}[d]
& \cat C_{3,n} \ar@{-}[d]
&& \kern-2em{\cdots}
& \cat C_{p,n} \ar@{-}[d]
& {\cdots}
\\
&&{\vdots} \ar@{-}[d]
& {\vdots} \ar@{-}[d]
&&& {\vdots} \ar@{-}[d]
\\
&&\cat C_{2,2} \ar@{-}[rrd]
& \cat C_{3,2} \ar@{-}[rd]
&& \kern-2em{\cdots}
& \cat C_{p,2} \ar@{-}[lld]
& {\cdots}
\\
&&&& \cat C_{0,1}
\\\\
\Spec(\bbZ)=
&&2\bbZ \ar@{-}[rrd]
& 3\bbZ \ar@{-}[rd]
&& \kern-2em{\cdots}
& p\bbZ \ar@{-}[lld]
& {\cdots}
\\
&&&& (0)}}
\end{equation}
Here $\cat C_{p,n}$ is the kernel in~$\SHc$ of the $p$-local $(n-1)$-th Morava $K$-theory (composed with localization $\SHc\to \SHcp$ at~$p$).  In particular, $\cat C_{p,1}=\SH^{c,\textrm{tor}}=:\cat C_{0,1}$ is the subcategory of torsion finite spectra, independently of~$p$, while $\cat C_{p,\infty}=\cap_{n\geq 1}\cat C_{p,n}=\Ker(\SHc\to \SH^c_{(p)})$ is the subcategory of $p$-acyclic finite spectra. The closure of each point consists of the points above it in the picture (related by an upwards line), so that each $\cat C_{p,\infty}$ is a closed point, while $\cat C_{0,1}$ is a dense point. A general closed subset is the closure of finitely many points.

For an arbitrary $\otimes$-triangulated category~$\cat K$, there is a so-called \emph{comparison map}
\begin{equation}
\label{eq:comp}%
\begin{array}{ccc}
\rho_{\cat K}:\quad\Spc(\cat K) & \too & \Spec(\End_{\cat K}(\unit))
\\[.5em]
\qquad \cat P & \mapstoo & \SET{f:\unit\to \unit}{\cone(f)\notin\cat P}
\end{array}
\end{equation}
from the tensor-triangular spectrum of~$\cat K$ to the Zariski spectrum of the endomorphism ring of the $\otimes$-unit object $\unit$. For $\cat K = \SHc$, we have $\End_{\cat K}(\unit) = \mathbb Z$ and $\rho_{\SHc}$ is the projection $\cat C_{p,*}\mapsto p\bbZ$ displayed in picture~\eqref{eq:Spc(SH)}. This viewpoint highlights how the spectrum of the non-equivariant stable homotopy category is a chromatic refinement of $\Spec(\mathbb Z)$: The Morava $K$-theories provide an interpolation between $\rmH\mathbb Q$ and~$\rmH\mathbb{F}_p$, or in other words, between the prime $(0)$ and the prime~$p\bbZ$.

For our tensor-triangulated category $\cat K = \SHGc$, the ring $\End_{\cat K}(\unit)$ is isomorphic to the Burnside ring~$A(G)$, whose Zariski spectrum was determined by Dress~\cite{Dress69}, and the comparison map $\rho_{\SHGc} : \Spc(\SHGc) \to \Spec(A(G))$ will similarly exhibit $\Spc(\SHGc)$ as a kind of ``chromatic refinement'' of~$\Spec(A(G))$.

We review Dress's result in Section~\ref{se:burnside}.
In short, $\Spec(A(G))$ is covered by copies of $\Spec(\bbZ)$ coming via the
``$H$-fixed points'' ring homomorphisms
\begin{eqnarray*}
\kern5em	f^H:A(G) &\too& A(1)=\bbZ \\
	\left[ X \right] &\mapstoo& | X^H | \qquad(\textrm{for all finite $G$-sets }X)
\end{eqnarray*}
for each subgroup~$H \le G$ (up to conjugacy). These copies of~$\Spec(\bbZ)$ overlap at some primes~$p$ dividing the order of~$G$; namely, for two subgroups $H,K\le G$, the primes $\mathfrak{p}(H,p):=(f^H)\inv(p\bbZ)$ and $\mathfrak{p}(K,p)=(f^K)\inv(p\bbZ)$ coincide in~$\Spec(A(G))$ when $\Op(H)$ and $\Op(K)$ are conjugate in~$G$; here $\Op(H)$ denotes the smallest normal subgroup of~$H$ of index a power of~$p$ (see details in Definition~\ref{def:p-perfect}).

For every subgroup~$H\le G$, the above $H$-fixed-points homomorphism $f^H:A(G)=\End_{\SHGc}(\unit) \to \End_{\SHc}(\unit)=A(1)$ is precisely the map induced on endomorphism rings by the \emph{geometric $H$-fixed points} functor
\[
\phigeomb{H} :
\SHGc \to \SHc
\]
recalled in Section~\ref{se:background}, Part~\ref{it:phi^N} below.  Moreover, since $\phigeomb{H}$ is a tensor-triangulated functor it provides a means for pulling back the non-equivariant primes~$\cat C_{p,n}$ to obtain ``$G$-equivariant primes", \ie primes in~$\SHGc$, as follows:
\[
\cat P(H,p,n) :=
(\phigeomb{H})^{-1}(\cat C_{p,n})\quad \in\,\Spc(\SHGc).
\]
Our first main result states that every prime of~$\SHGc$ is of this form, and that the overlap is minimal.
\begin{Thm*}[Theorems~\ref{thm:the-set} and~\ref{thm:prime-uniqueness}]
All $G$-equivariant primes are obtained by pulling back non-equivariant primes via geometric fixed point functors with respect to the various subgroups~$H\le G$. Moreover, there is no redundancy, in the sense that the primes $\cat P(H,p,n)$ and $\cat P(H',p',n')$ coincide only if the subgroups $H$ and $H'$ are conjugate in~$G$ and the chromatic primes~$\cat C_{p,n}$ and~$\cat C_{p',n'}$ coincide in~$\SHc$.
\end{Thm*}
We thus obtain a complete description of $\Spc(\SHGc)$ as a set, for every finite group~$G$.  Interestingly, this already shows that the height one collisions observed by Dress in the Zariski spectrum of the Burnside ring do not occur in the tensor-triangular spectrum of~$\SHGc$. Thus, $\Spc(\SHGc)$ is not only a chromatic refinement of $\Spec(A(G))$ but it also remembers some group-theoretic information which was lost in~$\Spec(A(G))$. Another interesting consequence of these results is an equivariant version of the Nilpotence Theorem (Theorem~\ref{thm:nilpotence-thm}). The latter was obtained by Strickland~\cite{Strickland12up} by different methods.

There are two main ingredients in the proofs of Theorems~\ref{thm:the-set} and~\ref{thm:prime-uniqueness}. The first is the exploitation of the fact that the geometric fixed point functor $\phigeomb{G} : \SH(G) \to \SH$ is a finite Bousfield localization. Although this property already appears in~\cite{LewisMaySteinbergerMcClure86}, it does not seem to have been given the attention it deserves in the literature. We will make a systematic use of finite localizations on the category of~$G$-equivariant spectra --- both in the equivariant and chromatic directions.  However, we will need to go beyond such ``classical" localizations, towards the tensor-triangular analogue of the \'etale topology. In this vein, the second crucial ingredient is a result of~\cite{BalmerDellAmbrogioSanders15} which states that the restriction functor $\Res^G_H : \SH(G) \to \SH(H)$ is a separable extension (\aka a finite \'{e}tale extension). We remind the reader in~\ref{it:A-Mod} below. This fact enables us to apply results from~\cite{Balmer13ppb} on the tensor-triangular geometry of separable extensions.  These techniques give us complete control of the geometric fixed point functors $\phigeomb{H}$ at the level of tensor-triangular spectra, as each $\phigeomb{H}$ can be understood as a finite \'{e}tale extension followed by a finite localization.

With this knowledge of the underlying set of~$\Spc(\SHGc)$, we focus on its topology from Section~\ref{se:pre-top} onwards. This question reduces (cf.~Proposition~\ref{prop:finite-union-irred}) to understanding the inclusions between equivariant primes: $\cat P(K,q,n) \subset \cat P(H,p,m)$. For primes corresponding to the same conjugacy class ($K\sim_G H$), these inclusions precisely match the inclusions in the original non-equivariant case, \ie the inclusions displayed in diagram~\eqref{eq:Spc(SH)}. On the other hand, the comparison map $\rho_{\SHGc}:\Spc(\SHGc)\to \Spec(A(G))$ greatly restricts the possible inclusions among primes associated to different conjugacy classes of subgroups. Ultimately, the determination of the topology reduces to the question of understanding possible inclusions $\cat P(K,p,n) \subset \cat P(H,p,m)$ for the same prime~$p$ dividing the order of~$G$ and for ``$p$-subnormal subgroups'' $K \le H$ in the sense of~\ref{ter:subnormal}.

We first attack this question for $G=C_p$, where we are able to give a complete answer by utilizing results of Hovey-Sadofsky~\cite{HoveySadofsky96} and Kuhn~\cite{Kuhn04} on blue-shift phenomena in Tate cohomology. It turns out that there is a shift by one: $\cat P(1,p,n) \subset \cat P(C_p,p,m)$ iff $n \ge m+1$. Here is a complete picture of the spectrum in this case, together with the comparison map to the Zariski spectrum of~$A(C_p)$:
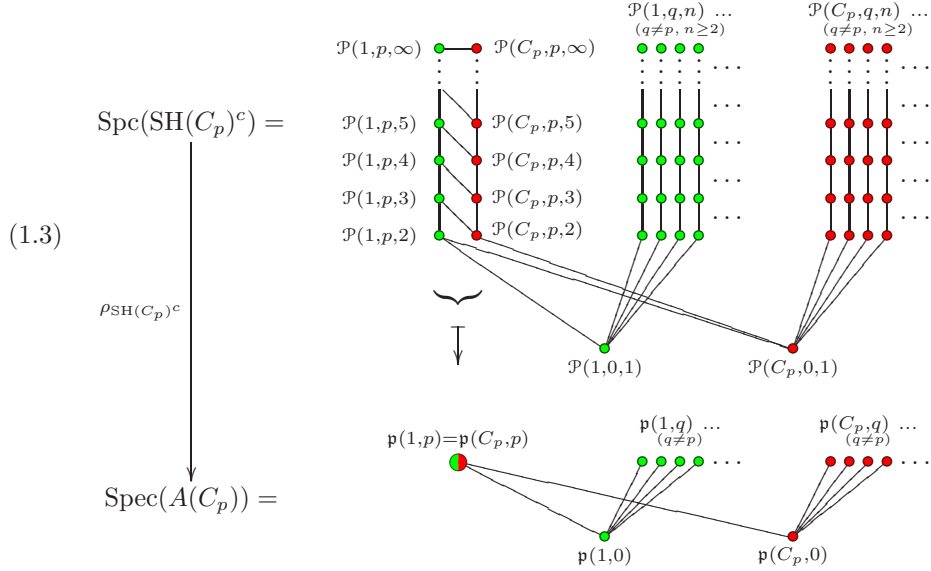
\begin{equation}
\label{eq:Spc(C_p)}%
\xy
{\ar@{->}_{\rho_{\SH(C_p)^c}} (-40,12.5)*{};(-40,-32.5)*{}};
(-40,15)*{\Spc(\SH(C_p)^c) = };
(-40,-35)*{\Spec(A(C_p)) = };
{\ar@{-} (15,-15)*{};(20,0)*{}};
{\ar@{-} (15,-15)*{};(22.5,0)*{}};
{\ar@{-} (15,-15)*{};(25,0)*{}};
{\ar@{-} (15,-15)*{};(27.5,0)*{}};
{\ar@{-} (40,-15)*{};(45,0)*{}};
{\ar@{-} (40,-15)*{};(47.5,0)*{}};
{\ar@{-} (40,-15)*{};(50,0)*{}};
{\ar@{-} (40,-15)*{};(52.5,0)*{}};
{\ar@{-} (15,-15)*{};(-7,0)*{}};
{\ar@{-} (40,-15)*{};(-7,0)*{}};
{\ar@{-} (40,-15)*{};(-2,0)*{}};
{\ar@{-} (-2,0)*{};(-7,5)*{}};
{\ar@{-} (-2,5)*{};(-7,10)*{}};
{\ar@{-} (-2,10)*{};(-7,15)*{}};
{\ar@{-} (-2,15)*{};(-6.5,19.5)*{}};
{\ar@{-} (-2,25)*{};(-7,25)*{}};
  (-15,25)*{\scriptstyle \cat P(1,p,\infty)};
  (-15,0)*{\scriptstyle \cat P(1,p,2)};
  (-15,5)*{\scriptstyle \cat P(1,p,3)};
  (-15,10)*{\scriptstyle \cat P(1,p,4)};
  (-15,15)*{\scriptstyle \cat P(1,p,5)};
{\ar@{-} (-7,0)*{};(-7,5)*{}};
{\ar@{-} (-7,5)*{};(-7,10)*{}};
{\ar@{-} (-7,10)*{};(-7,15)*{}};
{\ar@{-} (-7,15)*{};(-7,19.5)*{}};
 (-7,0)*{\color{green}\bullet};
 (-7,0)*{\circ};
  (-7,5)*{\color{green}\bullet};
  (-7,5)*{\circ};
  (-7,10)*{\color{green}\bullet};
  (-7,10)*{\circ};
  (-7,15)*{\color{green}\bullet};
  (-7,15)*{\circ};
  (-7,23)*{\vdots};
  (-7,25)*{\color{green}\bullet};
  (-7,25)*{\circ};
  (7,25)*{\scriptstyle \cat P(C_p,p,\infty)};
  (6,0.75)*{\scriptstyle \cat P(C_p,p,2)};
  (6,5)*{\scriptstyle \cat P(C_p,p,3)};
  (6,10)*{\scriptstyle \cat P(C_p,p,4)};
  (6,15)*{\scriptstyle \cat P(C_p,p,5)};
{\ar@{-} (-2,0)*{};(-2,5)*{}};
{\ar@{-} (-2,5)*{};(-2,10)*{}};
{\ar@{-} (-2,10)*{};(-2,15)*{}};
{\ar@{-} (-2,15)*{};(-2,19.5)*{}};
 (-2,0)*{\color{red}\bullet};
 (-2,0)*{\circ};
  (-2,5)*{\color{red}\bullet};
  (-2,5)*{\circ};
  (-2,10)*{\color{red}\bullet};
  (-2,10)*{\circ};
  (-2,15)*{\color{red}\bullet};
  (-2,15)*{\circ};
  (-2,23)*{\vdots};
  (-2,25)*{\color{red}\bullet};
  (-2,25)*{\circ};
 (25,30)*{\scriptstyle \cat P(1,q,n)\ \ldots};
 (25,27.5)*{\scriptscriptstyle (q\neq p,\, n\ge2)};
{\ar@{-} (20,0)*{};(20,5)*{}};
{\ar@{-} (20,5)*{};(20,10)*{}};
{\ar@{-} (20,10)*{};(20,15)*{}};
{\ar@{-} (20,15)*{};(20,19.5)*{}};
 (20,0)*{\color{green}\bullet};
 (20,0)*{\circ};
  (20,5)*{\color{green}\bullet};
  (20,5)*{\circ};
  (20,10)*{\color{green}\bullet};
  (20,10)*{\circ};
  (20,15)*{\color{green}\bullet};
  (20,15)*{\circ};
  (20,23)*{\vdots};
  (20,25)*{\color{green}\bullet};
  (20,25)*{\circ};
{\ar@{-} (22.5,0)*{};(22.5,5)*{}};
{\ar@{-} (22.5,5)*{};(22.5,10)*{}};
{\ar@{-} (22.5,10)*{};(22.5,15)*{}};
{\ar@{-} (22.5,15)*{};(22.5,19.5)*{}};
 (22.5,0)*{\color{green}\bullet};
 (22.5,0)*{\circ};
  (22.5,5)*{\color{green}\bullet};
  (22.5,5)*{\circ};
  (22.5,10)*{\color{green}\bullet};
  (22.5,10)*{\circ};
  (22.5,15)*{\color{green}\bullet};
  (22.5,15)*{\circ};
  (22.5,23)*{\vdots};
  (22.5,25)*{\color{green}\bullet};
  (22.5,25)*{\circ};
{\ar@{-} (25,0)*{};(25,5)*{}};
{\ar@{-} (25,5)*{};(25,10)*{}};
{\ar@{-} (25,10)*{};(25,15)*{}};
{\ar@{-} (25,15)*{};(25,19.5)*{}};
     (25,0)*{\color{green}\bullet};
	 (25,0)*{\circ};
  (25,5)*{\color{green}\bullet};
  (25,5)*{\circ};
  (25,10)*{\color{green}\bullet};
  (25,10)*{\circ};
  (25,15)*{\color{green}\bullet};
  (25,15)*{\circ};
  (25,23)*{\vdots};
  (25,25)*{\color{green}\bullet};
  (25,25)*{\circ};
{\ar@{-} (27.5,0)*{};(27.5,5)*{}};
{\ar@{-} (27.5,5)*{};(27.5,10)*{}};
{\ar@{-} (27.5,10)*{};(27.5,15)*{}};
{\ar@{-} (27.5,15)*{};(27.5,19.5)*{}};
(27.5,0)*{\color{green}\bullet};
(27.5,0)*{\circ};
  (27.5,5)*{\color{green}\bullet};
  (27.5,5)*{\circ};
  (27.5,10)*{\color{green}\bullet};
  (27.5,10)*{\circ};
  (27.5,15)*{\color{green}\bullet};
  (27.5,15)*{\circ};
  (27.5,23)*{\vdots};
  (27.5,25)*{\color{green}\bullet};
  (27.5,25)*{\circ};
  (31.25,2.5)*{\hdots};
  (31.25,7.5)*{\hdots};
  (31.25,12.5)*{\hdots};
  (31.25,17.5)*{\hdots};
  (31.25,22.5)*{\hdots};
 (50,30)*{\scriptstyle \cat P(C_p,q,n)\ \ldots};
 (50,27.5)*{\scriptscriptstyle (q\neq p,\,n\ge2)};
{\ar@{-} (45,0)*{};(45,5)*{}};
{\ar@{-} (45,5)*{};(45,10)*{}};
{\ar@{-} (45,10)*{};(45,15)*{}};
{\ar@{-} (45,15)*{};(45,19.5)*{}};
 (45,0)*{\color{red}\bullet};
 (45,0)*{\circ};
  (45,5)*{\color{red}\bullet};
  (45,5)*{\circ};
  (45,10)*{\color{red}\bullet};
  (45,10)*{\circ};
  (45,15)*{\color{red}\bullet};
  (45,15)*{\circ};
  (45,23)*{\vdots};
  (45,25)*{\color{red}\bullet};
  (45,25)*{\circ};
{\ar@{-} (47.5,0)*{};(47.5,5)*{}};
{\ar@{-} (47.5,5)*{};(47.5,10)*{}};
{\ar@{-} (47.5,10)*{};(47.5,15)*{}};
{\ar@{-} (47.5,15)*{};(47.5,19.5)*{}};
(47.5,0)*{\color{red}\bullet};
(47.5,0)*{\circ};
  (47.5,5)*{\color{red}\bullet};
  (47.5,5)*{\circ};
  (47.5,10)*{\color{red}\bullet};
  (47.5,10)*{\circ};
  (47.5,15)*{\color{red}\bullet};
  (47.5,15)*{\circ};
  (47.5,23)*{\vdots};
  (47.5,25)*{\color{red}\bullet};
  (47.5,25)*{\circ};
{\ar@{-} (50,0)*{};(50,5)*{}};
{\ar@{-} (50,5)*{};(50,10)*{}};
{\ar@{-} (50,10)*{};(50,15)*{}};
{\ar@{-} (50,15)*{};(50,19.5)*{}};
  (50,0)*{\color{red}\bullet};
  (50,0)*{\circ};
  (50,5)*{\color{red}\bullet};
  (50,5)*{\circ};
  (50,10)*{\color{red}\bullet};
  (50,10)*{\circ};
  (50,15)*{\color{red}\bullet};
  (50,15)*{\circ};
  (50,23)*{\vdots};
  (50,25)*{\color{red}\bullet};
  (50,25)*{\circ};
{\ar@{-} (52.5,0)*{};(52.5,5)*{}};
{\ar@{-} (52.5,5)*{};(52.5,10)*{}};
{\ar@{-} (52.5,10)*{};(52.5,15)*{}};
{\ar@{-} (52.5,15)*{};(52.5,19.5)*{}};
  (52.5,0)*{\color{red}\bullet};
  (52.5,0)*{\circ};
  (52.5,5)*{\color{red}\bullet};
  (52.5,5)*{\circ};
  (52.5,10)*{\color{red}\bullet};
  (52.5,10)*{\circ};
  (52.5,15)*{\color{red}\bullet};
  (52.5,15)*{\circ};
  (52.5,23)*{\vdots};
  (52.5,25)*{\color{red}\bullet};
  (52.5,25)*{\circ};
  (56.25,2.5)*{\hdots};
  (56.25,7.5)*{\hdots};
  (56.25,12.5)*{\hdots};
  (56.25,17.5)*{\hdots};
  (56.25,22.5)*{\hdots};
  (15,-15)*{\color{green}\bullet};
  (15,-15)*{\circ};
  (40,-15)*{\color{red}\bullet};
  (40,-15)*{\circ};
  (15,-17.5)*{\scriptstyle \cat P(1,0,1)};
  (40,-17.5)*{\scriptstyle \cat P(C_p,0,1)};
  (-4.5,-8)*{\underbrace{\ }};
{\ar@{|->} (-4.5,-12)*{};(-4.5,-17)*{}};
%
%
{\ar@{-} (15,-40)*{};(20,-30)*{}};
{\ar@{-} (15,-40)*{};(22.5,-30)*{}};
{\ar@{-} (15,-40)*{};(25,-30)*{}};
{\ar@{-} (15,-40)*{};(27.5,-30)*{}};
{\ar@{-} (40,-40)*{};(45,-30)*{}};
{\ar@{-} (40,-40)*{};(47.5,-30)*{}};
{\ar@{-} (40,-40)*{};(50,-30)*{}};
{\ar@{-} (40,-40)*{};(52.5,-30)*{}};
{\ar@{-} (15,-40)*{};(-4.5,-30)*{}};
{\ar@{-} (40,-40)*{};(-4.5,-30)*{}};
(-4.5,-27)*{\scriptstyle \mathfrak{p}(1,p)=\mathfrak{p}(C_p,p)};
(-4.5,-30)*{{\scriptscriptstyle{\color{green}\LEFTCIRCLE}\kern-.77em{\color{red}\RIGHTCIRCLE}}};
(-4.5,-30)*{\Circle};
(25,-25)*{\scriptstyle \mathfrak{p}(1,q)\ \ldots};
(25,-27)*{\scriptscriptstyle (q\neq p)};
(20,-30)*{\color{green}\bullet};
(20,-30)*{\circ};
(22.5,-30)*{\color{green}\bullet};
(22.5,-30)*{\circ};
(25,-30)*{\color{green}\bullet};
(25,-30)*{\circ};
(27.5,-30)*{\color{green}\bullet};
(27.5,-30)*{\circ};
(31.25,-30)*{\hdots};
(50,-25)*{\scriptstyle \mathfrak{p}(C_p,q)\ \ldots};
(50,-27)*{\scriptscriptstyle (q\neq p)};
(45,-30)*{\color{red}\bullet};
(45,-30)*{\circ};
(47.5,-30)*{\color{red}\bullet};
(47.5,-30)*{\circ};
(50,-30)*{\color{red}\bullet};
(50,-30)*{\circ};
(52.5,-30)*{\color{red}\bullet};
(52.5,-30)*{\circ};
(56.25,-30)*{\hdots};
(15,-40)*{\color{green}\bullet};
(15,-40)*{\circ};
(40,-40)*{\color{red}\bullet};
(40,-40)*{\circ};
(15,-42.5)*{\scriptstyle \mathfrak{p}(1,0)};
(40,-42.5)*{\scriptstyle \mathfrak{p}(C_p,0)};
\endxy
\end{equation}
Looking only at the primes involving the subgroup~$C_p$ itself (the red dots), we see a copy of the chromatic picture in~\eqref{eq:Spc(SH)}; similarly with the trivial subgroup (the green dots). Note how the collision $\mathfrak{p}(1,p)=\mathfrak{p}(C_p,p)$ in the spectrum of the Burnside ring $A(C_p)$ disappears in the spectrum of the $C_p$-equivariant stable homotopy category~$\SH(C_p)^c$, in which the fiber over that green-red point is a connected pair of ``chromatic towers". We provide further comments in Remark~\ref{rem:more-dress}.

More generally, we can fully determine the topology of~$\Spc(\SHGc)$ for groups whose order is square-free (Theorem~\ref{thm:square-free}). For instance, the spectrum of the symmetric group $G=S_3$ is discussed in Example~\ref{ex:S_3}.

For arbitrary finite groups, we can completely describe the topology of our spectrum up to an unresolved indeterminacy in the case of $p$-groups.  This indeterminacy can be explained as follows. For a $p$-group $G$ of order $p^r$ and for $n>r$, we can prove that $\cat P(1,p,n) \subset \cat P(G,p,n-r)$ but it might be possible \emph{a priori} that this is not the best such inclusion: We could have an inclusion into a smaller prime, \ie $\cat P(1,p,n) \subset \cat P(G,p,n-r+1)$. Our current guess is that the inclusion $\cat P(1,p,n) \subset \cat P(G,p,n-r)$ is the sharpest one, and we introduce this as Conjecture~\ref{conj:log-p}, mainly to fix the ideas. If the answer turns out to be different than the one predicted by Conjecture~\ref{conj:log-p}, the complete determination of the topology of~$\Spc(\SHGc)$ will nevertheless follow, \textsl{mutatis mutandis}. Here is a summary of what we know about the topology, for $G$ an arbitrary finite group:

\begin{Thm}[Corollary~\ref{cor:summary}]
Every closed subset of $\Spc(\SH(G)^c)$ is a finite union of irreducible closed subsets, and every irreducible closed set is of the form $\adhpt{\cat P} = \SET{\cat Q \in \Spc(\SHGc)}{\cat Q \subseteq \cat P}$ for a unique prime $\cat P \in \Spc(\SHGc)$. The only inclusions $\cat Q \subseteq \cat P$ between equivariant primes happen when $\cat Q = \cat P(K,p,n)$ and $\cat P=\cat P(H,p,m)$ with $K$ conjugate in $G$ to a $p$-subnormal subgroup of $H$ (see~\ref{ter:subnormal}) and when $n$ is greater or equal to a certain chromatic integer $n_{\min}$, which depends on $p$, $H$, $K$ and $m$. We have that $n_{\min} \le m + \log_p(|H|/|K|)$. Conjecture~\ref{conj:log-p} is equivalent to $n_{\min} = m + \log_p(|H|/|K|)$. This holds if $G$ has square-free order.
\end{Thm}

In Section~\ref{se:log-p-conjecture} we translate our conjecture into a new chromatic blue-shift phenomenon for the Tate construction, which would enhance and clarify known results in the subject. Existing blue-shift results in the literature usually consider the functor $t_G(\triv(-))^G : \SH \to \SH$ from non-equivariant spectra to itself and roughly state that this functor lowers chromatic degree (at~$p$) by one \emph{for all finite groups~$G$} whose order is divisible by~$p$. (See Theorem~\ref{thm:kuhovsky} for a precise statement.) In particular, the chromatic degree only goes down by one even if $p^2$ divides the order of~$G$, or $p^3$, etc. In contrast, we conjecture that the functor $\phigeomb{G}(t_G(\triv(-))) : \SH \to \SH$, obtained by replacing categorical fixed points by geometric fixed points, reduces chromatic degree by $\log_p(|G|)$ for any $p$-group~$G$, and more generally for Tate cohomology with respect to suitable families of subgroups. The precise statements can be found in Section~\ref{se:log-p-conjecture}. Specifically, we prove in Theorem~\ref{thm:equiv-Tate} that this conjectural property of Tate cohomology is equivalent to Conjecture~\ref{conj:log-p} and would therefore complete the determination of the topology of~$\Spc(\SHGc)$ for all finite groups. Again, if future research isolates a different behavior of Tate cohomology than the shift by~$\log_p(|G|)$, the methods we provide in Section~\ref{se:log-p-conjecture} will still allow the determination of the topology of~$\Spc(\SHGc)$, which is the ultimate goal.

Finally, Section~\ref{se:classif} contains the translation of the computation of~$\Spc(\SHGc)$ into the classification of tt-ideals. See Corollary~\ref{cor:classif}.

Let us conclude our introduction with a word about the existing literature. Beyond of course the non-equivariant case~\cite{HopkinsSmith98} already mentioned, the problem of classifying tt-ideals in~$\SHGc$ is seriously considered only in Strickland's unpublished notes~\cite{Strickland12up}, parts of which have been written up in Joachimi~\cite[Chap.\,3]{Joachimi15pp}.  Although we do not rely on any results from this source, we are extremely grateful to Neil Strickland for sharing this material with us while we were working on the project.  The inspiration for Proposition~\ref{prop:C_p-inclusion} clearly came from his work.  It is our understanding that~\cite{Strickland12up} does not contain a complete classification of tt-ideals, even for the cyclic group~$G=C_p$; in contrast, we do provide a complete answer in this case, among others. Nevertheless, \cite{Strickland12up} remains a highly valuable source, discussing several other topics. For instance, Strickland proves the equivariant version of the Nilpotence Theorem, which is also our already mentioned Theorem~\ref{thm:nilpotence-thm}.

\subsection*{Acknowledgements\,:}
We are very grateful to Neil Strickland, for the reasons explained above. We also thank John Greenlees and Mike Hill for several stimulating discussions.

\smallskip
\section{Background on equivariant stable homotopy theory}
\label{se:background}%
\medskip

Everywhere, ``tt" is short for ``tensor-triangulated" or ``tensor-triangular".

\smallbreak

In this preliminary section we isolate some basic features of equivariant stable homotopy theory, seen from the angle of \mbox{tt-geometry}. All results we need concern the stable homotopy category~$\SHG$, not any given model, and we will remain agnostic about the choice of such model-theoretic foundations.  All of the following facts can be checked using, \eg, the Lewis-May approach to $G$-spectra~\cite{LewisMaySteinbergerMcClure86}, but more recent foundations such as equivariant orthogonal spectra~\cite{MandellMay02} would serve equally well.  Our goal is not to be exhaustive, nor brilliantly pedagogical, but simply to exhibit key features of the tt-category $\SH(G)$ which enable our proofs to work.  One motivation is that our techniques may have some success in computing the spectrum of other equivariant \mbox{tt-categories} which exhibit similar features (\eg, have a nice enough set of generators and well-behaved analogues of the geometric fixed point functors).  In any case, the reader will find a modern detailed discussion of the fundamental features of equivariant stable homotopy theory in the Appendices of \cite{HillHopkinsRavenel16}, while the triangulated category aspects are discussed in \cite{HoveyPalmieriStrickland97}. Before we begin, let us get the following notation and terminology out of the way:
\begin{Not}
\label{not:le_G}%
We write $H \sim_G H^g=\SET{g\inv h\,g}{h\in H}$ to indicate conjugation.
We say that $H$ is \emph{subconjugate} to~$K$ if $H^g \leq K$ for some $g \in G$ and write $H \leq_G K$.
\end{Not}
\begin{Ter}
	\label{ter:tt-cat}
We use standard terminology from the theory of tensor triangulated categories (such as thick subcategory, localizing subcategory, $\otimes$-ideal, compact, rigid) all of which may be found in \cite[\S\S 1-2]{BalmerFavi11}.  Note that the topology literature sometimes uses ``small'' or ``finite'' as a synonym for ``compact'', and ``strongly dualizable'' for ``rigid''.  The compatibility axioms for our tt-categories are those of \cite[Appendix~A]{HoveyPalmieriStrickland97}.  Finally, by a tt-functor we mean a triangulated functor which is also a strong tensor functor.
\end{Ter}

\subsection*{The equivariant stable homotopy category}
\begin{enumerate}[label=(\Alph*)]
\item
 \label{it:SH(G)}%
For any finite group $G$, the equivariant stable homotopy category of genuine $G$-spectra~$\SH(G)$ is a rigidly-compactly generated tt-category; \ie it is a compactly generated triangulated category with a compatible closed symmetric monoidal structure having the property that the compact objects coincide with the rigid (\aka strongly dualizable) objects.  Every pointed $G$-space $X$ defines a suspension $G$-spectrum $\Sigma^\infty_G X$ in~$\SHG$ and the collection of~$\Sigma^\infty_G G/H_+$ for all subgroups $H \le G$ provides a set of compact-rigid generators (see \eg\;\cite[\S 9.4]{HoveyPalmieriStrickland97}). We denote the tt-category of compact-rigid objects by $\SHGc \subset \SH(G)$.
\smallbreak
\item
\label{it:alpha*}%
Any group homomorphism $\alpha : H \to G$ induces a tt-functor $\alpha^*:\SH(G) \to \SH(H)$ which
necessarily preserves compact objects (since compact=rigid in these categories and any strong tensor functor preserves rigid objects).  For example, inclusion of a subgroup $H \hook G$ provides the restriction functor $\Res^G_H : \SH(G) \to \SH(H)$, while a quotient $G \twoheadrightarrow G/N$ provides the inflation functor $\Infl_{G/N}^G : \SH(G/N) \to \SH(G)$.
\smallbreak
\item
\label{it:cat-fixed-pts}%
Restriction admits adjoints on both sides (induction and coinduction) which are isomorphic: $\Ind_H^G \simeq \CoInd_H^G$.  On the other hand, the right adjoint to inflation is the ``categorical'' fixed-point functor $(-)^N : \SH(G) \to \SH(G/N)$. This does not recover usual fixed-points on suspension spectra, nor does it preserve compact objects (cf.~the tom Dieck splitting theorem) and therefore inflation does not have a left adjoint (by~\cite{BalmerDellAmbrogioSanders16pp}).  Although $(-)^N \circ \Infl_{G/N}^G \not\simeq \Id_{\SH(G/N)}$, we do have a projection formula: \[(\Infl_{G/N}^G(x) \otimes y)^N \simeq x \otimes y^N\] for every $x\in \SH(G/N)$ and $y \in \SH(G)$.
\end{enumerate}
\subsection*{Restriction as a separable-extension}
\begin{enumerate}[resume*]
\item
\label{it:A-Mod}%
For every subgroup $H \leq G$, the $G$-spectrum $A^G_H := G/H_+$ (we often drop the $\Sigma^\infty_G$ for readability) is a separable commutative ring object (=a tt-ring, or a finite \emph{\'etale} ring) in the tt-category $\SHG$.  (See \cite{Balmer11} and \cite{Balmer13ppb} for the definition of separable and for a discussion of how such rings provide tt-geometry with an analogue of the {\'etale} topology.)
By~\cite{BalmerDellAmbrogioSanders15}, there is a tt-equivalence
\[
A^G_H\MMod_{\SHG}\cong \SH(H)
\]
such that extension-of-scalars $\SH(G)\too A^G_H\MMod_{\SHG}$ along~$A^G_H$ is isomorphic to the restriction functor $\Res^G_H:\SHG\too \SH(H)$. Since $A^G_H$ is compact, the same relation holds for the compact objects\,: $\SH(H)^c\cong A^G_H\MMod_{\SHGc}$.

\smallbreak
\item
\label{it:res-supp}%
By~\cite{Balmer13ppb}, for any tt-category $\cat K$ and any tt-ring $A$ in~$\cat K$, the map induced on spectra $\Spc(A\MMod_{\cat K})\to \Spc(\cat K)$ has image exactly $\supp(A)$. Hence, in view of~\ref{it:A-Mod}, we have the following equality of subsets of~$\Spc(\SHGc)$:
\[
\Img\big(\Spc(\Res^G_H)\big) = \supp(G/H_+).
\]
\smallbreak
\item
\label{it:going-up}
Also by~\cite{Balmer13ppb}, if a tt-ring $A$ has finite degree (in the sense of~\cite{Balmer14}) then $\varphi_A : \Spc(A\MMod_{\cat K}) \to \Spc(\cat K)$ satisfies the ``Going-Up Theorem'': For every $\cat Q \in \Spc(A\MMod_{\cat K})$ and every $\cat P' \in \overline{\{\varphi_A(\cat Q)\}}$ there exists $\cat Q' \in \overline{\{\cat Q\}}$ such that $\varphi_A(\cat Q') = \cat P'$.  It follows from~\cite[Corollary~4.8]{Balmer14} and the fact that the geometric fixed point functors (see below) are jointly conservative that every tt-ring in $\SHGc$ has finite degree. So, in particular, the \mbox{Going-Up} Theorem holds for $A=G/H_+$, that is, for~$\Res^G_H$.
\end{enumerate}
\subsection*{Geometric fixed points}
\begin{enumerate}[label=(\Alph*),resume*]
\item
\label{it:geom-motiv}%
To quickly motivate the construction below -- in light of the incompatibility between categorical fixed-points and suspension spectra -- note that for subgroups $H,K\le G$, the $H$-fixed points of the $G$-set $G/K$ is the set $N_G(H,K)/K$ where $N_G(H,K):=\SET{g\in G}{H^g\subseteq K}$. For instance, $(G/K)^H$ is empty if $N_G(H,K)=\varnothing$, that is, if $H\not\le_G K$. In particular, when $H=N\normal G$ is normal, the $N$-fixed points of the $G$-set $G/K$ is either empty when $N\not\leq K$, or $G/K$ when $N\leq K$. So, adding base-points, we expect $N$-fixed points to ``kill" $G/K_+$ (\ie map it to $\varnothing_+=0$) when $N\not\le K$. This ``killing" is nothing but a localization, as we formalize below.
\smallbreak
\item
\label{it:phi^N}%
Let $N \lenormal G$ be a normal subgroup and let
\[
\cat L_N:=\Loc_{\otimes}(G/K_+\Mid K\leq G,\ N \not\leq K )
\]
denote the localizing $\otimes$-ideal of $\SH(G)$ generated by the compact objects
$\SET{G/K_+}{K \le G, N \not\leq K}$; see~\ref{it:geom-motiv}.
Then inflation $\SH(G/N) \to \SHG$ composed with the (Bousfield) localization of $\SHG$ with respect to $\cat L_N$
$$
\xymatrix@C=4em{\SH(G/N) \ar[r]^-{\Infl_{G/N}^G} \ar@/_1em/[rr]_-{\cong} & \SHG \ar@{->>}[r] & \SHG / \cat L_N}
$$
is an equivalence of tt-categories (by \cite[Cor.\,II.9.6]{LewisMaySteinbergerMcClure86}).
Then, the composite
\begin{equation}
	\label{eq:phi^N-def}
\SH(G) \onto \SH(G)/\cat L_N \xra{\cong} \SH(G/N)
\end{equation}
is the so-called \emph{geometric $N$-fixed point functor} $\phigeom{N} : \SH(G) \onto \SH(G/N)$. By construction, inflation splits the geometric fixed point functor:
\begin{equation}
\label{eq:phinfl}%
\phigeom{N}\circ\Infl_{G/N}^G\cong\Id_{\SH(G/N)}.
\end{equation}
Since $\cat L_N$ is compactly generated, \eqref{eq:phi^N-def} presents $\SH(G/N)$ as a finite localization of~$\SH(G)$, with $\phigeom{N}$ as the localization functor. It follows from Neeman~\cite[Thm.\,2.1]{Neeman92b} that the same is also true on the subcategories of compact objects, up to an idempotent completion (denoted~${}^\natural$):
\[ \SHGc\onto \SHGc/\cat J_N\hook (\SHGc/\cat J_N)^\natural\xra{\cong} \SH(G/N)^c \]
where $\cat J_N := \thick_\otimes(G/K_+ | N \not\le K)$.  (We will discuss such finite localizations in more detail in Section~\ref{se:e,f,Tate}; see especially Remark~\ref{rem:normal-family}.) Note that in this special case, the idempotent completion~$(-)^\natural$ is not necessary, since~$\phigeom{N}$ is split by~$\Infl_{G/N}^G$, which preserves compact objects.
\smallbreak
\item
\label{it:phixed-pts}%
For an arbitrary subgroup $H \le G$, we define the \emph{(absolute) geometric \mbox{$H$-fixed} point functor}
$\phigeomb{H} : \SH(G) \to \SH$
as the composite
\[ \phigeomb{H}:\ \SH(G) \xra{\Res^G_H} \SH(H) \xra{\phigeom{H}} \SH.\]
Note that this functor is the composition of an ``\'{e}tale'' extension in the sense of~\ref{it:A-Mod} with a finite localization in the sense of~\ref{it:phi^N}.  When necessary to distinguish the ambient group, we shall write $\phigeomb{H,G} : \SH(G) \to \SH$. Since the functor $\phigeomb{H,G} : \SH(G) \to \SH$ is a tt-functor which preserves compact objects, it induces a continuous map on tt-spectra, that we denote
\[ \varphi^{H,G} : \Spc(\SHc) \to \Spc(\SHGc).\]
It is defined by $\cat P\mapsto (\phigeomb{H,G})\inv(\cat P)=\SET{x\in \SHGc}{\phigeomb{H,G}(x)\in\cat P}$.
\smallbreak
\item
\label{it:phixed-nested}%
For any $N \le N' \normal G$, with $N \lenormal G$, we have $\phigeom{N'/N} \circ \phigeom{N} \cong \phigeom{N'}$ from $\SHG$ to $\SH(G/N')$, as localization can be performed in two ``nested" steps.  In fact, for any $N \le H \le G$, with $N \normal G$, we also have $\phigeomb{H/N} \circ \phigeom{N} \cong \phigeomb{H}$ from $\SHG$ to $\SH$.  For example, $\phigeomb{G/N} \circ \phigeom{N} \cong \phigeomb{G}$ and $\Res^{G/N}_1 {\mkern-9mu} \circ \phigeom{N} \cong \phigeomb{N}$ for any $N \normal G$.
\smallbreak
\item
\label{it:Phixed-alpha}%
Let $\alpha:G\to G'$ be a group homomorphism. Consider $H\leq G$ and its image \mbox{$\alpha(H)\le G'$}. Then we have an isomorphism $\phigeomb{H,G}\circ\alpha^*=\phigeomb{\alpha(H),G'}$. Indeed, the following diagram commutes up to isomorphism, for $N:=\ker(\alpha)\cap H$:
\[
\xymatrix @R=.5em{
\SH(G') \ar[r]^-{\Res} \ar[dd]_-{\alpha^*} \ar@/^3em/[rrrrd]^(.7){\ \phigeomb{\alpha(H),G'}}
& \SH(\alpha(H)) \ar@/^/[rd]^-{\id} \ar[dd]_-{\alpha_{|H}^*=\Infl_{\alpha(H)}^H}
&
\\
&& \SH(\alpha(H)) \ar@{->>}[rr]^-{\phigeom{\alpha(H)}}
&& \SH\,.
\\
\SHG \ar[r]_{\Res} \ar@/_3em/[rrrru]_(.7){\phigeomb{H,G}}
& \SH(H) \ar@{->>}@/_/[ru]^-{\phigeom{N}} \ar@{->>}@/_.5em/[rrru]_{\phigeom{H}}
}\]
It uses~\eqref{eq:phinfl} for the middle triangle and~\ref{it:phixed-nested} on the right, plus the definitions. Hence the induced maps from $\Spc(\SHc)$ to~$\Spc(\SH(G')^c)$ coincide:
\[
\Spc(\alpha^*)\circ \varphi^{H,G}=\varphi^{\alpha(H),G'}.
\]
\smallbreak
\item
\label{it:conj}%

For any $g \in G$, the isomorphism $(-)^g : H \xra{\sim} H^g$ induces an equivalence $\SH(H^g) \xra{\cong} \SH(H)$ such that the diagram
\[\xymatrix @R=.5em{
				& \SH(H^g) \ar@/^/[dr]^{\phigeom{H^g}} \ar[dd]^{\cong} &\\
				\SH(G) \ar@/_/[dr]_{\Res^G_{H}} \ar@/^/[ur]^{\Res^G_{H^g}} && \SH \\
				& \SH(H) \ar@/_/[ur]_{\phigeom{H}}
}\]
commutes up to isomorphism. (The equivalence $\SHG \xra{\cong} \SHG$ induced by the inner automorphism $(-)^g:G \xra{\sim} G$ is naturally isomorphic to the identity functor.) Hence $\phigeomb{H,G} \cong \phigeomb{H^g,G}$ and $\varphi^{H,G} = \varphi^{H^g,G}$.
\smallbreak
\item
\label{it:triv}%
The trivial homomorphism $G \to 1$ induces a tt-functor $\Infl_1^G$ abbreviated
\[\triv:\SH \to \SHG
\]
which sends a spectrum to the ``trivial'' $G$-spectrum. It splits the (absolute) $H$-fixed point functor $\phigeomb{H} : \SH(G) \to \SH$ for every $H \leq G$. Hence $\Spc(\triv)$ retracts $\varphi^{H,G}:\Spc(\SHc)\to \Spc(\SHGc)$ for every~$H\le G$, which shows that the maps~$\varphi^{H,G}$ are all injective.
\smallbreak
\item
\label{it:X^H}%
The geometric fixed point functors have the nice property that they commute with suspension spectra (see \eg\;\cite[Cor.\,II.9.9]{LewisMaySteinbergerMcClure86}).  In particular, for a finite $G$-set $X$ and for any subgroup $H \le G$, we have
$\phigeomb{H}(\Sigma^\infty_G X_+) \cong \Sigma^\infty X_+^H$ in~$\SHc$.
\end{enumerate}

\smallskip
\section{The spectrum of the Burnside ring}
\label{se:burnside}%
\medskip

The purpose of this section is to briefly recall Dress's description of the spectrum of the Burnside ring~\cite{Dress69} and related ideas concerning $p$-perfect subgroups.
\begin{Def}
\label{def:p-perfect}
A group is said to be \emph{$p$-perfect} if it admits no non-trivial homomorphism to a $p$-group, or equivalently if the only normal subgroup of index a power of~$p$ is the group itself. Every finite group $G$ has a unique normal $p$-perfect subgroup $\Op (G)$ such that the quotient $G/\Op(G)$ is a $p$-group; $\Op(G)$ is the intersection of all normal subgroups of index a power of~$p$. So, $\Op (G)$ contains all $p$-perfect subgroups $H$ of~$G$ (since the composite $H\to G\to G/\Op(G)$ must be trivial) and $\Op(G)$ is contained in every normal subgroup of~$G$ of index a power of~$p$.
\end{Def}

\begin{Ter}
\label{ter:subnormal}%
A \emph{$p$-subnormal tower} $H = H_r \lnormal_p H_{r-1} \lnormal_p$ $\cdots$ $\lnormal_p H_1 \lnormal_p H_0 = G$ is a subnormal tower $H_{i+1}\lnormal H_{i}$ all of whose subquotients $H_i/H_{i+1}$ have order~$p$.
\end{Ter}
\begin{Lem}
\label{lem:p-subnormal}
Let $H \le G$ be a subgroup. The following are equivalent:
\begin{enumerate}[label=\rm(\arabic*)]
\item $\Op (H) = \Op (G)$;
\item $H \supseteq \Op (G)$;
\item There exists a $p$-subnormal tower from $H$ to~$G$.
\end{enumerate}
In this case, we say that $H$ is a \emph{$p$-subnormal subgroup} of~$G$.
\end{Lem}
\begin{proof}
Since $\Op(G)$ contains all $p$-perfect subgroups of~$G$, we have $\Op (H) \subseteq \Op (G)$. The equivalence $(1) \Leftrightarrow (2)$ follows.  To prove $(1) \Rightarrow (3)$, assume~$(1)$, let $N:=\Op (H)=\Op (G)$ and consider $H/N \le G/N$ in the $p$-group $G/N$.  Every subgroup of a $p$-group admits a $p$-subnormal tower to the ambient group (by an easy inductive argument using the non-trivial center, for instance) and a $p$-subnormal tower from $H/N$ to $G/N$ lifts to a $p$-subnormal tower from $H$ to~$G$.  Finally, $(3) \Rightarrow (1)$ follows inductively from the observation that if $N \lnormal_p G$ then $N \supseteq \Op (G)$.
\end{proof}

\begin{Rem}
Recall that the \emph{Burnside ring}, $A(G)$, is the Grothendieck ring of the category of finite $G$-sets with respect to disjoint union, and equipped with Cartesian product as multiplication. For each subgroup $H \le G$, there is a ring homomorphism ${f^{H}}:A(G) \to \mathbb Z$ which sends (the class of) a finite $G$-set $X$ to the number of $H$-fixed points~$|X^H|$.  By pulling back the prime ideals of $\mathbb Z$, we get prime ideals
\begin{equation}
\label{eq:p(H)}%
\mathfrak{p}(H,p) := (f^{H})\inv(p\mathbb Z)\qquadtext{and}\mathfrak{p}(H,0) := (f^{H})\inv((0))
\end{equation}
of~$A(G)$. Dress's result~\cite{Dress69} can be summarized as follows:
\end{Rem}
\begin{Thm}[Dress]
\label{thm:dress}%
Every prime ideal of $A(G)$ is of the form $\mathfrak{p}(H,p)$ or $\mathfrak{p}(H,0)$ for some $H \le G$. Moreover, for any subgroups $H,K \le G$ and any primes $p,q$:
\begin{enumerate}
\item $\mathfrak{p}(H,0) \subseteq \mathfrak{p}(K,0)$ iff $\mathfrak{p}(H,0) = \mathfrak{p}(K,0)$ iff $H \sim_G K$.
\item $\mathfrak{p}(H,p) \subseteq \mathfrak{p}(K,q)$ implies $p=q$ and $\mathfrak{p}(H,p) = \mathfrak{p}(K,p)$.
\item $\mathfrak{p}(H,0) \subseteq \mathfrak{p}(K,p)$ iff $\mathfrak{p}(H,p) = \mathfrak{p}(K,p)$ iff $\Op (H) \sim_G \Op (K)$.
\item $\mathfrak{p}(H,0) \subset \mathfrak{p}(H,p)$ and $\mathfrak{p}(H,p) \not\subseteq \mathfrak{p}(K,0)$.
\end{enumerate}
\end{Thm}

\begin{Rem}
In other words, the ring $A(G)$ has Krull dimension 1 and $\Spec(A(G))$ is covered by copies of~$\Spec(\bbZ)$ indexed by conjugacy classes of subgroups~$H\le G$. The copy of $\Spec(\bbZ)$ in~$\Spec(A(G))$ corresponding to $H\le G$ is an irreducible closed subset with generic point $\mathfrak{p}(H,0)$. In the closure of~$\mathfrak{p}(H,0)$ are height 1 closed points $\mathfrak{p}(H,p)$ for all prime numbers~$p$. The only overlap between two such copies, say for two subgroups $H$ and $K$, happens at closed points, corresponding to the same prime~$p$ in each copy, exactly when $\Op(H)$ and $\Op(K)$ are conjugate in~$G$. For given~$H$ and~$K$, this can happen at several primes~$p$.  However, note that if $p$ does not divide the order of~$G$ then $\Op (H) = H$ for all $H \le G$. Hence there can only be non-trivial collisions for $p$ dividing the order of~$G$.  See the bottom of~\eqref{eq:Spc(C_p)} for a picture of the $G=C_p$ case (which is representative of the situation for all $p$-groups) and see Example~\ref{ex:S_3} for a picture of the $G=S_3$ case.
\end{Rem}

\begin{Rem}
The ring homomorphism $A(G) \cong \End_{\SHGc}(\unit) \to \End_{\SHc}(\unit) \cong \mathbb Z$ induced by the geometric fixed point functor $\phigeomb{H} : \SHGc \to \SHc$ is exactly the map ${f^{H}} : A(G) \to \mathbb Z$ used above (see the proof of Prop.\,\ref{prop:comp-map}, if necessary). From this point of view, it is very natural to attempt a categorified version of what Dress accomplished for the Burnside ring, as announced in the Introduction. Furthermore, this compatibility between $\phigeomb{H}$ and ${f^{H}}$ will allow (cf.~Proposition~\ref{prop:comp-map}) for an easy description of the comparison map $\rho_{\SHGc} : \Spc(\SHGc) \to \Spec(A(G))$ once we know what the primes in $\Spc(\SHGc)$ actually are.
\end{Rem}

\smallskip
\section{The set $\Spc(\SHGc)$}
\label{se:the-set}
\medskip

Recall the Hopkins-Smith non-equivariant primes~$\cat C_{p,n}\in\Spc(\SHc)$ from~\eqref{eq:Spc(SH)}. We now consider their image under the map~$\varphi^{H,G} : \Spc(\SHc) \to \Spc(\SHGc)$ induced by geometric $H$-fixed points $\phigeomb{H,G}:\SHG\to \SH$, as explained in~\ref{it:phixed-pts}.

\begin{Def}
\label{def:P(H,p,n)}%
For every subgroup~$H\leq G$, every prime number~$p$ and every ``chromatic integer" $1\le n\le \infty$, we have a tt-prime in~$\SHGc$ given by
\[
\cat P(H,p,n):=\varphi^{H,G}(\cat C_{p,n})=\SET{x\in\SHGc}{\phigeomb{H}(x)\in\cat C_{p,n}\textrm{ in }\SHc}.
\]
By~\ref{it:conj}, we have that $\cat P(H,p,n)=\cat P(H^g,p,n)$ for every $g\in G$.  So $\cat P(H,p,n)$ only depends on the conjugacy class of~$H$ in~$G$.  Also, since $\cat C_{p,1} = \cat C_{q,1}$ for all primes $p,q$, the same is true for $\cat P(H,p,1) = \cat P(H,q,1)$ and we could replace $p$ by~$0$ in the notation, if it avoids putting the wrong emphasis on a particular~$p$:
\[
\cat P(H,0,1) := \cat P(H,p,1) \ {\scriptstyle(\textrm{for any~}p)} =\SET{x\in\SHGc}{\phigeomb{H}(x)\textrm{ is torsion}}.
\]
Finally, when necessary, we will write $\cat P_G(H,p,n)$ to indicate the ambient group~$G$.
\end{Def}

\begin{Rem}
\label{rem:P-infty}%
Since $\cat C_{p,\infty}=\cap_{n\geq 1}\cat C_{p,n}$ we get $\cat P(H,p,\infty)=\cap_{n\geq 1}\cat P(H,p,n)$.
\end{Rem}

Let us say a word about the functoriality of these primes~$\cat P_G(H,p,n)$ in the group~$G$. We recall that for every tt-functor $F:\cat K\to \cat L$ the induced map $\Spc(F):\Spc(\cat L)\to \SpcK$ sends $\cat Q$ to~$F\inv(\cat Q)$ and is inclusion-preserving.

\begin{Prop}
\label{prop:induced-alpha}%
Let $\alpha : G \to G'$ be a group homomorphism, and $\alpha^* : \SH(G') \to \SH(G)$ the induced tt-functor.
Then the induced map
$\Spc(\alpha^*) : \Spc(\SHGc) \to \Spc(\SH(G')^c)$ maps $\cat P_G(H,p,n)$ to $\cat P_{G'}(\alpha(H),p,n)$ for every~\mbox{$H \le G$}.
\end{Prop}

\begin{proof}
This follows directly from~\ref{it:Phixed-alpha} and Definition~\ref{def:P(H,p,n)}.
\end{proof}

\begin{Cor}
\label{cor:induced-res}%
Consider $H\le G$ and the tt-functor $\Res^G_H : \SH(G) \to \SH(H)$.
The induced map $\Spc(\Res^G_H) : \Spc(\SH(H)^c) \to \Spc(\SHGc)$ maps
$\cat P_{H}(K,p,n)$ to $\cat P_G(K,p,n)$ for any $K \le H$.
\qed
\end{Cor}

\begin{Cor}
\label{cor:induced-infl}%
Consider $N\normal G$ and the tt-functor $\Infl^G_{G/N} : \SH(G/N) \to \SHG$.
The induced map $\Spc(\Infl^G_{G/N}) : \Spc(\SHGc) \onto \Spc(\SH(G/N)^c)$ maps
$\cat P_{G}(H,p,n)$ to $\cat P_{G/N}(HN/N,p,n)$ for any $H\le G$.
\qed
\end{Cor}

The special case $N=G$ shows that we can, in some sense, drop the group-theoretic information
by projecting onto the chromatic information; see also~\ref{it:triv}:

\begin{Cor}
\label{cor:project-H}%
The tt-functor $\triv=\Infl_1^G:\SH\to \SHG$ induces an inclusion-preserving map $\Spec(\triv):\Spc(\SHGc)\to \Spc(\SHc)$
which sends the equivariant prime $\cat P(H,p,n)$
to the non-equivariant prime $\cat C_{p,n}$.
\qed
\end{Cor}

\begin{Prop}
\label{prop:induced-geom}%
Consider $N \lenormal G$ and the tt-functor $\phigeom{N} : \SH(G) \onto \SH(G/N)$.  The induced map $\Spc(\phigeom{N}) : \Spc(\SH(G/N)^c) \hook \Spc(\SHGc)$ is injective and maps $\cat P_{G/N}(H/N,p,n)$ to $\cat P_G(H,p,n)$ for any $N \le H \le G$. The image of~$\Spc(\phigeom{N})$ is exactly~$\SET{\cat P\in\Spc(\SHGc)}{G/H_+\in\cat P\mathrm{\ for\ all\ }H\mathrm{\ s.t.\ }N\not\le H}$.
\end{Prop}

\begin{proof}
In view of~\ref{it:phi^N}, $\phigeom{N}:\SHGc\onto\SH(G/N)^c$ is the localization with respect to $\cat J_N := \thick_\otimes(G/K_+ | N \not\le K)$. It is a general fact that $\Spc(\cat K/\cat J)\cong\SET{\cat P\in\SpcK}{\cat J\subseteq\cat P}$, hence the injectivity of the induced map on spectra and the description of its image. The description of its value on $\cat P_{G/N}(H/N,p,n)$ comes from the functoriality of~$\Spc(-)$ and the relation $\phigeomb{H/N} \circ \phigeom{N} \cong \phigeomb{H}$
in~\ref{it:phixed-nested}.
\end{proof}

\begin{Rem}
Contrary to what happens in Proposition~\ref{prop:induced-geom}, the apparently harmless map of Corollary~\ref{cor:induced-res} is not injective in general. There is some fusion phenomenon happening, since $\cat P_H(K,p,n)$ depends on the conjugacy class of~$K$ in~$H$, whereas $\cat P_G(K,p,n)$ only depends on the conjugacy class of~$K$ in the larger group~$G$.
\end{Rem}

\begin{Thm}
\label{thm:the-set}
Every prime $\cat P\in \Spc(\SHGc)$ is of the form $\cat P(H,p,n)$ for some $H\leq G$, some prime~$p$ and some $1\leq n\leq \infty$.
\end{Thm}

\begin{proof}
We need to show that $\Spc(\SHGc) = \bigcup_{H \le G} \Img(\varphi^{H,G})$.
We proceed by induction on $|G|$. The result is clear for $G$ trivial. So, let us suppose the result known for all proper subgroups of~$G$.
By Proposition~\ref{prop:induced-geom} for $N=G$, we can then identify the image of $\varphi^{G,G} = \Spc(\phigeomb{G})$ as follows:
\[
\Img(\varphi^{G,G})=\SET{\cat P\in\Spc(\SHGc)}{\cat P\ni G/H_+\textrm{ for all }H\lneq G}.
\]
Since $\SET{\cat P}{\cat P\not\ni G/H_+}=\supp(G/H_+)$ by definition of the support, we have
\[
\Spc(\SHGc)=\Img(\varphi^{G,G})\cup \bigcup_{H\lneq G}\supp(G/H_+)\,.
\]
By~\ref{it:res-supp}, we have $\supp(G/H_+)=\Img(\Spc(\Res^G_H))$. By the induction hypothesis applied to $H\lneq G$, we know that $\Spc(\SH(H)^c)$ is covered by the images of $\varphi^{K,H}$ for $K\leq H$ and we need to know what happens to those images $\Img(\varphi^{K,H})$ under $\Spc(\Res^G_H)$.
By~\ref{it:Phixed-alpha}, we know that
$\Spc(\Res^G_H)$ will map $\Img(\varphi^{K,H})$ into $\Img(\varphi^{K,G})$ and this completes the proof.
\end{proof}

We now want to describe the support of some basic objects in~$\SHGc$.

\begin{Lem}
\label{lem:G/K^H}%
For two subgroups $H,K\leq G$, we have in~$\SH$ that
\[
\phigeomb{H}(G/K_+)=\left\{
\begin{array}{cl}
\unit^{\oplus \ell} & \textrm{for some $\ell>0$ if $H\leq_G K$ (Notation~\ref{not:le_G})},
\\
0 &\textrm{if $H\not\leq_G K$}.
\end{array}\right.
\]
\end{Lem}

\begin{proof}
By~\ref{it:X^H}, it suffices to compute the $H$-fixed subset of the finite $G$-set $G/K$. The result holds with $\ell=|N_G(H,K)/K|$ by the discussion in~\ref{it:geom-motiv}.
\end{proof}

\begin{Prop}
\label{prop:G/K^H}%
Let $H,K\leq G$ be two subgroups, $p$ a prime and $1\le n\le\infty$. Then $G/K_+\in \cat P(H,p,n)$ if and only if $H\not\leq_G K$.
\end{Prop}

\begin{proof}
This is immediate from the definition of $\cat P(H,p,n)$ as~$(\phigeomb{H})\inv(\cat C_{p,n})$ and Lemma~\ref{lem:G/K^H} since $\unit\notin\cat P$ and $0\in \cat P$ for every prime $\cat P=\cat C_{p,n}$ in~$\Spc(\SHc)$.
\end{proof}

\begin{Cor}
\label{cor:G/K^H}%
If $\cat P(K,q,n)\subseteq\cat P(H,p,m)$ then $K\le_G H$.
\end{Cor}

\begin{proof}
Use (the contrapositive of) Proposition~\ref{prop:G/K^H} twice: As $H\le_G H$, we have $G/H_+\notin\cat P(H,p,m)$ and therefore $G/H_+\notin\cat P(K,q,n)$, which implies $K\le_G H$.
\end{proof}

\begin{Cor}
\label{cor:supp(G/K)}%
Let $K\le G$ be a subgroup. We have
\[
\supp(G/K_+)=\SET{\cat P(H,p,n)}{H\leq_G K}=\Img(\Spc(\Res^G_K)).
\]
\end{Cor}

\begin{proof}
By definition, $\supp(G/K_+)=\SET{\cat P}{G/K_+\notin\cat P}$ and by Theorem~\ref{thm:the-set}, every prime~$\cat P$ is of the form~$\cat P(H,p,n)$. So, the first equality results from Proposition~\ref{prop:G/K^H}. The other equality $\supp(G/K_+)=\Img(\Spc(\Res^G_K))$ is~\ref{it:res-supp} again.
\end{proof}

\begin{Thm}
\label{thm:prime-uniqueness}%
The primes $\cat P(H,p,n)$ are uniquely characterized by the conjugacy class of~$H$, the prime~$p$ and the chromatic integer $1\leq n\leq \infty$. More precisely, if $\cat P(H,p,m)=\cat P(K,q,n)$ then $H^g=K$ for some $g\in G$ and $\cat C_{p,m}=\cat C_{q,n}$ in~$\SHc$ (meaning $m=n$, and, as long as $m=n>1$ then also $p=q$).
\end{Thm}

\begin{proof}
Suppose that $\cat P(H,p,m)=\cat P(K,q,n)$. By Corollary~\ref{cor:G/K^H}, we have $H\leq_G K$ and $K\leq_G H$, so $H\sim_G K$. Finally Corollary~\ref{cor:project-H} gives~$\cat C_{p,m}=\cat C_{q,n}$ as desired.
\end{proof}

At this stage, we have a complete description of the \emph{set} $\Spc(\SHGc)$. We also have a complete description of the maps $\Spc(\alpha^*)$ induced by group homomorphisms~$\alpha:G\to G'$.  Before moving on to the topology of~$\Spc(\SHGc)$, we can prove the following result, which is also obtained in~\cite{Strickland12up}. It will not be used in the rest of the paper.

\begin{Thm}[Equivariant Nilpotence Theorem]
\label{thm:nilpotence-thm}
The collection of functors
\[\rmK(p,n)_*(\phigeomb{H}(-)) : \SHGc \to \rmK(p,n)_*\MMod_*\]
detects $\otimes$-nilpotence. That is, a morphism $f:x\to y$ between compact $G$-spectra $x,y \in \SHGc$ is $\otimes$-nilpotent in~$\SHGc$ (\ie $f\potimes{\ell}=0\,:\ x\potimes{\ell}\to y\potimes{\ell}$ for some $\ell\geq 1$) if and only if $\rmK(p,n)_*(\phigeomb{H}(f)) = 0$ for all $H \le G$, all primes $p$ and all $1 \le n \le \infty$.
\end{Thm}

\begin{proof}
	By the non-equivariant Nilpotence Theorem, up to replacing $f$ by some $\otimes$-power, we can assume that $\phigeomb{H}(f)=0$ for all of the finitely many subgroups~$H\le G$. By induction on the order of the group, applying the result to proper subgroups ${H\lneq G}$, together with the fact that $\phigeomb{K,G}=\phigeomb{K,H}\circ\Res^G_H$ for all ${K\le H\le G}$, we can similarly assume that $\Res^G_H(f)=0$ for all $H\lneq G$.  Consequently, ${f\otimes G/H_+=0}$ for all $H\lneq G$. Since the subcategory ``on which $f$ is $\otimes$-nilpotent", $\SET{z\in \SHGc}{f\potimes{\ell}\otimes z=0\textrm{ for some }\ell\ge 1}$, is a tt-ideal (see~\cite[Prop.\,2.12]{Balmer10b}) we deduce that for every $z\in\cat J:=\thick_\otimes{\SET{G/H_+}{H\lneq G}}$ there exists $\ell=\ell(z)\ge 1$ such that $f\potimes{\ell}\otimes z=0$. Finally, from $\phigeomb{G}(f)=0$ and the description of~$\phigeomb{G}$ as a localization (see~\ref{it:phi^N} for $N=G$), we see that $f\mapsto 0$ under $\SHGc\onto \SHGc/\cat J$.  Therefore, by a general fact on Verdier localization, $f$ factors via some object~$z\in\cat J$. Combining the two results, we see that $f\potimes{(\ell+1)}$ factors via $f\potimes{\ell}\otimes z=0$ for $\ell=\ell(z)$ as above.
\end{proof}

\smallskip
\section{Tensor idempotents and the Tate construction}
\label{se:e,f,Tate}%
\medskip

In the second part of the paper, about the topology of $\Spc(\SHGc)$, we shall need additional tools relating to generalized Rickard idempotents in the sense of~\cite{BalmerFavi11} and generalized Tate cohomology in the sense of Greenlees~\cite{Greenlees01}.

\smallbreak
Let $\cat T$ be a rigidly-compactly generated tt-category, like $\SHG$.  We shall use $[-,-]$ to denote the internal hom functor, which restricts to $\cat T^c$ by assumption, and we shall sometimes write $x^\vee:=[x,\unit]$ for the dual of an object $x \in \cat T$.

\begin{Rem}
\label{rem:Thom+class}%
Let $\cat K$ be an essentially small and rigid tt-category, like $\cat K=\cat T^c$.  Recall that a \emph{Thomason subset} of $\Spc(\cat K)$ is a set which can be written as a (possibly infinite) union of closed subsets each having a quasi-compact complement. By~\cite[Thm.\,4.10]{Balmer05a}, the tt-ideals of $\cat K$ correspond to the Thomason subsets of $\Spc(\cat K)$ via $\cat C \mapsto \supp(\cat C):=\bigcup_{x \in \cat C} \supp(x)$, with inverse given by $Y \mapsto \cat K_Y := \SET{{x \in \cat K}}{\supp(x)\subseteq Y}$. This is the \emph{classification of tt-ideals} resulting from a description of the space~$\SpcK$.
\end{Rem}

\begin{Not}
\label{not:idempotents}%
For any Thomason subset $Y\subseteq \Spc(\cat T^c)$, recall from~\cite{BalmerFavi11} that
$$
\Delta_Y:=\big(\ e_Y \to \unit \to f_Y \to \Sigma e_Y\ \big)
$$
denotes the \emph{idempotent triangle} in~$\cat T$ for the finite localization associated to the \mbox{tt-ideal} $(\cat T^c)_Y=\SET{x\in \cat T^c}{\supp(x)\subseteq Y}$.  The localizing subcategory $\cat T_Y$ of acyclic objects is by definition the one generated by $(\cat T^c)_Y$ (and we have $(\cat T_Y)^c=(\cat T^c)_Y$), and the idempotents $e_Y\potimes{2}\cong e_Y$ and $f_Y\potimes{2}\cong f_Y$ satisfy $\cat T_Y=e_Y\otimes \cat T=\Ker(f_Y\otimes-)$ and $(\cat T_Y)^\perp=f_Y\otimes \cat T=\Ker(e_Y\otimes-)=\Ker([e_Y,-])$.
\end{Not}

\begin{Rem}
\label{rem:loc}%
For $Y \subseteq \Spc(\cat T^c)$ a Thomason subset and $V:= \Spc(\cat T^c)\sminus Y$ its complement, we have the following diagram of adjunctions (\aka recollement)
$$
\xymatrix@C=.5em@R=4em{
{}^\perp\cat T(V)=e_Y\otimes \cat T=\cat T_Y \vphantom{\int_{\int_\int}}
 \ar@<-1em>@{ >->}[rd]_-{\incl} \ar@<1em>@{ >->}[rd]^-{[e_Y,-]}
 \ar@/^.5em/[rr]^-{[e_Y,-]}_-{\cong}
&&
[e_Y,\cat T]=\cat T(V)^\perp=(\cat T_Y)^{\perp\perp} \vphantom{\int_{\int_\int}}
 \ar@<-1em>@{ >->}[ld]_-{e_Y\otimes-} \ar@<1em>@{ >->}[ld]^-{\incl}
 \ar@/^.5em/[ll]^-{e_Y\otimes-}
\\
& \cat T \vphantom{\int^{\int^{\int^\int}}}
 \ar@{->>}[lu]|-{e_Y\otimes-}
 \ar@{->>}[ru]|-{[e_Y,-]}
 \ar@<-1.5em>@{->>}[d]_-{f_Y\otimes-} \ar@<1.5em>@{->>}[d]^-{[f_Y,-]}
\\
& \kern-5em
 \cat T(V):=f_Y\otimes \cat T=[f_Y,\cat T]=(\cat T_Y)^{\perp}
 \kern-5em
 \ar@{ >->}[u]|-{\incl}
}
$$
in which each of the six ``vertical" sequences $\sbull\into \cat T \onto \sbull$ is a (Bousfield) localization sequence of triangulated categories. In particular, the localization \mbox{$f_Y\otimes-:$} $\cat T\onto \cat T(V)=f_Y\otimes \cat T$ is a tt-functor, the tensor on~$\cat T(V)$ being that of~$\cat T$. Moreover, $\cat T(V)$ is also a rigidly-compactly generated tt-category whose subcategory of compact objects is, by Neeman's Theorem~\cite[Thm.\,2.1]{Neeman92b}, the idempotent completion of the image of the compact objects of~$\cat T$:
$$
(f_Y\otimes \cat T^c)^\natural\cong\cat T(V)^c.
$$
The spectrum of this tt-category~$\cat T(V)^c$ identifies with~$V=\Spc(\cat T^c)\sminus Y$,
\begin{equation}
\label{eq:Spc(T(V))}%
\Spc(\cat T(V)^c)\cong V,
\end{equation}
via $\Spc(f_Y\otimes-:\cat T^c\to \cat T(V)^c)$; see~\cite[\S 2.2]{BalmerICM}. The internal hom $[-,-]$ on~$\cat T$ restricts to an internal hom on the local category $\cat T(V)=f_Y\otimes\cat T$. However, localization $f_Y\otimes-:\cat T\to \cat T(V)$ is \emph{not} closed monoidal. Indeed, one can show that for every $x,y\in \cat T$, the following triangle is exact in~$\cat T(V)$:
$$
f_Y\otimes [x,e_Y\otimes y] \to f_Y\otimes [x,y] \to [f_Y\otimes x, f_Y\otimes y] \to \cdot
$$
The second object is the image of the internal hom and the third is the internal hom of the images. The first object describes the difference.
\end{Rem}

\begin{Rem}
\label{rem:e-f}%
For every $x,y\in \cat T$, it follows from $[e_Y\otimes \cat T,f_Y\otimes \cat T]=0$ and from the exact triangle $x\otimes \Delta_Y$ that we have $[x,f_Y\otimes y]\cong [f_Y\otimes x,f_Y\otimes y]$. Similarly, $[e_Y\otimes x,y]\cong [e_Y\otimes x, e_Y\otimes y]$.
\end{Rem}

\begin{Lem}
\label{lem:splitting_lemma}%
Let $e_Z \to \unit \to f_Z \to \Sigma e_Z$ be the idempotent triangle in a rigidly-compactly generated tt-category~$\cat T$ associated to a Thomason \emph{closed} subset $Z \subseteq \Spc(\cat T^c)$. If there is a non-trivial decomposition $e_Z = e_1 \oplus e_2$ with $e_1 \otimes e_2 = 0$ then $Z$ is disconnected, $Z=Z_1\sqcup Z_2$, for closed~$Z_i$ such that $e_i=e_{Z_i}$ for $i=1,2$.
\end{Lem}

\begin{proof}
Let $\cat J := \cat T^c_Z$ denote the compact part generating the idempotent triangle; that is, $e_Z\otimes \cat T = \Loc(\cat J) = \Ker(f_Z\otimes-)$.  Since $Z$ is Thomason closed, there is (by \cite[Prop.~2.14]{Balmer05a}) some $w \in \cat T^c$ such that $\supp(w) = Z$ and $\cat J = \thick_\otimes(w)$ (the thick tensor-ideal of $\cat T^c$ generated by $w$). The object $w$ belongs to~$\cat J\subset\Loc(\cat J)$, hence $e_Z \otimes w \cong w$. So, we obtain a decomposition $w\cong w_1 \oplus w_2$ with $w_1 \otimes w_2 = 0$ by setting $w_1 := e_1 \otimes w$ and $w_2 := e_2 \otimes w$. Moreover, it is simple to check that $\cat J_1 := e_1 \otimes \cat J = \thick_\otimes(w_1)$ and $\cat J_2 := e_2 \otimes \cat J = \thick_\otimes(w_2)$. Finally, these thick tensor-ideals $\cat J_i$ are nonzero. For example, if $e_1 \otimes \cat J = 0$ then $\Loc(\cat J)$ is contained in $\Ker(e_1 \otimes -)$ and, since $e_Z\in \Loc(\cat J)$, we would have $0 = e_1 \otimes e_Z \cong e_1$ contradicting the assumption that $e_1\neq0$. In conclusion, $Z= \supp(w) = \supp(w_1) \sqcup \supp(w_2)$ is a disjoint union of non-empty closed sets and the rest follows easily.
\end{proof}

\begin{Def}[Greenlees~\cite{Greenlees01}]
\label{def:t_Y}%
Let $Y\subseteq \Spc(\cat T^c)$ be a Thomason subset. We define the \emph{Tate functor with respect to~$Y$}
to be
$$
t_Y:=[f_Y,\Sigma e_Y\otimes -]\,:\ \cat T\to \cat T.
$$
\end{Def}

\begin{Rem}
\label{rem:Tate-lax}%
Many classical results about the Tate spectrum~\cite{GreenleesMay95b} hold for this more general construction.
For example, we have
\begin{equation}
\label{eq:Tate}%
t_Y = [f_Y,\Sigma e_Y \otimes -] \cong f_Y \otimes [e_Y,-]
\end{equation}
(cf.~\cite[Corollary~2.5]{Greenlees01}).
Furthermore, the Tate functor $t_Y$ is $\cat T^c$-linear, \ie
\begin{equation}
\label{eq:t_Y@comp}%
t_Y(x)\otimes y\cong t_Y(x\otimes y)
\end{equation}
for all $x\in\cat T$ and $y\in \cat T^c$. Indeed, every compact object $y\in\cat T^c$ is rigid, hence $[-,-]\otimes y\cong [-,-\otimes y]$. Finally, although $t_Y$ is not a monoidal functor, it is lax-monoidal (from the ``lax-monoidality" of $[-,-]$, the fact that $e_Y$ and $f_Y$ are $\otimes$-idempotents, and description~\eqref{eq:Tate} of~$t_Y$). Hence $t_Y(-)$ preserves ring objects and modules.  In fact, if $R$ is a ring object in $\cat T$ then not only is $t_Y(R)$ also a ring object, but also the natural map $R \to t_Y(R)$ is a ring homomorphism.  If $M$ is an $R$-module then $t_Y(M)$ inherits the structure of a $t_Y(R)$-module and hence can also be regarded as an $R$-module.
\end{Rem}

\begin{Prop}
\label{prop:F(idemp)}%
Let $F:\cat T\to \cat T'$ be a tt-functor preserving coproducts. Let $\cat K=\cat T^c$ and $\cat K'=(\cat T')^c$ and let $\varphi:\Spc(\cat K')\to \SpcK$ be the induced map. Let $Y\subseteq\SpcK$ be a Thomason subset and set $Y':=\varphi\inv(Y)\subseteq \Spc(\cat K')$.
Then
$\cat K'_{Y'} = \thick_\otimes( F(\cat K_Y) )$
and there is a unique isomorphism of idempotent triangles
$$
F\big(e_Y\to \unit \to f_Y\to \Sigma e_Y \big)\cong \big(e_{Y'}\to \unit \to f_{Y'}\to \Sigma e_{Y'}\big)
$$
in~$\cat T'$. If moreover $F$ is \emph{closed} monoidal then $F\circ t_Y\cong t_{Y'}\circ F$.
\end{Prop}

\begin{proof}
For any collection of objects $\cat E \subset \cat K$, the tt-ideal $\thick_\otimes(\cat E )$
corresponds to the Thomason subset $\bigcup_{x \in \cat E} \supp(x)$.
The equality of tt-ideals $\cat K'_{Y'} = \thick_\otimes( F(\cat K_Y) )$
then follows from the following equality of Thomason subsets of~$\Spc(\cat K')$:
\[
Y' = \varphi^{-1}(Y) = \varphi^{-1}(\bigcup_{x \in \cat K_Y}\supp(x))
= \bigcup_{x \in \cat K_Y} \varphi^{-1}(\supp(x))
= \bigcup_{x \in \cat K_Y} \supp(Fx).
\]
Then $\cat T'_{Y'}=\Loc(\cat K'_{Y'})=\Loc(F(\cat K_Y))$ and it follows that $F(f_Y)\otimes\cat T'_{Y'}=0$. Since $F$ preserves coproducts,  $F(e_Y)\in F(\Loc(\cat K_Y ))\subseteq \Loc(F(\cat K_Y))=\cat T'_{Y'}$. It then follows from the uniqueness of idempotent triangles (cf.~\cite[Theorem~3.5]{BalmerFavi11}) that the triangle $F(e_Y) \to \unit \to F(f_Y) \to \Sigma F(e_Y)$ is isomorphic to $e_{Y'} \to \unit \to f_{Y'} \to \Sigma e_{Y'}$. The statement about the Tate construction then follows from the definitions.
\end{proof}

\begin{Exa}
\label{ex:fn}%
Let $e_{p,n} \to \unit \to f_{p,n} \to \Sigma e_{p,n}$ be the idempotent triangle in $\SH$
associated to the tt-ideal
$\cat C_{p,n} \subseteq \SHc =: \cat K$. Here, $\cat C_{p,n} = \cat K_{Y_{p,n}}$ for $Y_{p,n} = \supp(\cat C_{p,n})=\SET{\cat Q}{\cat C_{p,n}\not\subseteq \cat Q}=\SET{ \cat C_{p,m} }{m \ge n+1}\cup\SET{\cat C_{q,m}}{q\neq p, m>1}$ in~$\Spc(\SHc)$.
The right idempotent $f_{p,n}$ is sometimes denoted~$L_{n-1}^f S^0$ in the literature (the exponent~``$f$'' referring to ``finite" localization and the index ``$n-1$" referring to the Morava $K$-theory in use).

It follows from Proposition~\ref{prop:F(idemp)} that
$Y^G_{p,n} := \Spc(\triv)^{-1}(Y_{p,n}) = \big\{ \cat P(H,p,m) \mid H \le G, m \ge n+1 \big\}\cup\SET{\cat P(H,q,m)}{H\le G, q\neq p, m>1}$
is a Thomason subset of $\Spc(\SHGc)$ with associated idempotent triangle
\begin{equation}
\label{eq:DeltaZn}%
\Delta_{Y^G_{p,n}} = \big( \triv(e_{p,n}) \to \unit \to \triv(f_{p,n}) \to \Sigma \triv(e_{p,n}) \big)
\end{equation}
in~$\SH(G)$. The corresponding localized category
\[
\cat{SH}(G)_{p,n}:=\triv(f_{p,n})\otimes\SHG
\]
will be very useful in our discussion of~$\Spc(\SHGc)$. It can be called \emph{the truncation of~$\SHG$ below $p$-chromatic level~$n$}.

Note that the $n=\infty$ case is just localization at~$p$. In particular, localization with respect to~$f_{p,n}$ in~$\SH$, or $\triv(f_{p,n})$ in~$\SHG$, already localizes at~$p$ for free.
\end{Exa}

\begin{Exa}
\label{ex:fF}%
For a family $\cat F$ of subgroups of~$G$, \emph{always assumed below to be stable under conjugation and under passing to subgroups}, we can consider the tt-ideal $\cat C_{\cat F} := \thickt{G/L_+ \mid L \in \cat F}$ of $\SHGc$. (By the Mackey formula, we can drop $\otimes$ in the notation~$\thickt{...}$ but this is a pedantic detail.) This tt-ideal $\cat C_{\cat F}$ corresponds to the Thomason closed subset $Y_{\cat F}:=\bigcup_{L \in \cat F} \supp(G/L_+)$. By Corollary~\ref{cor:supp(G/K)}, we have $Y_{\cat F}=\SET{\cat P(H,p,n)}{H \in \cat F\mathrm{,\ all\ }p,n}$. In this example, we write
\[
\Delta_{\cat F}=\big(\ e_{\cat F} \to \unit \to f_{\cat F} \to \Sigma e_{\cat F}\ \big)
\]
for the associated idempotent triangle $\Delta_{Y_{\cat F}}$ in~$\SH(G)$. We also use the notation
\[
t_{\cat F} := [f_{\cat F}, \Sigma e_{\cat F}\otimes -]
\]
for the associated Tate functor. In more usual notation, $e_{\cat F} = \Sigma^\infty_G(E\cat F_+)$ and $f_{\cat F} = \Sigma^\infty_G(\tilde{E}\cat F)$.
\end{Exa}

\begin{Rem}
\label{rem:triv-fams}%
The \emph{trivial} families are
$\cat F=\varnothing$ and $\calF=\{\textrm{all subgroups}\}$ whose idempotent triangles are $0\to \unit \to \unit \to 0$ and $\unit \to \unit \to 0 \to \Sigma \unit$ respectively. Their Tate functors are zero.
\end{Rem}

\begin{Rem}
\label{rem:t_G}%
The extreme non-trivial examples are the family $\cat F=\{ 1\}$ consisting only of the trivial subgroup, and the family $\cat{F}_{\proper}$ consisting of all proper subgroups.  For the former, the notation $e_{\{1\}} = \Sigma^\infty_G EG_+$ and $f_{\{1\}} = \Sigma^\infty_G \tilde{E}G$ is probably more familiar to some readers.  The Tate functor $t_G := t_{\{1\}}$ is the original construction studied in~\cite{GreenleesMay95b}.  We shall also use the notation $e_G := e_{\{1\}}$ and $f_G := f_{\{1\}}$ to keep some semblance to the classical notation.
\end{Rem}

\begin{Exa}
\label{ex:Fp^s}%
For every group~$G$, one can define a family of subgroups by bounding the order of the subgroups, or equivalently their index in~$G$.  We shall denote by $\cat{F}_{\le\ell}:=\SET{H\le G}{|H|\le\ell}$ the family of subgroups of order at most~$\ell$. Such families will play an important role in Section~\ref{se:log-p-conjecture}.
\end{Exa}

\begin{Exa}
\label{ex:res-Tate}%
Let $\cat F$ be a family of subgroups of~$G$ as in Example~\ref{ex:fF}. Let $H\le G$ be a subgroup and set $\cat F\cap H:=\SET{K\in \cat F}{K\le H}$. Then we can apply Proposition~\ref{prop:F(idemp)} to the tt-functor $\Res^G_H:\SHG\to \SH(H)$, which is closed monoidal. Combining this with Example~\ref{ex:fF} and the fact that $\Res^G_H(G/L_+)\simeq\oplus_{[g]\in H\backslash G/L}(H/H\cap {}^gL)_+$, it is easy to see that $\Res^G_H\circ t_{\cat F}\cong t_{\cat F\cap H}\circ\Res^G_H$.
\end{Exa}

\begin{Rem}
\label{rem:normal-family}%
We have already encountered in~\ref{it:phi^N} the localization of~$\SHG$ with respect to the family~$\cat F[{\not\ge}N]:=\SET{K \le G}{K \not\supseteq N}$ associated to a normal subgroup~$N\normal G$. The thick $\otimes$-ideal here is $\thick_\otimes\SET{G/K_+ }{K \not\supseteq N}$. We repeat the crucial fact that the corresponding localization $f_{\cat F[\not\ge N]}\otimes \SHG\cong\SH(G/N)$ yields the equivariant stable homotopy category for the quotient group~$G/N$, in such a way that $\phigeom{N}:\SHG\to \SH(G/N)$ becomes the localization functor.
\end{Rem}

\begin{Rem}
According to~\cite{HoveySadofsky96} and~\cite{Kuhn04}, the original Tate construction $t_G$ (Rem.\,\ref{rem:t_G}) lowers chromatic degree by one. The following theorem states this result in the form we will need.
\end{Rem}

\begin{Thm}[Hovey-Sadofsky, Kuhn]
\label{thm:kuhovsky}%
Let $p$ be a prime dividing the order of a finite group $G$, and let $f_{p,n}$ denote the right idempotent for the localization of $\SH$ associated to $\cat C_{p,n}$ (see~Example~\ref{ex:fn}). For any $n \ge 2$, the kernel of the functor $t_G(\triv(f_{p,n}))^G \otimes - : \SHc \to \SH$ is precisely $\cat C_{p,n-1}$.
\end{Thm}

\begin{proof}
Note that if $x \in \SHc$ then $t_G(\triv(f_{p,n}))^G\otimes x = t_G(\triv(f_{p,n} \otimes x))^G$. Define $\cat T_G(-):= t_G(\triv(-))^G$ for convenience. The main results of~\cite[Theorem~1.5 and Lemma~4.1]{Kuhn04} imply that $\cat T_G(f_{p,n}\otimes x) = 0$ for all $x \in \cat C_{p,n-1}$. Just observe that $f_{p,n} = L_{n-1}^f S^0$ and that $f_{p,n} \otimes x \simeq L_{n-1}^f(x) \simeq \Tel(x;v_{n-1})$ for any $x \in \cat C_{p,n-1} \sminus \cat C_{p,n}$.

Conversely, $L_m(S^0)$ is an $L_m^f(S^0)$-module, hence $\cat T_G(L_m(S^0))$ is a $\cat T_G(L_m^f(S^0))$-module, so that $\cat T_G(L_m(S^0))$ is $\cat T_G(L_m^f(S^0))$-local.  It follows that if $x \in \SHc$ is a finite spectrum such that $\cat T_G(L_{m}^f(S^0))\otimes x = 0$ then
\begin{equation}
\label{eq:kuhovsky-eq-temp}%
\cat T_G(L_{m}(S^0)) \otimes x^\vee = [x,\cat T_G(L_{m}(S^0))] = 0.
\end{equation}
However, the main result of~\cite{HoveySadofsky96} establishes that $\cat T_G(L_m(S^0))$ and $L_{m-1}(S^0)$ have the same Bousfield class provided that $p$ divides the order of~$G$.  Hence \eqref{eq:kuhovsky-eq-temp} implies that $L_{m-1}(S^0) \otimes x^\vee = L_{m-1}(x^\vee)=0$ so that $x^\vee$ (and hence $x$) is contained in $\cat C_{p,m}$.  Hence, plugging $m=n-1$ and using $f_{p,n} = L_{n-1}^f S^0$ again, we see that $\cat T_G(f_{p,n} \otimes x) = 0$ for a finite spectrum $x \in \SHc$ implies that $x \in \cat C_{p,n-1}$.
\end{proof}

\begin{Rem}
\label{rem:t_G-with-fixed-points}%
If $G=C_p$ then the family of proper subgroups coincides with the trivial family $\{1\}$.  In light of Remark~\ref{rem:normal-family}, we see that in this case, $t_G(-)^G \cong \phigeomb{G}t_G(-)$.  This is something very special about the case of the cyclic group~$C_p$ and is certainly not true for arbitrary groups.  This remark lies at the heart of our application of Theorem~\ref{thm:kuhovsky} and our conjectured blue-shift results in Section~\ref{se:log-p-conjecture}.
\end{Rem}

We conclude this section by describing how chromatic localization (Example~\ref{ex:fn}) affects the tt-spectrum of~$\SHGc$.

\begin{Prop}
\label{prop:loc}%
Let $p$ be a prime and let $1\le n\le \infty$ be a chromatic integer. Consider the tt-category $\cat{SH}(G)_{p,n}=\triv(f_{p,n})\otimes\SHG$. Then the localization functor $\SHGc\to \cat{SH}(G)_{p,n}^c$ induces a homeomorphism between $\Spc(\cat{SH}(G)_{p,n}^c)$ and the following subset $V_{p,n}$ of~$\Spc(\SHGc)$:
\[
V_{p,n}:=\SET{\cat P(H,p,n)}{H\le G,\ 1\le m\le n}.
\]
\end{Prop}

\begin{proof}
The statement is a special case of~\eqref{eq:Spc(T(V))} in Remark~\ref{rem:loc}, for the big tt-category $\cat T=\SHG$ and for the Thomason subset $Y=Y^G_{p,n}$ of its spectrum, where~$Y^G_{p,n}$ is given in Example~\ref{ex:fn}. Its complement $\Spc(\SHGc)\sminus Y^G_{p,n}$ is exactly our~$V_{p,n}$.
\end{proof}

\begin{Cor}
\label{cor:loc}%
For every prime~$p$, the spectrum of $\SHGcp$ is homeomorphic to the subset $V_{p,\infty}:=\SET{\cat P(H,p,m)}{H\le G,\ 1\le m\le \infty}$ of~$\Spc(\SHGc)$.
\qed
\end{Cor}

\begin{Rem}
As is customary, we shall use the same notation $\cat P(H,p,m)$ for the prime in~$\SHGc$ and the corresponding prime in~$\cat{SH}(G)^c_{p,n}$.
\end{Rem}

\smallskip
\section{Inclusions of primes and the comparison map}
\label{se:pre-top}%
\medskip

Our goal is to determine the topology of~$\Spc(\SHGc)$. We begin by reducing the problem to the task of understanding all inclusions among the equivariant primes:

\begin{Prop}
\label{prop:finite-union-irred}%
Every closed subset of $\Spc(\SHGc)$ is a finite union of irreducible closed subsets.  Every (non-empty) irreducible closed subset is of the form
$\adhpt{\cat P}=$ \mbox{$\SET{\cat Q\in \Spc(\SHGc)}{\cat Q\subseteq\cat P}$} for some unique point~$\cat P$.
\end{Prop}

\begin{proof}
It is basic tt-geometry~\cite{Balmer05a} that every irreducible closed subset of $\SpcK$ has a unique generic point and that the closure of a point~$\cat P$ is exactly $\SET{\cat Q\in\SpcK}{\cat Q\subseteq \cat P}$; beware the reversal of the inclusion compared to the Zariski topology. For the non-equivariant case, the proposition was proved in~\cite[Corollary~9.5(e)]{Balmer10b}.  By Theorem~\ref{thm:the-set}, the space $\Spc(\SHGc)$ is covered by the images of~$\varphi^{H,G}=\Spc(\phigeomb{H}):\Spc(\SHc)\to \Spc(\SHGc)$ for subgroups~$H\le G$. Then consider any closed subset $Z \subseteq \Spc(\SHGc)$. By what we have just said, $Z=\cup_{H\le G}\,\varphi^{H,G}((\varphi^{H,G})\inv(Z))$.  For any $H\le G$, $(\varphi^H)^{-1}(Z)$ is a closed subset of $\Spc(\SHc)$. So, by the non-equivariant result it is a finite union of irreducible closed sets $\overline{\{ \cat C_{p,n} \}}$ for some primes $p$ and integers $1 \le n \le \infty$.  It follows that $Z$ is the union of the closures of the points $\cat P(H,p,n)$ corresponding to the generic points $\cat C_{p,n}$ of the irreducible components of $(\varphi^H)^{-1}(Z)$, for all $H \le G$. As there are only finitely many subgroups $H \le G$, and finitely many generic points for $(\varphi^H)^{-1}(Z)$, this is a finite union.
\end{proof}

\begin{Prop}
\label{prop:incl-same-subgroup}%
For every subgroup $H \le G$, primes~$p,q$, and chromatic integers $1\le m,n\le \infty$, the following are equivalent:
\begin{enumerate}[label=\rm(\arabic*)]
\item
$\cat P(H,q,n) \subseteq \cat P(H,p,m)$;
\item
$\cat C_{q,n}\subseteq\cat C_{p,m}$;
\item
$n \ge m$ and, either $m=1$ or $q=p$.
\end{enumerate}
\end{Prop}

\begin{proof}
The inclusions among the $\cat C_{p,n}$ manifest as inclusions among the $\cat P(H,p,n)$, which are their preimages under~$\phigeomb{H}$. That these are the only inclusions among the primes $\cat P(H,p,m)$ for fixed $H$ follows from Corollary~\ref{cor:project-H}.
\end{proof}

\begin{Rem}
In other words, the inclusions among the primes $\cat P(H,p,m)$ -- for a fixed $H$ -- precisely match the inclusions among the non-equivariant primes $\cat C_{p,m}$.  Thus, $\Spc(\SHGc)$ consists of a number of copies of the non-equivariant spectrum $\Spc(\SHc)$ displayed in~\eqref{eq:Spc(SH)} -- one such copy for each conjugacy class of subgroups of~$G$.  The question is what the interaction is among the different copies; that is, what are the inclusions $\cat P(K,q,n) \subseteq \cat P(H,p,m)$ for $K \not\sim_G H$?
\end{Rem}

We can start by projecting onto the chromatic information, via Corollary~\ref{cor:project-H}:

\begin{Cor}
\label{cor:incl-chrom}%
Suppose $\cat P(K,q,n) \subseteq \cat P(H,p,m)$ for two subgroups $H,K \le G$, primes $p,q$ and chromatic integers $1\le n,m \le \infty$.
Then $\cat C_{q,n}\subseteq \cat C_{p,n}$ in~$\SHc$ and in particular $n\ge m$. If furthermore $m>1$, then $p=q$.
\qed
\end{Cor}

\begin{Rem}
\label{rem:incl-at-p}%
If $m=1$ in the previous statement, we cannot conclude that $p=q$ but this is slightly artificial because then $\cat P(H,p,1)=\cat P(H,0,1)=\cat P(H,q,1)$ and $p$ is irrelevant anyway. In that case, we are equivalently discussing the inclusion $\cat P(K,q,n)\subseteq \cat P(H,q,m)$. In other words, when studying inclusions $\cat P(K,q,n) \subseteq \cat P(H,p,m)$ between primes in~$\SHGc$, we can just as well assume~$p=q$.
\end{Rem}
\begin{Rem}
\label{rem:yes-we-can-localize}%
Recall the homeomorphism $\Spc(\cat{SH}(G)_{p,n}^c) \simeq V_{p,n}$ of Proposition~\ref{prop:loc}. This homeomorphism respects and detects inclusions among primes. So, whenever we have to decide whether $\cat P(K,p,n) \subseteq \cat P(H,p,m)$ holds, we can just as well localize in~$\cat{SH}(G)_{p,n}=\triv(f_{p,n})\otimes \SHG$ or in~$\cat{SH}(G)_{p,\infty} = \SHGp$.
\end{Rem}

Next we can project onto the group-theoretic information:
\begin{Prop}
\label{prop:comp-map}%
The comparison map $\rho_{\SHGc} : \Spc(\SHGc) \to \Spec(A(G))$ is inclusion-\emph{reversing}. It sends $\cat P(H,p,1)=\cat P(H,0,1)$ to $\mathfrak{p}(H,0)$ and it sends $\cat P(H,p,n)$ to $\mathfrak{p}(H,p)$ for $n>1$.
\end{Prop}

\begin{proof}
Recall that the isomorphism $A(G) \cong \End_{\SH(G)}(\unit)$ sends the class $[G/K]$ to the composite $\unit \to G/K_+ \to \unit$ where $\unit \to G/K_+$ is the stable transfer map (\ie the unit for the ring structure on the $G$-spectrum $G/K_+$) and $G/K_+ \to \unit$ is the stable map induced by the projection $G/K_+ \to S^0$.  Applying $\Phi^H$ this composite becomes $\unit \to \unit^{\oplus \ell} \to \unit$, a sum of $\ell$ copies of the identity map, where $\ell = |(G/K)^H|$ (cf.~\ref{it:X^H} and Lemma~\ref{lem:G/K^H}).  It follows that the ring homomorphism $A(G) \cong \End_{\SH(G)}(\unit) \to \End_{\SH}(\unit)\cong \mathbb Z$ induced by the geometric fixed point functor $\phigeomb{H}:\SH(G) \to \SH$ is precisely the ring homomorphism ${f^{H}}$ from Section~\ref{se:burnside} and therefore the following diagram
\[ \xymatrix @C=5em{
\Spc(\SHc) \ar[d]_{\rho_{\SHc}} \ar[r]^{\Spc(\phigeomb{H})} & \Spc(\SHGc) \ar[d]^{\rho_{\SHGc}} \\
\Spec(\mathbb Z) \ar[r]^{\Spec({f^{H}})}& \Spec(A(G))
}\]
commutes (by naturality of the $\rho_{\cat K}$ construction, see~\cite{Balmer10b}). The result then follows from the definitions and the description of~$\rho_{\cat K}$ recalled in~\eqref{eq:comp}.
\end{proof}

\begin{Exa}
\label{ex:incl-0}%
If $\cat P(K,q,1)\subseteq\cat P(H,p,m)$ then $m=1$ by Corollary~\ref{cor:incl-chrom} and $\mathfrak{p}(K,0)=\mathfrak{p}(H,0)$ by Proposition~\ref{prop:comp-map}. The latter forces $H\sim_G K$ by Dress's Theorem~\ref{thm:dress}. So we have equality $\cat P(K,q,1)=\cat P(K,0,1)=\cat P(H,0,1)=\cat P(H,p,m)$.
\end{Exa}

\begin{Prop}
\label{prop:burnside-info}%
Suppose $\cat P_G(K,p,n) \subseteq \cat P_G(H,p,m)$ for subgroups \mbox{$H,K \le G$}, a prime $p$ and chromatic integers $1\le n,m \le \infty$.
Then the following hold:
\begin{enumerate}
\item $K$ is $G$-conjugate to a $p$-subnormal subgroup of $H$; \ie $K^g \le H$ and $\Op (K^g) = \Op (H)$ for some $g \in G$ (see~\ref{ter:subnormal}). Moreover, we can find $g\in G$ such that $\cat P_H(K^g,p,n) \subseteq \cat P_H(H,p,m)$ in~$\SH(H)^c$.
\item If $n = 1$ then $m=1$ and $K \sim_G H$ (so the two tt-primes were equal).
\end{enumerate}
\end{Prop}

\begin{proof}
By Corollary~\ref{cor:G/K^H}, we must have $K \le_G H$.  So, replacing $K$ by a conjugate, we can assume that $K \le H$.  Then Going-Up~\ref{it:going-up} applied to $\Res^G_H$ implies that there is a prime $\cat P_H(K',p',n') \in \Spc(\SH(H)^c)$ such that $\cat P_H(K',p',n') \subseteq \cat P_H(H,p,m)$ and such that $\cat P_G(K',p',n') = \cat P_G(K,p,n)$ (cf.\ Cor.\,\ref{cor:induced-res}).  The latter equality implies by Theorem~\ref{thm:prime-uniqueness} that $K \sim_G K'$ and $\cat C_{p,n}=\cat C_{p',n'}$.  Hence we have an inclusion $\cat P_H(K',p,n) =\cat P_H(K',p',n') \subseteq \cat P_H(H,p,m)$.  Applying Proposition~\ref{prop:comp-map} yields an inclusion of primes in the Burnside ring.  By Dress's description (cf.~Theorem~\ref{thm:dress}) of the inclusions of prime ideals of~$A(H)$, we must have $\Op (K') \sim_H \Op (H)$.  However, $\Op (H) \lenormal H$, so in fact $\Op (K') = \Op (H)$.  (Lemma~\ref{lem:p-subnormal} explains how this means that $K'$ is a $p$-subnormal subgroup of~$H$.) Part~(b) is Example~\ref{ex:incl-0}.
\end{proof}

\begin{Rem}
\label{rem:incl-at-p-in-G}%
By Remark~\ref{rem:incl-at-p} and Proposition~\ref{prop:burnside-info}, we thus need only understand the inclusions $\cat P_G(K,p,n) \subseteq \cat P_G(H,p,m)$ for the same prime~$p$ and for \mbox{$p$-subnormal} subgroups~$K$ of~$H$.  In particular, we need only consider those primes $p$ dividing the order of~$G$. (Otherwise, no subgroup of $G$ has any proper $p$-subnormal subgroup.)
\end{Rem}

Finally, we can further reduce our problem to the case where $G$ is a $p$-group:

\begin{Prop}
\label{prop:reduce-to-p-groups}%
Let $H,K \le G$ be subgroups. Then for any $1 \le n,m \le \infty$, we have an inclusion $\cat P_G(K,p,n) \subseteq \cat P_G(H,p,m)$ if and only if $K$ is $G$-conjugate to a $p$-subnormal subgroup $K^g \le H$ such that $\cat P_{H/N}(K^g/N,p,n) \subseteq \cat P_{H/N}(H/N,p,m)$ where $N := \Op (H) = \Op (K^g)$. Note that the subquotient $H/N$ of~$G$ is a $p$-group.
\end{Prop}

\begin{proof}
By Proposition~\ref{prop:burnside-info}\,(a), there exists a $G$-conjugate $K^g$ of~$K$ such that $K^g \le H$, such that $\cat P_H(K^g,p,n) \subseteq \cat P_H(H,p,m)$, and such that $\Op (H) = \Op (K^g)$. Let $N:=\Op (H) = \Op (K^g)$ be this normal subgroup of~$H$.  Under the map induced by inflation $\Infl_{G/N}^G$ (Cor.\,\ref{cor:induced-infl}), our inclusion becomes $\cat P_{H/N}(K^g/N,p,n) \subseteq \cat P_{H/N}(H/N,p,m)$.  Conversely, the latter inclusion of primes in $\SH(H/N)$ gets mapped by $\Spc(\phigeom{N})$ to the inclusion $\cat P_H(K^g,p,n) \subseteq \cat P_H(H,p,m)$ by Prop.\,\ref{prop:induced-geom}, which then becomes $\cat P_G(K^g,p,n) \subseteq \cat P_G(H,p,m)$ under restriction (Cor.\,\ref{cor:induced-res}).
\end{proof}

\section{The topology of $\Spc(\SH(C_p)^c)$}
\label{se:C_p}
\medskip

Let us consider $G=C_p$ cyclic of prime order. Since there are two subgroups, there are only two classes of primes in $\SHGc$:
$$
\cat P(C_p,q,n)=\SET{x\in\SHGc}{\phigeomb{G}(x)\in\cat C_{q,n}}
$$
and
$$
\cat P(1,q,n)=\SET{x\in\SHGc}{\Res^G_1(x)\in\cat C_{q,n}}\,.
$$
By Remark~\ref{rem:incl-at-p-in-G}, we only need to understand when there is an inclusion $\cat P(1,p,n) \subseteq \cat P(C_p,p,m)$, at the prime~$q=p$.

\begin{Prop}
\label{prop:C_p-inclusion}%
We have $\cat P(1,p,n)\subseteq \cat P(C_p,p,n-1)$ for every $n\geq 2$.
\end{Prop}

\begin{proof}
Let $x\in \SHGc$ be a compact $G$-spectrum such that $\Res^G_1(x) \in \cat C_{p,n}$.  This means that $0=f_{p,n} \otimes \Res^G_1(x) \cong \Res^G_1(\triv(f_{p,n})\otimes x)$ where $f_{p,n}$ is, as always, the idempotent associated to $\cat C_{p,n}$ (cf.~Ex.\,\ref{ex:fn}).  We deduce that $G_+\otimes \triv(f_{p,n})\otimes x=0$. Therefore $\Loc_{\otimes}(G_+)\otimes \triv(f_{p,n})\otimes x=0$; in particular $e_G\otimes \triv(f_{p,n})\otimes x =0$, where $e_G \to \unit \to f_G \to \Sigma e_G$ is the idempotent triangle in~$\SHG$ associated to the smashing subcategory $\Loc_{\otimes}(G_+)$ (cf.~Remark~\ref{rem:t_G}).
It follows that
\[t_G(\triv(f_{p,n})\otimes x)=[f_G,\Sigma e_G\otimes \triv(f_{p,n})\otimes x]=0.\]
Applying geometric fixed points we obtain
\[0=\phigeomb{G}(t_G(\triv(f_{p,n})\otimes x)) = \phigeomb{G}(t_G(\triv(f_{p,n}))\otimes x)=\phigeomb{G}(t_G(\triv(f_{p,n})))\otimes \phigeomb{G}(x).\]
Note that $\phigeomb{G}(t_G(\triv(f_{p,n}))) = t_G(\triv(f_{p,n}))^G$ since $G=C_p$ (cf.~Remark~\ref{rem:t_G-with-fixed-points}) hence Theorem~\ref{thm:kuhovsky} implies that $f_{p,n-1} \otimes \phigeomb{G}(x)=0$.  That is, $\phigeomb{G}(x) \in \cat C_{p,n-1}$.
\end{proof}

\begin{Cor}
\label{cor:C_p-incl-infty}%
We have $\cat P(1,p,\infty) \subseteq \cat P(C_p,p,\infty)$.
\end{Cor}

\begin{proof}
This is immediate from the above proposition, since $\cat C_{p,\infty} = \bigcap_{n \ge 1} \cat C_{p,n}$.
\end{proof}

\begin{Rem}
\label{rem:geom}%
We know that $\cat P(1,p,n) \not\subseteq \cat P(C_p,p,n+1)$, \eg\ by Cor.\,\ref{cor:incl-chrom}. So, the only remaining question is whether or not $\cat P(1,p,n)$ can be included in~$\cat P(C_p,p,n)$.

\smallskip
Although not necessary in the formal arguments for the hyper-brainy readers, there is a way to build intuition from geometric pictures which can help average readers, down around 160 I.Q.\ points.
Localizing at~$p$, we have by Corollary~\ref{cor:loc} the following (left-hand side) towers of primes in~$\SH(C_p)^c$:
\[ \xy
  (4,-5)*{\bullet};
  (-4,-5)*{\scriptstyle \cat P(1,p,1)};
  (2,5)*{\bullet};
  (-6,5)*{\scriptstyle \cat P(1,p,2)};
  (2,10)*{\bullet};
  (-6,10)*{\scriptstyle \cat P(1,p,3)};
  (2,16)*{\vdots};
  (2,20)*{\bullet};
  (-8,20)*{\scriptstyle \cat P(1,p,n-1)};
  (2,25)*{\bullet};
  (-6,25)*{\scriptstyle \cat P(1,p,n)};
  (2,33)*{\vdots};
  (2,35)*{\bullet};
  (-7,35)*{\scriptstyle \cat P(1,p,\infty)};
{\ar@{-} (4,-5)*{};(2,5)*{}};
{\ar@{-} (2,5)*{};(2,10)*{}};
{\ar@{-} (2,10)*{};(2,13)*{}};
{\ar@{-} (2,17)*{};(2,20)*{}};
{\ar@{-} (2,20)*{};(2,25)*{}};
{\ar@{-} (2,25)*{};(2,29)*{}};
  (15,-5)*{\bullet};
  (24,-5)*{\scriptstyle \cat P(C_p,p,1)};
  (13,5)*{\bullet};
  (22,5)*{\scriptstyle \cat P(C_p,p,2)};
  (13,10)*{\bullet};
  (22,10)*{\scriptstyle \cat P(C_p,p,3)};
  (13,16)*{\vdots};
  (13,20)*{\bullet};
  (23.7,20)*{\scriptstyle \cat P(C_p,p,n-1)};
  (13,25)*{\bullet};
  (22,25)*{\scriptstyle \cat P(C_p,p,n)};
  (13,33)*{\vdots};
  (13,35)*{\bullet};
  (22.5,35)*{\scriptstyle \cat P(C_p,p,\infty)};
{\ar@{-} (15,-5)*{};(13,5)*{}};
{\ar@{-} (13,5)*{};(13,10)*{}};
{\ar@{-} (13,10)*{};(13,13)*{}};
{\ar@{-} (13,17)*{};(13,20)*{}};
{\ar@{-} (13,20)*{};(13,25)*{}};
{\ar@{-} (13,25)*{};(13,29)*{}};
{\ar@{-} (15,-5)*{};(2,5)*{}};
{\ar@{-} (13,5)*{};(2,10)*{}};
{\ar@{-} (13,10)*{};(7.5,12.5)*{}};
{\ar@{-} (7.5,17.5)*{};(2,20)*{}};
{\ar@{-} (13,20)*{};(2,25)*{}};
{\ar@{-} (13,25)*{};(7.5,27.5)*{}};
{\ar@{-} (2,35)*{};(13,35)*{}};
{\ar@{--} (2,25)*{};(13,25)*{}};
{\ar@{--} (2,20)*{};(13,20)*{}};
{\ar@{--} (2,10)*{};(13,10)*{}};
{\ar@{--} (2,5)*{};(13,5)*{}};
(8,-15)*{\Spc(\SH(C_p)^c_{(p)})};
\endxy
\qquad  \supset \qquad
\xy
  (4,-5)*{\bullet};
  (2,5)*{\bullet};
  (2,10)*{\bullet};
  (2,16)*{\vdots};
  (2,20)*{\bullet};
  (2,25)*{\bullet};
  (-6,25)*{\scriptstyle \cat P(1,p,n)};
{\ar@{-} (4,-5)*{};(2,5)*{}};
{\ar@{-} (2,5)*{};(2,10)*{}};
{\ar@{-} (2,10)*{};(2,13)*{}};
{\ar@{-} (2,17)*{};(2,20)*{}};
{\ar@{-} (2,20)*{};(2,25)*{}};
  (15,-5)*{\bullet};
  (13,5)*{\bullet};
  (13,10)*{\bullet};
  (13,16)*{\vdots};
  (13,20)*{\bullet};
  (13,25)*{\bullet};
  (22,25)*{\scriptstyle \cat P(C_p,p,n)};
{\ar@{-} (15,-5)*{};(13,5)*{}};
{\ar@{-} (13,5)*{};(13,10)*{}};
{\ar@{-} (13,10)*{};(13,13)*{}};
{\ar@{-} (13,17)*{};(13,20)*{}};
{\ar@{-} (13,20)*{};(13,25)*{}};
{\ar@{-} (15,-5)*{};(2,5)*{}};
{\ar@{-} (13,5)*{};(2,10)*{}};
{\ar@{-} (13,10)*{};(7.5,12.5)*{}};
{\ar@{-} (7.5,17.5)*{};(2,20)*{}};
{\ar@{-} (13,20)*{};(2,25)*{}};
{\ar@{--} (2,25)*{};(13,25)*{}};
{\ar@{--} (2,20)*{};(13,20)*{}};
{\ar@{--} (2,10)*{};(13,10)*{}};
{\ar@{--} (2,5)*{};(13,5)*{}};
(8,-15)*{\Spc(\cat{SH}(C_p)_{p,n}^c)};
\endxy
\]
The dashed lines $- -$ indicate that the inclusions $\cat P(1,p,n)\overset{?}\subset \cat P(C_p,p,n)$ have not been ruled out yet. As in Proposition~\ref{prop:loc}, we can further localize below chromatic level~$n$, via $\SHG\onto \triv(f_{p,n})\otimes\SHG=\cat{SH}(G)_{p,n}$. This yields the right-hand side tt-spectrum in the above picture. The advantage of doing so is that we see the two primes in question, $\cat P(1,p,n)$ and $\cat P(C_p,p,n)$, at the top of the picture. Consider now the closed subset consisting of those two primes, diagrammatically depicted as
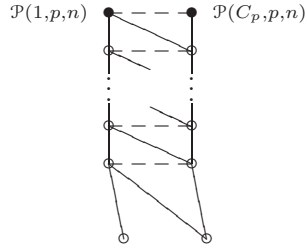
\[
\xy
  (4,-5)*{\circ};
  (2,5)*{\circ};
  (2,10)*{\circ};
  (2,16)*{\vdots};
  (2,20)*{\circ};
  (2,25)*{\bullet};
  (-6,25)*{\scriptstyle \cat P(1,p,n)};
{\ar@{-} (4,-5)*{};(2,5)*{}};
{\ar@{-} (2,5)*{};(2,10)*{}};
{\ar@{-} (2,10)*{};(2,13)*{}};
{\ar@{-} (2,17)*{};(2,20)*{}};
{\ar@{-} (2,20)*{};(2,25)*{}};
  (15,-5)*{\circ};
  (13,5)*{\circ};
  (13,10)*{\circ};
  (13,16)*{\vdots};
  (13,20)*{\circ};
  (13,25)*{\bullet};
  (22,25)*{\scriptstyle \cat P(C_p,p,n)};
{\ar@{-} (15,-5)*{};(13,5)*{}};
{\ar@{-} (13,5)*{};(13,10)*{}};
{\ar@{-} (13,10)*{};(13,13)*{}};
{\ar@{-} (13,17)*{};(13,20)*{}};
{\ar@{-} (13,20)*{};(13,25)*{}};
{\ar@{-} (15,-5)*{};(2,5)*{}};
{\ar@{-} (13,5)*{};(2,10)*{}};
{\ar@{-} (13,10)*{};(7.5,12.5)*{}};
{\ar@{-} (7.5,17.5)*{};(2,20)*{}};
{\ar@{-} (13,20)*{};(2,25)*{}};
{\ar@{--} (2,25)*{};(13,25)*{}};
{\ar@{--} (2,20)*{};(13,20)*{}};
{\ar@{--} (2,10)*{};(13,10)*{}};
{\ar@{--} (2,5)*{};(13,5)*{}};
\endxy
\]
where the two $\bullet$ indicate the two points that we include in the closed subset and the $\circ$ mark the primes we exclude. The inclusion marked by the top dashed line $\bullet - - \bullet$ holds if and only if this closed subset is connected (and is the closure of the right-hand prime, corresponding to $\cat P(C_p,p,n)$). To show that this inclusion does not hold, we are going to show that this closed set is disconnected. As explained in Section~\ref{se:e,f,Tate}, this can be done by carefully analyzing idempotent triangles and this is why we proceed with the following computation.
\end{Rem}

\begin{Lem}
\label{lem:splitting}%
Let $G=C_p$. For any $n \ge 2$, we have a decomposition
\[ \triv(e_{p,n-1}\otimes f_{p,n}) \simeq \big( e_G \otimes \triv(e_{p,n-1} \otimes f_{p,n}) \big) \oplus \big( f_G \otimes \triv(e_{p,n-1} \otimes f_{p,n}) \big)\]
in $\SH(G)$, where $e_{p,n}$ and $f_{p,n}$ are the chromatic idempotents (Ex.\,\ref{ex:fn}) and $e_G$ and $f_G$ are the group-theoretic ones (Ex.\,\ref{ex:fF} and Rem.\,\ref{rem:t_G}).
\end{Lem}

\begin{proof}
More precisely, we prove that the triangle $\Delta_{G} \otimes \triv(e_{p,n-1} \otimes f_{p,n})$ splits, by showing that the last map in this triangle (the vertical one below) is zero:
\[\xymatrix@C=2em{
e_G \otimes \triv(e_{p,n-1}\otimes f_{p,n}) \ar[r] & \triv(e_{p,n-1} \otimes f_{p,n}) \ar[r] & f_G\otimes \triv(e_{p,n-1} \otimes f_{p,n}) \ar[d] \\
&&\Sigma e_G \otimes \triv(e_{p,n-1}\otimes f_{p,n}).
}\]
It suffices to show that $\Hom(f_G \otimes \triv(e_{p,n-1}\otimes f_{p,n}), \Sigma e_G \otimes \triv(e_{p,n-1} \otimes f_{p,n})) = 0$, which in turn follows from
$[f_G \otimes \triv(e_{p,n-1}\otimes f_{p,n}),\Sigma e_G \otimes \triv(e_{p,n-1}\otimes f_{p,n})]= 0$.
Let us temporarily stop writing ``triv'' for readability. Using Remark~\ref{rem:e-f}, we can remove the factor $e_{p,n-1}$ on the right and the factor $f_{p,n}$ on the left. This reads:
\[
[f_G \otimes e_{p,n-1} \otimes f_{p,n}, \Sigma e_G \otimes e_{p,n-1} \otimes f_{p,n}] = [f_G \otimes e_{p,n-1}, \Sigma e_G \otimes f_{p,n}].
\]
So we want to prove that $[f_G\otimes e_{p,n-1}, \Sigma e_G \otimes f_{p,n}] = 0$.

To this end, we start from Theorem~\ref{thm:kuhovsky}, which gives us $t_G(f_{p,n})^G\otimes \cat C_{p,n-1} = 0$.  Thus, $t_G(f_{p,n} \otimes x^\vee)^G = 0$ for all $x \in \cat C_{p,n-1}$.  Moreover, since $G=C_p$, this actually gives an equality $t_G(f_{p,n} \otimes x^\vee) = 0$ in $\SH(G)$ because $t_G(-)$ takes $f_G$-local values and here, $f_G=f_{\proper}$.  (See Rem.\,\ref{rem:t_G-with-fixed-points}.) Thus $[f_G \otimes \triv(x),\Sigma e_G \otimes f_{p,n} ]=0$.  But now, the kernel of the functor $[f_G \otimes \triv(-), \Sigma e_G \otimes f_{p,n}]$ is a localizing subcategory of $\SH$.  Since we just proved that this kernel contains $\cat C_{p,n-1}$, it contains $e_{p,n-1} \in \Loc(\cat C_{p,n-1})$. Hence $[f_G \otimes e_{p,n-1}, \Sigma e_G \otimes f_{p,n}] = 0$ as desired.
\end{proof}

\begin{Prop}
\label{prop:C_p-non-inclusion}%
Let $G=C_p$ and $n\geq 1$. Then $\cat P(1,p,n)\not\subseteq \cat P(C_p,p,n)$.
\end{Prop}

\begin{proof}
Example~\ref{ex:incl-0} deals with the $n=1$ case, so assume $n \ge 2$.  Let $\cat{SH}(G)_{p,n} = \triv(f_{p,n}) \otimes \SH(G)$ be the truncation of~$\SHG$ below $p$-chromatic level~$n$; see Ex.\,\ref{ex:fn}.  As explained in Remark~\ref{rem:geom}, the tt-spectrum $\Spc(\cat{SH}(G)_{p,n}^c)$ is homeomorphic to the open subset of $\Spc(\SHGc)$ consisting of those $\cat P_G(H,p,m)$ -- same $p$ -- for all $H\in\{1,G\}$ and $1\le m \le n$. Moreover, the image in $\cat{SH}(G)_{p,n}$ of the $(n-1)$-th triangle $e_{p,n-1}\to \unit\to f_{p,n-1}\to \Sigma e_{p,n-1}$ is the idempotent triangle in $\cat{SH}(G)_{p,n}$ associated with the closed set $\{\cat P(1,p,n), \cat P(C_p,p,n)\}$; see Example~\ref{ex:fn}.  By Lemma~\ref{lem:splitting}, the left idempotent $e_{p,n-1}\otimes f_{p,n}$ of this triangle splits in~$\cat{SH}(G)_{p,n}$.  That is, we have $e_{p,n-1}\otimes f_{p,n} = (e_{p,n-1}\otimes f_{p,n} \otimes e_G) \oplus (e_{p,n-1} \otimes f_{p,n} \otimes f_G)$.  By applying $\phigeomb{G}$ and $\Res^G_1$ respectively we see that the two direct summands are non-zero; hence it is a non-trivial decomposition.  Invoking Lemma~\ref{lem:splitting_lemma}, we conclude that the closed subset $\{ \cat P(1,p,n), \cat P(C_p,p,n)\}$ of $\Spc(\cat{SH}(G)_{p,n}^c)$ consisting of precisely two points is disconnected.  The claim follows (see Remark~\ref{rem:geom} if necessary).
\end{proof}

\begin{Rem}
We thus have a complete description of the topology for $G=C_p$.  See~\eqref{eq:Spc(C_p)} for a picture of the spectrum. As a consequence we obtain a complete classification of the thick $\otimes$-ideals in $\SH(C_p)^c$. More on this topic in Section~\ref{se:classif}.
\end{Rem}

\smallskip
\section{The topology for general $G$ and the $\log_p$-Conjecture}
\label{se:general-case}%
\medskip

Next we consider the problem of determining the topology of~$\Spc(\SHGc)$ for an arbitrary finite group~$G$. As explained in Proposition~\ref{prop:reduce-to-p-groups}, we have reduced the problem to understanding the possible inclusions $\cat P_G(K,p,n) \subseteq \cat P_G(H,p,m)$ for $p$ a prime number and~$G$ a $p$-group.

\medskip
We obtain an immediate generalization of Proposition~\ref{prop:C_p-non-inclusion}:

\begin{Prop}
\label{prop:non-inclusion}%
Let $G$ be a finite group, $H,K \le G$ subgroups, $p$ a prime and $1 \le n < \infty$. Then $\cat P_G(K,p,n) \not\subseteq \cat P_G(H,p,n)$ -- for the \emph{same}~$n$ -- unless $H \sim_G K$ (in which case the two primes coincide).
\end{Prop}

\begin{proof}
	Suppose ab absurdo that $\cat P_G(K,p,n) \subseteq \cat P_G(H,p,n)$ for some $n \ge 1$ and $H,K \le G$ with $H \not\sim_G K$. Proposition~\ref{prop:reduce-to-p-groups} allows us to assume that $G$ is a $p$-group and that $H=G$. Our (absurd) situation is now: $\cat P_G(K,p,n) \subseteq \cat P_G(G,p,n)$, where $G$ is a $p$-group and $K < G$. Then there exists a maximal proper subgroup $N \lhd G$ containing~$K$. Since $G$ is a $p$-group, $N$ must have index $p$ and $G/N \simeq C_p$. Finally, apply Corollary~\ref{cor:induced-infl} to $\cat P_G(K,p,n) \subseteq \cat P_G(G,p,n)$ to obtain the absurd conclusion $\cat P_{C_p}(1,p,n) \subseteq \cat P_{C_p}(C_p,p,n)$, contradicting Proposition~\ref{prop:C_p-non-inclusion}.
\end{proof}

We can also generalize Proposition~\ref{prop:C_p-inclusion} to arbitrary finite groups:
\begin{Thm}
\label{thm:shift-by-s}%
Let $G$ be a finite group, $p$ a prime, and $K\le H \le G$ subgroups such that $K$ admits a $p$-subnormal tower to $H$ of length~$s$ (see Terminology~\ref{ter:subnormal}).
Then $\cat P_G(K,p,n) \subseteq \cat P_G(H,p,n-s)$
for all~$n>s$.
\end{Thm}

\begin{proof}
By Corollary~\ref{cor:induced-res}, we can assume that $H=G$. We then proceed by induction on~$s$. For the $s=1$ base case, Proposition~\ref{prop:induced-geom} applied to $N=K$ turns the inclusion $\cat P_{C_p}(1,p,n) \subseteq \cat P_{C_p}(C_p,p,n-1)$ of Prop.\,\ref{prop:C_p-inclusion} into $\cat P_G(K,p,n) \subseteq \cat P_G(G,p,n-1)$. For the general case, let $K \le K' \lenormal G$ with $[G:K']=p$ and $K \le K'$ a $p$-subnormal subgroup of index~$p^{s-1}$. By induction hypothesis, \mbox{$\cat P_{K'}(K,p,n) \subseteq \cat P_{K'}(K',p,n-s+1)$} for all $n>s-1$. Applying Corollary~\ref{cor:induced-res}, we get $\cat P_G(K,p,n) \subseteq \cat P_G(K',p,n-s+1)$. The base case $s=1$ gives the last bit: $\cat P_G(K',p,n-s+1)\subseteq \cat P_G(G,p,n-s)$.
\end{proof}

\begin{Cor}
\label{cor:shift-by-s}%
Let $G$ be a $p$-group and $K\le H \le G$ subgroups with \mbox{$[H:K]=p^s$}. Then $\cat P_{G}(K,p,n) \subseteq \cat P_{G}(H,p,n-s)$ for all~$n>s$.
\qed
\end{Cor}

Finally, Corollary~\ref{cor:C_p-incl-infty} (and its proof) passes to arbitrary finite groups:

\begin{Cor}
\label{cor:incl-infty}%
Let $G$ be a finite group, $p$ a prime, and $K\le H \le G$ subgroups such that $K$ admits a $p$-subnormal tower to~$H$. Then $\cat P_{G}(K,p,\infty) \subseteq \cat P_{G}(H,p,\infty)$.
\qed
\end{Cor}

\begin{Rem}
\label{rem:more-dress}%
We saw in Proposition~\ref{prop:comp-map} how $\Spc(\SHGc)$ provides a ``chromatic refinement'' of~$\Spec(A(G))$, in the same way that $\Spc(\SHc)$ refines~$\Spec(\bbZ)$. However, Theorem~\ref{thm:shift-by-s} goes beyond this chromatic refinement. Namely, the collision $\mathfrak{p}(H,p)=\mathfrak{p}(K,p)$ when $\Op(H)\sim_G\Op(K)$ for non-conjugate subgroups~$H$ and~$K$ (Thm.\,\ref{thm:dress}) does \emph{not} happen in~$\Spc(\SHGc)$. We have instead the inclusions of Theorem~\ref{thm:shift-by-s}, which under the comparison map $\rho_{\SHGc}:\Spc(\SHGc)\to \Spec(A(G))$ yield inclusions between maximal primes, \ie necessarily equalities $\mathfrak{p}(H,p)=\mathfrak{p}(K,p)$ in~$A(G)$. These are exactly the equalities observed by~Dress.
\end{Rem}

\begin{Rem}
By the above results, the question we have in the general case is similar to the one we had for $G=C_p$. Theorem~\ref{thm:shift-by-s} shows that we have a shift by~$s$ when $K$ admits a $p$-subnormal tower of length~$s$ to~$H$. The problem is whether this is optimal. Proposition~\ref{prop:non-inclusion} shows that we do not have inclusion $\cat P_G(K,p,n) \subseteq \cat P_G(H,p,n)$, but for $s>1$ there is still the possibility that we could have a more efficient inclusion than $\cat P_G(K,p,n) \subseteq \cat P_G(H,p,n-s)$.  Moreover, by Proposition~\ref{prop:reduce-to-p-groups} the question of this indeterminacy is equivalent to the question for a $p$-group.  To be clear, if $G$ is a $p$-group and $H \le G$ is a subgroup of index~$p^s$ then we have $\cat P_G(H,p,n) \subseteq \cat P_G(G,p,n-s)$ for all $n> s$ but the question is whether we could have $\cat P_G(H,p,n) \subseteq \cat P_G(G,p,n-s+1)$. We conjecture that this does not happen. In fact, this conjecture reduces to the $H=1$ case, which we state below:
\end{Rem}

\begin{Conj}[$\log_p$-Conjecture]
\label{conj:log-p}%
We say that a group $G$ of order~$p^r$ \emph{satisfies the $\log_p$-Conjecture} if, for all $n \ge r$, we have $\cat P_G(1,p,n) \not\subseteq \cat P_G(G,p,n-r+1)$.
\end{Conj}

\begin{Exa}
Proposition~\ref{prop:C_p-non-inclusion} says that $C_p$ satisfies the $\log_p$-Conjecture.
\end{Exa}

\begin{Rem}
\label{rem:SHGp}%
As we have seen in Remark~\ref{rem:yes-we-can-localize}, inclusions $\cat P(K,p,n)\subseteq\cat P(H,p,m)$ -- for the same $p$ -- can be tested $p$-locally. In other words, the $\log_p$-Conjecture is truly a conjecture about~$\SHGp$; see Corollary~\ref{cor:loc}.
\end{Rem}

In the next section we will connect this conjecture to Tate cohomology but for now we would like to demonstrate its consequences for the complete determination of the topology of~$\Spc(\SHGc)$, in conjunction with Proposition~\ref{prop:finite-union-irred}.
\begin{Thm}
\label{thm:log-p-H-G}%
Let $G$ be a $p$-group and let $H \le G$. Suppose that $G$ satisfies the $\log_p$-Conjecture. Then $\cat P(H,p,n) \subseteq \cat P(G,p,m)$ if and only if $n \ge m +\log_p[G:H]$.
\end{Thm}

\begin{proof}
One direction is Theorem~\ref{thm:shift-by-s}. Conversely suppose $\cat P(H,p,n) \subseteq \cat P(G,p,m)$ and let $r=\log_p|G|$ and $s=\log_p|H|$. By Theorem~\ref{thm:shift-by-s} again, we know that $\cat P(1,p,n+s)\subseteq \cat P(H,p,n)$. Combining these inclusions, we obtain $\cat P(1,p,n+s)\subseteq \cat P(G,p,m)$. By hypothesis, $G$ satisfies the $\log_p$-Conjecture, hence $n+s\ge m+r$.
\end{proof}

More generally, we get:

\begin{Thm}
\label{thm:complete-statement}%
Let $G$ be a finite group and $p$ a prime. Let $H,K \le G$ be arbitrary subgroups and $1 \le n,m \le \infty$. Suppose the $p$-group $H/\Op(H)$, a subquotient of~$G$, satisfies the $\log_p$-Conjecture~\ref{conj:log-p}. Then we have an inclusion of equivariant primes $\cat P(K,p,n) \subseteq \cat P(H,p,m)$ if and only if $K$ is $G$-conjugate to a $p$-subnormal subgroup of $H$ and $n \ge m + \log_p(|H|/|K|)$.
\end{Thm}

\begin{proof}
In view of Proposition~\ref{prop:reduce-to-p-groups}, we can reduce the question to the announced $p$-group~$H/\Op(H)$; note that $|H|/|K|$ does not change under this reduction. So we are down to the special case of $p$-groups established in Theorem~\ref{thm:log-p-H-G}.
\end{proof}

From the results of Section~\ref{se:C_p}, we can deduce a conjecture-free statement:

\begin{Thm}
\label{thm:square-free}%
Let $G$ be a finite group of square-free order. For any prime $p$ dividing $|G|$, any subgroups $H,K \le G$ with $H \not\sim_G K$ and any chromatic integers $1\le m,n\le\infty$, the following are equivalent:
\begin{enumerate}[label=\rm(\arabic*)]
 \item $\cat P(K,p,n) \subset \cat P(H,p,m)$.
 \item $n \ge m +1$ and $K$ is $G$-conjugate to a normal subgroup of index~$p$ in~$H$.
\end{enumerate}
\end{Thm}

\begin{proof}
We can apply the conclusion of Theorem~\ref{thm:complete-statement} unconditionally since the only $p$-subgroup which appears as a subquotient of~$G$ is~$C_p$, which satisfies the $\log_p$-Conjecture by Proposition~\ref{prop:C_p-non-inclusion}. Also because $|G|$ is square-free, we see that there is no $p$-subnormal tower beyond length~$s=1$. This gives the statement.
\end{proof}

\begin{Rem}
In the situation of Theorem~\ref{thm:square-free}, the prime $p$ cannot divide~$|K|$. So $K$ is $p$-perfect, $\Op(K)=K\sim_G\Op(H)$ and the latter has index~$p$ in~$H$.
\end{Rem}

\begin{Exa}
\label{ex:S_3}%
Consider $G=S_3$ and $p$ a prime. Let us draw a picture of the space~$\Spc(\SHGc)$ at the prime~$p$, that is, above~$\Spec(A(G)_{(p)})$, or equivalently let us draw the picture of~$\Spc(\SHGcp)$; see Rem.\,\ref{rem:SHGp}. This depends on the prime~$p$, with three cases to discuss: $p=2$, $p=3$, and $p$ not dividing~$|G|$.
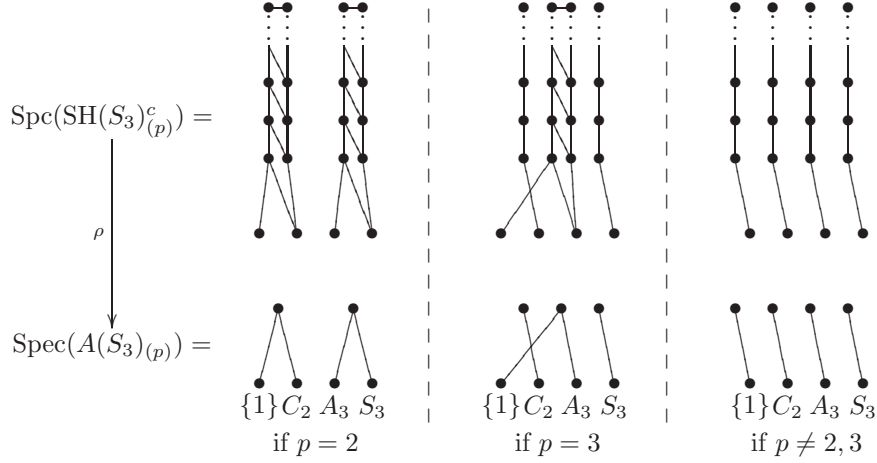
\[\xy
(-15,10)*{\Spc(\SH(S_3)^c_{(p)}) = };
(-15,-20)*{\Spec(A(S_3)_{(p)}) = };
{\ar@{->}_{\rho} (-15,7.5)*{};(-15,-17.5)*{}};
\endxy
\hspace{1em}
\xy
 (0,-5)*{\bullet};
  (1.25,5)*{\bullet};
  (1.25,10)*{\bullet};
  (1.25,15)*{\bullet};
  (1.25,23.5)*{\vdots};
  (1.25,25)*{\bullet};
{\ar@{-} (0,-5)*{};(1.25,5)*{}};
{\ar@{-} (1.25,5)*{};(1.25,10)*{}};
{\ar@{-} (1.25,10)*{};(1.25,15)*{}};
{\ar@{-} (1.25,15)*{};(1.25,20)*{}};
 (5,-5)*{\bullet};
  (3.75,5)*{\bullet};
  (3.75,10)*{\bullet};
  (3.75,15)*{\bullet};
  (3.75,23.5)*{\vdots};
  (3.75,25)*{\bullet};
{\ar@{-} (5,-5)*{};(3.75,5)*{}};
{\ar@{-} (3.75,5)*{};(3.75,10)*{}};
{\ar@{-} (3.75,10)*{};(3.75,15)*{}};
{\ar@{-} (3.75,15)*{};(3.75,20)*{}};
{\ar@{-} (5,-5)*{};(1.25,5)*{}};
{\ar@{-} (3.75,5)*{};(1.25,10)*{}};
{\ar@{-} (3.75,10)*{};(1.25,15)*{}};
{\ar@{-} (3.75,15)*{};(1.25,20)*{}};
{\ar@{-} (3.25,25)*{};(1.25,25)*{}};
 (10,-5)*{\bullet};
  (11.25,5)*{\bullet};
  (11.25,10)*{\bullet};
  (11.25,15)*{\bullet};
  (11.25,23.5)*{\vdots};
  (11.25,25)*{\bullet};
{\ar@{-} (10,-5)*{};(11.25,5)*{}};
{\ar@{-} (11.25,5)*{};(11.25,10)*{}};
{\ar@{-} (11.25,10)*{};(11.25,15)*{}};
{\ar@{-} (11.25,15)*{};(11.25,20)*{}};
 (15,-5)*{\bullet};
  (13.75,5)*{\bullet};
  (13.75,10)*{\bullet};
  (13.75,15)*{\bullet};
  (13.75,23.5)*{\vdots};
  (13.75,25)*{\bullet};
{\ar@{-} (15,-5)*{};(13.75,5)*{}};
{\ar@{-} (13.75,5)*{};(13.75,10)*{}};
{\ar@{-} (13.75,10)*{};(13.75,15)*{}};
{\ar@{-} (13.75,15)*{};(13.75,20)*{}};
{\ar@{-} (15,-5)*{};(11.25,5)*{}};
{\ar@{-} (13.75,5)*{};(11.25,10)*{}};
{\ar@{-} (13.75,10)*{};(11.25,15)*{}};
{\ar@{-} (13.75,15)*{};(11.25,20)*{}};
{\ar@{-} (13.25,25)*{};(11.25,25)*{}};
%
(0,-25)*{\bullet};
(5,-25)*{\bullet};
(10,-25)*{\bullet};
(15,-25)*{\bullet};
(2.5,-15)*{\bullet};
(12.5,-15)*{\bullet};
{\ar@{-} (0,-25)*{};(2.5,-15)*{}};
{\ar@{-} (5,-25)*{};(2.5,-15)*{}};
{\ar@{-} (10,-25)*{};(12.5,-15)*{}};
{\ar@{-} (15,-25)*{};(12.5,-15)*{}};
(0,-28)*{\{1\}};
(5,-28)*{C_2};
(10,-28)*{A_3};
(15,-28)*{S_3};
(7.5,-33)*{\textrm{if } p=2};
{\ar@{--} (22.5,-30)*{};(22.5,25)*{}}
\endxy
\qquad
\xy
 (2,-5)*{\bullet};
  (8.75,5)*{\bullet};
  (8.75,10)*{\bullet};
  (8.75,15)*{\bullet};
  (8.75,23.5)*{\vdots};
  (8.75,25)*{\bullet};
{\ar@{-} (2,-5)*{};(8.75,5)*{}};
{\ar@{-} (8.75,5)*{};(8.75,10)*{}};
{\ar@{-} (8.75,10)*{};(8.75,15)*{}};
{\ar@{-} (8.75,15)*{};(8.75,20)*{}};
 (7,-5)*{\bullet};
  (5,5)*{\bullet};
  (5,10)*{\bullet};
  (5,15)*{\bullet};
  (5,23.5)*{\vdots};
  (5,25)*{\bullet};
{\ar@{-} (7,-5)*{};(5,5)*{}};
{\ar@{-} (5,5)*{};(5,10)*{}};
{\ar@{-} (5,10)*{};(5,15)*{}};
{\ar@{-} (5,15)*{};(5,20)*{}};
 (12,-5)*{\bullet};
  (11.25,5)*{\bullet};
  (11.25,10)*{\bullet};
  (11.25,15)*{\bullet};
  (11.25,23.5)*{\vdots};
  (11.25,25)*{\bullet};
{\ar@{-} (12,-5)*{};(8.75,5)*{}};
{\ar@{-} (12,-5)*{};(11.25,5)*{}};
{\ar@{-} (11.25,5)*{};(11.25,10)*{}};
{\ar@{-} (11.25,10)*{};(11.25,15)*{}};
{\ar@{-} (11.25,15)*{};(11.25,20)*{}};
{\ar@{-} (11.25,5)*{};(8.75,10)*{}};
{\ar@{-} (11.25,10)*{};(8.75,15)*{}};
{\ar@{-} (11.25,15)*{};(8.75,20)*{}};
{\ar@{-} (11.25,25)*{};(8.75,25)*{}};
 (17,-5)*{\bullet};
  (15,5)*{\bullet};
  (15,10)*{\bullet};
  (15,15)*{\bullet};
  (15,23.5)*{\vdots};
  (15,25)*{\bullet};
{\ar@{-} (17,-5)*{};(15,5)*{}};
{\ar@{-} (15,5)*{};(15,10)*{}};
{\ar@{-} (15,10)*{};(15,15)*{}};
{\ar@{-} (15,15)*{};(15,20)*{}};
%
(2,-25)*{\bullet};
(7,-25)*{\bullet};
(12,-25)*{\bullet};
(17,-25)*{\bullet};
(5,-15)*{\bullet};
(10,-15)*{\bullet};
(15,-15)*{\bullet};
{\ar@{-} (2,-25)*{};(10,-15)*{}};
{\ar@{-} (7,-25)*{};(5,-15)*{}};
{\ar@{-} (12,-25)*{};(10,-15)*{}};
{\ar@{-} (17,-25)*{};(15,-15)*{}};
(2,-28)*{\{1\}};
(7,-28)*{C_2};
(12,-28)*{A_3};
(17,-28)*{S_3};
(9.5,-33)*{\textrm{if } p=3};
{\ar@{--} (24,-30)*{};(24,25)*{}}
\endxy
\qquad\
\xy
 (2,-5)*{\bullet};
  (0,5)*{\bullet};
  (0,10)*{\bullet};
  (0,15)*{\bullet};
  (0,23.5)*{\vdots};
  (0,25)*{\bullet};
{\ar@{-} (2,-5)*{};(0,5)*{}};
{\ar@{-} (0,5)*{};(0,10)*{}};
{\ar@{-} (0,10)*{};(0,15)*{}};
{\ar@{-} (0,15)*{};(0,20)*{}};
%
 (7,-5)*{\bullet};
  (5,5)*{\bullet};
  (5,10)*{\bullet};
  (5,15)*{\bullet};
  (5,23.5)*{\vdots};
  (5,25)*{\bullet};
{\ar@{-} (7,-5)*{};(5,5)*{}};
{\ar@{-} (5,5)*{};(5,10)*{}};
{\ar@{-} (5,10)*{};(5,15)*{}};
{\ar@{-} (5,15)*{};(5,20)*{}};
%
 (12,-5)*{\bullet};
  (10,5)*{\bullet};
  (10,10)*{\bullet};
  (10,15)*{\bullet};
  (10,23.5)*{\vdots};
  (10,25)*{\bullet};
{\ar@{-} (12,-5)*{};(10,5)*{}};
{\ar@{-} (10,5)*{};(10,10)*{}};
{\ar@{-} (10,10)*{};(10,15)*{}};
{\ar@{-} (10,15)*{};(10,20)*{}};
%
 (17,-5)*{\bullet};
  (15,5)*{\bullet};
  (15,10)*{\bullet};
  (15,15)*{\bullet};
  (15,23.5)*{\vdots};
  (15,25)*{\bullet};
{\ar@{-} (17,-5)*{};(15,5)*{}};
{\ar@{-} (15,5)*{};(15,10)*{}};
{\ar@{-} (15,10)*{};(15,15)*{}};
{\ar@{-} (15,15)*{};(15,20)*{}};
%
%
(2,-25)*{\bullet};
(7,-25)*{\bullet};
(12,-25)*{\bullet};
(17,-25)*{\bullet};
(0,-15)*{\bullet};
(5,-15)*{\bullet};
(10,-15)*{\bullet};
(15,-15)*{\bullet};
{\ar@{-} (2,-25)*{};(0,-15)*{}};
{\ar@{-} (7,-25)*{};(5,-15)*{}};
{\ar@{-} (12,-25)*{};(10,-15)*{}};
{\ar@{-} (17,-25)*{};(15,-15)*{}};
(2,-28)*{\{1\}};
(7,-28)*{C_2};
(12,-28)*{A_3};
(17,-28)*{S_3};
(9.5,-33)*{\textrm{if } p\neq 2,3}
\endxy
\]
At the very right, when $p\neq 2,3$, the spectrum of the Burnside ring localized at~$p$ is a disjoint union of four copies of~$\Spec(\bbZ_{(p)})=\{(0),(p)\}$, one for each conjugacy class of subgroup~$H\le G$. Accordingly, above each copy of~$\Spec(\bbZ_{(p)})$, we see a chromatic tower with infinitely many primes above~$(p)$ and only one above~$(0)$. At $p=2$, Dress's collision happens between the trivial subgroup and any of the conjugate cyclic subgroups of order~$2$, as well as between $A_3=C_3$ and~$G$. Each of those closed points now admits a chromatic refinement \emph{for each~$H$}, \ie here there are \emph{two} towers of primes in~$\SHGc$ projecting down to the closed point~$(p)$. Moreover, the two towers are not disconnected but the one for the smaller subgroup is in some sense ``in the closure" of the one for the bigger subgroup. A similar phenomenon happens at~$p=3$ for the subgroups~$\{1\}$ and~$A_3=C_3$. Note however that although (any) $H=C_2$ has index~$p=3$ in~$G$, it is not $p$-subnormal simply because it is not normal. We see here that Dress's result already distinguishes $\mathfrak{p}(H,p)=\mathfrak{p}(C_2,3)$ from $\mathfrak{p}(G,p)=\mathfrak{p}(S_3,3)$, which shows that the corresponding chromatic towers $\SET{\cat P(C_2,3,n)}{n\ge 1}$ and $\SET{\cat P(S_3,3,n)}{n\ge 1}$ are disconnected.
\end{Exa}

Let us give a re-interpretation of Conjecture~\ref{conj:log-p} in more geometric terms:
\begin{Prop}
\label{prop:equiv-geom}%
Let $G$ be a group of order $p^r$. The following are equivalent:
\begin{enumerate}[label=\rm(\Alph*)]
\item
\label{it:log-p}%
Conjecture~\ref{conj:log-p} holds: $\cat P(1,p,n)\not\subseteq \cat P(G,p,n-r+1)$, whenever~$r\le n<\infty$.
\smallbreak
\item
\label{it:Z_p,N}%
For every $N>r$, the following subset of~$\Spc(\SHGcp)$ is closed:
\[Z_{p,N}:=\SET{\cat P(H,p,\ell)\in\Spc(\SHGcp)}{H\le G,\ N-\log_p|H|\le\ell\le\infty}\,.\]
\smallbreak
\item
\label{it:supp(x)=Z}%
For every $N> r$, there exists $x \in \SHGcp$ with $\supp(x)=Z_{p,N}$.
\end{enumerate}
When these conditions hold, the closed set~$Z_{p,N}$ is irreducible and has~${\cat P(G,p,N-r)}$ as generic point.
\end{Prop}

\begin{proof}
\ref{it:log-p}$\then$\ref{it:Z_p,N}:
Since $N>r$, Corollary~\ref{cor:shift-by-s} implies that $Z_{p,N}$ is contained in the closed set~$\SET{\cat Q \in \Spc(\SHGp)}{\cat Q\subseteq \cat P(G,p,N-r)}=\adhpt{\cat P(G,p,N-r)}$.  If we prove the reverse inclusion, $\adhpt{\cat P(G,p,N-r)}\subseteq Z_{p,N}$, then $Z_{p,N}$ is closed irreducible with the announced generic point.  To prove this, suppose ab absurdo that there exists a point $\cat Q\subset \cat P(G,p,N-r)$ in~$\adhpt{\cat P(G,p,N-r)}$ with $\cat Q\notin Z_{p,N}$. By Theorem~\ref{thm:the-set} (and Rem.\,\ref{rem:SHGp}), we have $\cat Q=\cat P(H,p,\ell)$ for some $H\le G$ and some $1\le\ell\le\infty$. Let $s=\log_p|H|$, in which case $\cat Q\notin Z_{p,N}$ implies that $N-1\ge \ell+s$. By Proposition~\ref{prop:incl-same-subgroup} and Corollary~\ref{cor:shift-by-s}, we have $\cat P(1,p,N-1)\subset \cat P(1,p,\ell+s)\subset \cat P(H,p,\ell)\subset \cat P(G,p,N-r)$, which contradicts~\ref{it:log-p} for $n=N-1\ge r$.

\ref{it:Z_p,N}$\then$\ref{it:supp(x)=Z}:
In $\Spc(\SHGcp)$, the closed subset~$Z_{p,N}$ has finite (hence quasi-compact) complement.  (The picture \eqref{eq:Z_p,N} below may be a useful guide.) It follows from general tt-geometry that any closed subset with quasi-compact complement is the support of some object: see~\cite[Prop.\,2.14]{Balmer05a}.

\ref{it:supp(x)=Z}$\then$\ref{it:log-p}:
Let $n\ge r$ and let $x\in\SHGcp$ be such that $\supp(x)=Z_{p,n+1}$, that is, use~\ref{it:supp(x)=Z} for~$N:=n+1>r$. Then $\cat P(G,p,n-r+1)\in Z_{p,n+1}=\supp(x)$ means $x\notin\cat P(G,p,n-r+1)$, whereas $\cat P(1,p,n)\notin Z_{p,n+1}=\supp(x)$ means $x\in \cat P(1,p,n)$. Together, this gives $\cat P(1,p,n)\not\subseteq\cat P(G,p,n-r+1)$ as desired.
\end{proof}

\begin{Rem}
One can almost make the same statement as above with $\SHGc$ instead of the $p$-local version~$\SHGcp$. However, there is some ``fringe effect" when $N$ is small, typically if $\cat P(G,p,1)$ is allowed in~$Z_{p,N}$; in that case, $Z_{p,N}$ is not closed but for the rather disjoint issue that $\cat P(G,p,1)=\cat P(G,q,1)$ for $q\neq p$ and thus contains in its closure many other primes $\cat P(H,q,n)$ for~$q\neq p$.
\end{Rem}

\begin{Rem}
As in Remark~\ref{rem:geom}, it can be useful to draw a picture of $\Spc(\SHGcp)$. Let us assume that $G$ is a $p$-group for simplicity, say $|G|=p^r$. So subgroups $H\le G$ can be filtered by the integer~$s=\log_p(|H|)$, which runs between 0 and~$r$. We read $s$ horizontally and the chromatic degree~$m$ of~$\cat P(H,p,m)$ vertically:
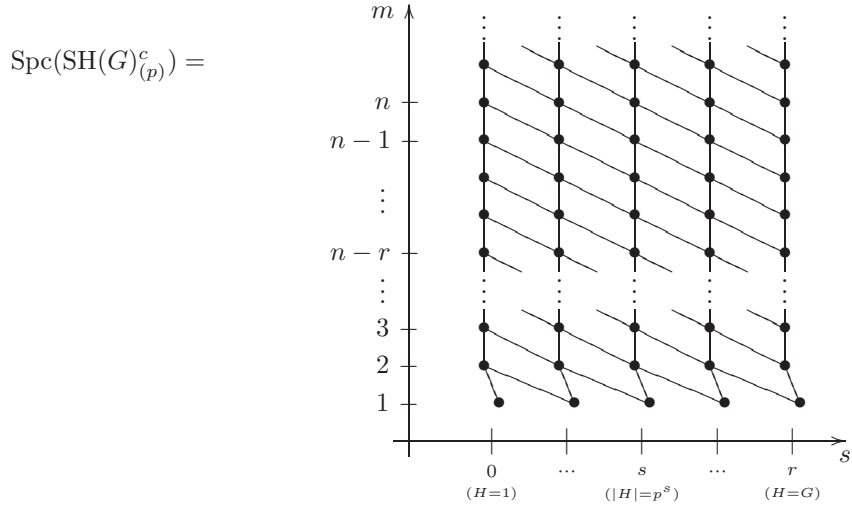
\[
\xy
{\ar@{->} (-40,-12)*{};(-40,48)*{}};
 (-43.5,47)*{m};
 (-41.5,35)*{n\ -};
 (-44.5,30)*{n-1\ -};
 (-43.5,23)*{\vdots};
 (-44.5,15)*{n-r\ -};
 (-43.5,11)*{\vdots};
 (-41.5,5)*{3\ -};
 (-41.5,0)*{2\ -};
 (-41.5,-5)*{1\ -};
{\ar@{->} (-42,-10)*{};(18,-10)*{}};
 (18,-12)*{s};
 (11,-10)*{|};
 (1,-10)*{|};
 (-9,-10)*{|};
 (-19,-10)*{|};
 (-29,-10)*{|};
 (11,-14)*{\scriptstyle r};
 (1,-14)*{\scriptstyle\cdots};
 (-9,-14)*{\scriptstyle s};
 (-19,-14)*{\scriptstyle\cdots};
 (-29,-14)*{\scriptstyle 0};
 (11,-17)*{\scriptscriptstyle (H=G)};
 (1,-17)*{};
 (-9,-17)*{\scriptscriptstyle (|H|=p^s)};
 (-19,-17)*{};
 (-29,-17)*{\scriptscriptstyle (H=1)};
%
%
  (-28,-5)*{\bullet};
  (-30,0)*{\bullet};
  (-30,5)*{\bullet};
  (-30,11)*{\vdots};
  (-30,15)*{\bullet};
  (-30,20)*{\bullet};
  (-30,25)*{\bullet};
  (-30,30)*{\bullet};
  (-30,35)*{\bullet};
  (-30,40)*{\bullet};
  (-30,46)*{\vdots};
{\ar@{-} (-28,-5)*{};(-30,0)*{}};
{\ar@{-} (-30,0)*{};(-30,5)*{}};
{\ar@{-} (-30,5)*{};(-30,7.5)*{}};
{\ar@{-} (-30,12.5)*{};(-30,15)*{}};
{\ar@{-} (-30,15)*{};(-30,20)*{}};
{\ar@{-} (-30,20)*{};(-30,25)*{}};
{\ar@{-} (-30,25)*{};(-30,30)*{}};
{\ar@{-} (-30,30)*{};(-30,35)*{}};
{\ar@{-} (-30,35)*{};(-30,40)*{}};
{\ar@{-} (-30,40)*{};(-30,43)*{}};
  (-18,-5)*{\bullet};
  (-20,0)*{\bullet};
  (-20,5)*{\bullet};
  (-20,11)*{\vdots};
  (-20,15)*{\bullet};
  (-20,20)*{\bullet};
  (-20,25)*{\bullet};
  (-20,30)*{\bullet};
  (-20,35)*{\bullet};
  (-20,40)*{\bullet};
  (-20,46)*{\vdots};
{\ar@{-} (-18,-5)*{};(-20,0)*{}};
{\ar@{-} (-20,0)*{};(-20,5)*{}};
{\ar@{-} (-20,5)*{};(-20,7.5)*{}};
{\ar@{-} (-20,12.5)*{};(-20,15)*{}};
{\ar@{-} (-20,15)*{};(-20,20)*{}};
{\ar@{-} (-20,20)*{};(-20,25)*{}};
{\ar@{-} (-20,25)*{};(-20,30)*{}};
{\ar@{-} (-20,30)*{};(-20,35)*{}};
{\ar@{-} (-20,35)*{};(-20,40)*{}};
{\ar@{-} (-20,40)*{};(-20,43)*{}};
  (-8,-5)*{\bullet};
  (-10,0)*{\bullet};
  (-10,5)*{\bullet};
  (-10,11)*{\vdots};
  (-10,15)*{\bullet};
  (-10,20)*{\bullet};
  (-10,25)*{\bullet};
  (-10,30)*{\bullet};
  (-10,35)*{\bullet};
  (-10,40)*{\bullet};
  (-10,46)*{\vdots};
{\ar@{-} (-8,-5)*{};(-10,0)*{}};
{\ar@{-} (-10,0)*{};(-10,5)*{}};
{\ar@{-} (-10,5)*{};(-10,7.5)*{}};
{\ar@{-} (-10,12.5)*{};(-10,15)*{}};
{\ar@{-} (-10,15)*{};(-10,20)*{}};
{\ar@{-} (-10,20)*{};(-10,25)*{}};
{\ar@{-} (-10,25)*{};(-10,30)*{}};
{\ar@{-} (-10,30)*{};(-10,35)*{}};
{\ar@{-} (-10,35)*{};(-10,40)*{}};
{\ar@{-} (-10,40)*{};(-10,43)*{}};
  (2,-5)*{\bullet};
  (0,0)*{\bullet};
  (0,5)*{\bullet};
  (0,11)*{\vdots};
  (0,15)*{\bullet};
  (0,20)*{\bullet};
  (0,25)*{\bullet};
  (0,30)*{\bullet};
  (0,35)*{\bullet};
  (0,40)*{\bullet};
  (0,46)*{\vdots};
{\ar@{-} (2,-5)*{};(0,0)*{}};
{\ar@{-} (0,0)*{};(0,5)*{}};
{\ar@{-} (0,5)*{};(0,7.5)*{}};
{\ar@{-} (0,12.5)*{};(0,15)*{}};
{\ar@{-} (0,15)*{};(0,20)*{}};
{\ar@{-} (0,20)*{};(0,25)*{}};
{\ar@{-} (0,25)*{};(0,30)*{}};
{\ar@{-} (0,30)*{};(0,35)*{}};
{\ar@{-} (0,35)*{};(0,40)*{}};
{\ar@{-} (0,40)*{};(0,43)*{}};
  (12,-5)*{\bullet};
  (10,0)*{\bullet};
  (10,5)*{\bullet};
  (10,11)*{\vdots};
  (10,15)*{\bullet};
  (10,20)*{\bullet};
  (10,25)*{\bullet};
  (10,30)*{\bullet};
  (10,35)*{\bullet};
  (10,40)*{\bullet};
  (10,46)*{\vdots};
{\ar@{-} (12,-5)*{};(10,0)*{}};
{\ar@{-} (10,0)*{};(10,5)*{}};
{\ar@{-} (10,5)*{};(10,7.5)*{}};
{\ar@{-} (10,12.5)*{};(10,15)*{}};
{\ar@{-} (10,15)*{};(10,20)*{}};
{\ar@{-} (10,20)*{};(10,25)*{}};
{\ar@{-} (10,25)*{};(10,30)*{}};
{\ar@{-} (10,30)*{};(10,35)*{}};
{\ar@{-} (10,35)*{};(10,40)*{}};
{\ar@{-} (10,40)*{};(10,43)*{}};
%
%
{\ar@{-} (-18,-5)*{};(-30,0)*{}};
{\ar@{-} (-20,0)*{};(-30,5)*{}};
{\ar@{-} (-20,5)*{};(-25,7.5)*{}};
{\ar@{-} (-25,12.5)*{};(-30,15)*{}};
{\ar@{-} (-20,15)*{};(-30,20)*{}};
{\ar@{-} (-20,20)*{};(-30,25)*{}};
{\ar@{-} (-20,25)*{};(-30,30)*{}};
{\ar@{-} (-20,30)*{};(-30,35)*{}};
{\ar@{-} (-20,35)*{};(-30,40)*{}};
{\ar@{-} (-20,40)*{};(-25,42.5)*{}};
{\ar@{-} (-8,-5)*{};(-20,0)*{}};
{\ar@{-} (-10,0)*{};(-20,5)*{}};
{\ar@{-} (-10,5)*{};(-15,7.5)*{}};
{\ar@{-} (-15,12.5)*{};(-20,15)*{}};
{\ar@{-} (-10,15)*{};(-20,20)*{}};
{\ar@{-} (-10,20)*{};(-20,25)*{}};
{\ar@{-} (-10,25)*{};(-20,30)*{}};
{\ar@{-} (-10,30)*{};(-20,35)*{}};
{\ar@{-} (-10,35)*{};(-20,40)*{}};
{\ar@{-} (-10,40)*{};(-15,42.5)*{}};
{\ar@{-} (2,-5)*{};(-10,0)*{}};
{\ar@{-} (0,0)*{};(-10,5)*{}};
{\ar@{-} (0,5)*{};(-5,7.5)*{}};
{\ar@{-} (-5,12.5)*{};(-10,15)*{}};
{\ar@{-} (0,15)*{};(-10,20)*{}};
{\ar@{-} (0,20)*{};(-10,25)*{}};
{\ar@{-} (0,25)*{};(-10,30)*{}};
{\ar@{-} (0,30)*{};(-10,35)*{}};
{\ar@{-} (0,35)*{};(-10,40)*{}};
{\ar@{-} (0,40)*{};(-5,42.5)*{}};
{\ar@{-} (12,-5)*{};(0,0)*{}};
{\ar@{-} (10,0)*{};(0,5)*{}};
{\ar@{-} (10,5)*{};(5,7.5)*{}};
{\ar@{-} (5,12.5)*{};(0,15)*{}};
{\ar@{-} (10,15)*{};(0,20)*{}};
{\ar@{-} (10,20)*{};(0,25)*{}};
{\ar@{-} (10,25)*{};(0,30)*{}};
{\ar@{-} (10,30)*{};(0,35)*{}};
{\ar@{-} (10,35)*{};(0,40)*{}};
{\ar@{-} (10,40)*{};(5,42.5)*{}};
(-80,40)*{\Spc(\SHGcp)=};
\endxy
\]
We warn the reader that this picture is slightly misleading since of course $G$ can have non-conjugate subgroups of a fixed order~$p^s$. So, the intermediate columns (for $0<s<r$) actually have several layers. This picture is nonetheless useful to build some intuition.  (The picky reader can restrict to $G=C_{p^r}$ if necessary.) The inclusions depicted above are the only ones if Conjecture~\ref{conj:log-p} holds true.

Truncating below $p$-chromatic level~$n$, \ie localizing via $\SHG\onto \cat{SH}(G)_{p,n}=\triv(f_{p,n})\otimes\SHG$, the spectrum would be as above but truncated to only include those $m\le n$. In particular, it is a finite space (but not a discrete one).

Also, the subset $Z_{p,N}$ of Proposition~\ref{prop:equiv-geom} looks as follows (the $\bullet$, not the~$\circ$):
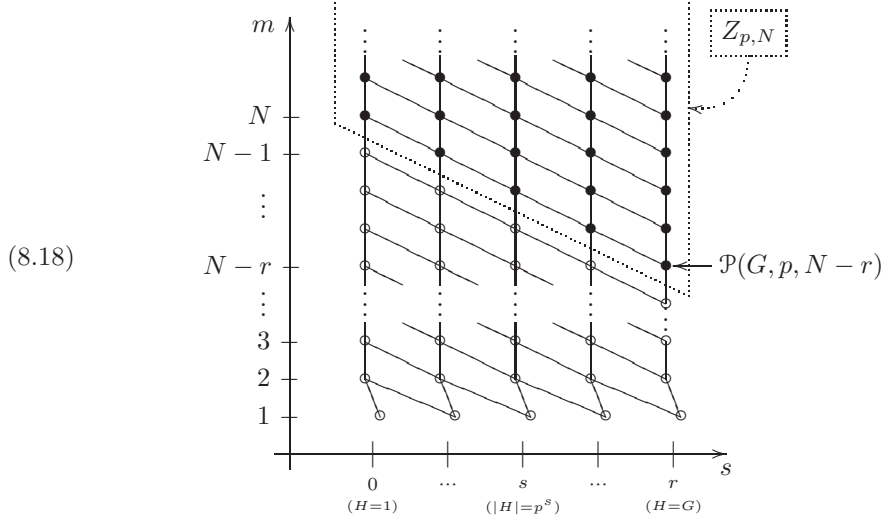
\begin{equation}
\label{eq:Z_p,N}%
\qquad\vcenter{\xy
{\ar@{->} (-40,-12)*{};(-40,48)*{}};
 (-43.5,47)*{m};
 (-42,35)*{N\ -};
 (-45,30)*{N-1\ -};
 (-43.5,24)*{\vdots};
 (-45,15)*{N-r\ -};
 (-43.5,11)*{\vdots};
 (-41.5,5)*{3\ -};
 (-41.5,0)*{2\ -};
 (-41.5,-5)*{1\ -};
{\ar@{->} (-42,-10)*{};(18,-10)*{}};
 (18,-12)*{s};
 (11,-10)*{|};
 (1,-10)*{|};
 (-9,-10)*{|};
 (-19,-10)*{|};
 (-29,-10)*{|};
 (11,-14)*{\scriptstyle r};
 (1,-14)*{\scriptstyle\cdots};
 (-9,-14)*{\scriptstyle s};
 (-19,-14)*{\scriptstyle\cdots};
 (-29,-14)*{\scriptstyle 0};
 (11,-17)*{\scriptscriptstyle (H=G)};
 (1,-17)*{};
 (-9,-17)*{\scriptscriptstyle (|H|=p^s)};
 (-19,-17)*{};
 (-29,-17)*{\scriptscriptstyle (H=1)};
%
%
  (-28,-5)*{\circ};
  (-30,0)*{\circ};
  (-30,5)*{\circ};
  (-30,11)*{\vdots};
  (-30,15)*{\circ};
  (-30,20)*{\circ};
  (-30,25)*{\circ};
  (-30,30)*{\circ};
  (-30,35)*{\bullet};
  (-30,40)*{\bullet};
  (-30,46)*{\vdots};
{\ar@{-} (-28,-5)*{};(-30,0)*{}};
{\ar@{-} (-30,0)*{};(-30,5)*{}};
{\ar@{-} (-30,5)*{};(-30,7.5)*{}};
{\ar@{-} (-30,12.5)*{};(-30,15)*{}};
{\ar@{-} (-30,15)*{};(-30,20)*{}};
{\ar@{-} (-30,20)*{};(-30,25)*{}};
{\ar@{-} (-30,25)*{};(-30,30)*{}};
{\ar@{-} (-30,30)*{};(-30,35)*{}};
{\ar@{-} (-30,35)*{};(-30,40)*{}};
{\ar@{-} (-30,40)*{};(-30,43)*{}};
  (-18,-5)*{\circ};
  (-20,0)*{\circ};
  (-20,5)*{\circ};
  (-20,11)*{\vdots};
  (-20,15)*{\circ};
  (-20,20)*{\circ};
  (-20,25)*{\circ};
  (-20,30)*{\bullet};
  (-20,35)*{\bullet};
  (-20,40)*{\bullet};
  (-20,46)*{\vdots};
{\ar@{-} (-18,-5)*{};(-20,0)*{}};
{\ar@{-} (-20,0)*{};(-20,5)*{}};
{\ar@{-} (-20,5)*{};(-20,7.5)*{}};
{\ar@{-} (-20,12.5)*{};(-20,15)*{}};
{\ar@{-} (-20,15)*{};(-20,20)*{}};
{\ar@{-} (-20,20)*{};(-20,25)*{}};
{\ar@{-} (-20,25)*{};(-20,30)*{}};
{\ar@{-} (-20,30)*{};(-20,35)*{}};
{\ar@{-} (-20,35)*{};(-20,40)*{}};
{\ar@{-} (-20,40)*{};(-20,43)*{}};
  (-8,-5)*{\circ};
  (-10,0)*{\circ};
  (-10,5)*{\circ};
  (-10,11)*{\vdots};
  (-10,15)*{\circ};
  (-10,20)*{\circ};
  (-10,25)*{\bullet};
  (-10,30)*{\bullet};
  (-10,35)*{\bullet};
  (-10,40)*{\bullet};
  (-10,46)*{\vdots};
{\ar@{-} (-8,-5)*{};(-10,0)*{}};
{\ar@{-} (-10,0)*{};(-10,5)*{}};
{\ar@{-} (-10,5)*{};(-10,7.5)*{}};
{\ar@{-} (-10,12.5)*{};(-10,15)*{}};
{\ar@{-} (-10,15)*{};(-10,20)*{}};
{\ar@{-} (-10,20)*{};(-10,25)*{}};
{\ar@{-} (-10,25)*{};(-10,30)*{}};
{\ar@{-} (-10,30)*{};(-10,35)*{}};
{\ar@{-} (-10,35)*{};(-10,40)*{}};
{\ar@{-} (-10,40)*{};(-10,43)*{}};
  (2,-5)*{\circ};
  (0,0)*{\circ};
  (0,5)*{\circ};
  (0,11)*{\vdots};
  (0,15)*{\circ};
  (0,20)*{\bullet};
  (0,25)*{\bullet};
  (0,30)*{\bullet};
  (0,35)*{\bullet};
  (0,40)*{\bullet};
  (0,46)*{\vdots};
{\ar@{-} (2,-5)*{};(0,0)*{}};
{\ar@{-} (0,0)*{};(0,5)*{}};
{\ar@{-} (0,5)*{};(0,7.5)*{}};
{\ar@{-} (0,12.5)*{};(0,15)*{}};
{\ar@{-} (0,15)*{};(0,20)*{}};
{\ar@{-} (0,20)*{};(0,25)*{}};
{\ar@{-} (0,25)*{};(0,30)*{}};
{\ar@{-} (0,30)*{};(0,35)*{}};
{\ar@{-} (0,35)*{};(0,40)*{}};
{\ar@{-} (0,40)*{};(0,43)*{}};
  (12,-5)*{\circ};
  (10,0)*{\circ};
  (10,5)*{\circ};
  (10,10)*{\circ};
  (10,15)*{\bullet};
  (10,20)*{\bullet};
  (10,25)*{\bullet};
  (10,30)*{\bullet};
  (10,35)*{\bullet};
  (10,40)*{\bullet};
  (10,46)*{\vdots};
{\ar@{-} (12,-5)*{};(10,0)*{}};
{\ar@{-} (10,0)*{};(10,5)*{}};
 (10,9)*{\vdots};
{\ar@{-} (10,10)*{};(10,15)*{}};
{\ar@{-} (10,15)*{};(10,20)*{}};
{\ar@{-} (10,20)*{};(10,25)*{}};
{\ar@{-} (10,25)*{};(10,30)*{}};
{\ar@{-} (10,30)*{};(10,35)*{}};
{\ar@{-} (10,35)*{};(10,40)*{}};
{\ar@{-} (10,40)*{};(10,43)*{}};
%
%
{\ar@{-} (-18,-5)*{};(-30,0)*{}};
{\ar@{-} (-20,0)*{};(-30,5)*{}};
{\ar@{-} (-20,5)*{};(-25,7.5)*{}};
{\ar@{-} (-25,12.5)*{};(-30,15)*{}};
{\ar@{-} (-20,15)*{};(-30,20)*{}};
{\ar@{-} (-20,20)*{};(-30,25)*{}};
{\ar@{-} (-20,25)*{};(-30,30)*{}};
{\ar@{-} (-20,30)*{};(-30,35)*{}};
{\ar@{-} (-20,35)*{};(-30,40)*{}};
{\ar@{-} (-20,40)*{};(-25,42.5)*{}};
{\ar@{-} (-8,-5)*{};(-20,0)*{}};
{\ar@{-} (-10,0)*{};(-20,5)*{}};
{\ar@{-} (-10,5)*{};(-15,7.5)*{}};
{\ar@{-} (-15,12.5)*{};(-20,15)*{}};
{\ar@{-} (-10,15)*{};(-20,20)*{}};
{\ar@{-} (-10,20)*{};(-20,25)*{}};
{\ar@{-} (-10,25)*{};(-20,30)*{}};
{\ar@{-} (-10,30)*{};(-20,35)*{}};
{\ar@{-} (-10,35)*{};(-20,40)*{}};
{\ar@{-} (-10,40)*{};(-15,42.5)*{}};
{\ar@{-} (2,-5)*{};(-10,0)*{}};
{\ar@{-} (0,0)*{};(-10,5)*{}};
{\ar@{-} (0,5)*{};(-5,7.5)*{}};
{\ar@{-} (-5,12.5)*{};(-10,15)*{}};
{\ar@{-} (0,15)*{};(-10,20)*{}};
{\ar@{-} (0,20)*{};(-10,25)*{}};
{\ar@{-} (0,25)*{};(-10,30)*{}};
{\ar@{-} (0,30)*{};(-10,35)*{}};
{\ar@{-} (0,35)*{};(-10,40)*{}};
{\ar@{-} (0,40)*{};(-5,42.5)*{}};
{\ar@{-} (12,-5)*{};(0,0)*{}};
{\ar@{-} (10,0)*{};(0,5)*{}};
{\ar@{-} (10,5)*{};(5,7.5)*{}};
{\ar@{-} (10,10)*{};(0,15)*{}};
{\ar@{-} (10,15)*{};(0,20)*{}};
{\ar@{-} (10,20)*{};(0,25)*{}};
{\ar@{-} (10,25)*{};(0,30)*{}};
{\ar@{-} (10,30)*{};(0,35)*{}};
{\ar@{-} (10,35)*{};(0,40)*{}};
{\ar@{-} (10,40)*{};(5,42.5)*{}};
%
%
{\ar@{..} (-34,50)*{};(-34,34)*{}};
{\ar@{..} (-34,34)*{};(13,11)*{}};
{\ar@{..} (13,11)*{};(13,50)*{}};
(21,46)*{Z_{p,N}};
{\ar@/^.8em/@{..>} (21,43)*{};(13,36)*{}};
{\ar@{..} (16,43)*{};(16,49)*{}};
{\ar@{..} (16,49)*{};(26,49)*{}};
{\ar@{..} (26,49)*{};(26,43)*{}};
{\ar@{..} (26,43)*{};(16,43)*{}};
(28,15)*{\cat P(G,p,N-r)};
{\ar@{->} (16,15)*{};(11,15)*{}};
\endxy}
\end{equation}
\end{Rem}

Let us conclude this section with a summary of what we know about the topology:

\begin{Cor}
\label{cor:summary}%
Let $G$ be a finite group. Every closed subset of $\Spc(\SHGc)$ is a finite union of irreducible closed subsets, and every irreducible closed set is of the form
$\adhpt{\cat P}=$ \mbox{$\SET{\cat Q\in \Spc(\SHGc)}{\cat Q\subseteq\cat P}$} for a unique prime $\cat P \in \Spc(\SHGc)$.

Let us consider two arbitrary primes in~$\SHGc$ (see Theorems~\ref{thm:the-set} and~\ref{thm:prime-uniqueness})
\[
\cat Q=\cat P(K,q,n)
\qquadtext{and}
\cat P = \cat P(H,p,m)
\]
for subgroups $H,K\le G$, primes $p,q$ and chromatic integers $1\le m,n\le \infty$. We fix $\cat P$ (\ie we fix $H$, $p$ and~$m$) and we discuss when the inclusion $\cat Q\subseteq\cat P$ holds, \ie when $\cat Q\in\adhpt{\cat P}$, as function of~$K$, $q$ and~$n$.
\begin{enumerate}[label=\rm(\alph*)]
\smallbreak
\item
\label{it:K-not-sub-H}%
If $K$ is not $G$-conjugate to a $q$-subnormal subgroup of~$H$ then $\cat Q\notin\adhpt{\cat P}$.
\smallbreak
\item
\label{it:p-neq-q}%
If $m>1$ and $p\neq q$ then $\cat Q\notin\adhpt{\cat P}$. (If $m=1$ then $\cat P=\cat P(H,q,1)$ anyway.)
\end{enumerate}
\smallbreak
\noindent
These statements were independent of~$n$. In the remaining cases, the inclusion $\cat Q\subseteq \cat P$ does always happen, for all~$n$ in a connected interval. More precisely:
\begin{enumerate}[label=\rm(\alph*),resume*]
\smallbreak
\item
\label{it:inclusion}%
Suppose that $K$ is $G$-conjugate to a $q$-subnormal subgroup of~$H$ and suppose that either $m=1$ or $p=q$. Then $\cat P(K,q,\infty)\in \adhpt{\cat P}$ and
there is a well-defined chromatic integer~$n_K := \min\SET{1\le \ell\le\infty}{\cat P(K,q,\ell) \in \adhpt{\cat P}}$ such that
\[
m \le n_K \le m+\log_q(|H|/|K|)
\]
(for $m=\infty$, this simply means $n_K=m=\infty$) with the property that
\[\cat Q\in\adhpt{\cat P}
\quadtext{if and only if}
n_K \le n \le \infty.
\]
If $n_K$ equals $m$ and is finite, then ${K \sim_G H}$. Finally, if Conjecture~\ref{conj:log-p} holds for the $q$-group $H/O^q(H)$ -- for instance if $G$ has square-free order -- then
\[
n_K = m+\log_q(|H|/|K|).
\]
\end{enumerate}
\end{Cor}

\begin{proof}
Parts~\ref{it:K-not-sub-H} and~\ref{it:p-neq-q} follow from Corollary~\ref{cor:incl-chrom}, Remark~\ref{rem:incl-at-p} and Proposition~\ref{prop:burnside-info}. Part~\ref{it:inclusion} follows from Theorem~\ref{thm:shift-by-s}, Proposition~\ref{prop:non-inclusion} and Theorem~\ref{thm:complete-statement}.
\end{proof}

\begin{Rem}
In~\ref{it:inclusion}, if $m=1$, one can replace $p$ by~$q$ without changing the problem since $\cat P$ is $\cat P(H,q,1)$ as well. So one can read $p$ instead of~$q$ everywhere in~(c).
\end{Rem}

\smallskip
\section{Tate re-interpretation of the $\log_p$-Conjecture}
\label{se:log-p-conjecture}
\medskip

In this section we provide equivalent formulations of the $\log_p$-Conjecture~\ref{conj:log-p}, in terms of blue shift phenomena for Tate cohomology. Let $G$ be a $p$-group of order~$p^r$ and let us denote by~$\cat F_{\le p^s}$ the family of subgroups~$\SET{H\le G}{|H|\le p^s}$ of order at most~$p^{s}$; see Ex.\,\ref{ex:Fp^s}.  We denote the associated idempotent triangle by $e_{\le p^s}\to \unit\to f_{\le p^s}\to \cdot$ and the associated Tate functor by
\[
t_{\le p^s}=[f_{\le p^s},\Sigma e_{\le p^s}\otimes-]\cong f_{\le p^s}\otimes [e_{\le p^s},-]\,:\ \SHG\to \SHG
\]
as in Section~\ref{se:e,f,Tate}.
Finally, let us agree that $\cat C_{p,0} := \SH^c_{(p)}$ denotes the whole category of compact $p$-local spectra, and recall that $x$ is said to be of type-$n$ if $x \in \cat C_{p,n} \setminus \cat C_{p,n+1}$.

\begin{Thm}
\label{thm:equiv-Tate}%
Let $G$ be a group of order $p^r$. The following are equivalent:
\begin{enumerate}[label=\rm(\Alph*)]
\item
\label{it:log-p-Tate}%
Conjecture~\ref{conj:log-p} holds\,: $\cat P(1,p,n)\not\subseteq \cat P(G,p,n-r+1)$ whenever~$r\le n<\infty$.
\smallbreak
\item
\label{it:Tate-G-for-fn}%
For all integers $s,\,t$ and $n$ such that $0 \le s < t \le r$ and $n \ge r-s$, and for every subgroup $H\le G$ with $|H|=p^t$, we have
\[
\phigeomb{H}(t_{\le p^s}(\triv(f_{p,n}))) \otimes \cat C_{p,n-(t-s)} = 0.
\]
In a slogan: the geometric fixed points of the Tate functor $\phigeomb{H}(t_{\le p^s}(-))$ lowers chromatic degree by the ($p$-exponent) ``distance" $t-s$ from the subgroup~$H$ to the family~$\cat F_{\le p^s}$.
\end{enumerate}
\end{Thm}

\begin{proof}
\ref{it:log-p-Tate}$\then$\ref{it:Tate-G-for-fn}:
Let $N:=s+n+1>r$. Then by Proposition~\ref{prop:equiv-geom}, there exists an object~$x\in\SHGcp$ whose support is $Z_{p,N}=\SET{\cat P(K,p,\ell)\in\Spc(\SHGcp)}{K\le G,\ N-\log_p|K|\le\ell\le\infty}$.  Localizing further to~$\cat{SH}(G)_{p,n}=f_{p,n}\otimes \SHG$ (see Prop.~\ref{prop:loc}), we obtain an object $x':=x\otimes f_{p,n} \in \cat{SH}(G)_{p,n}^c$ whose support is
$$
\supp(x')=\SET{\cat P(K,p,\ell)}{N-\log_p|K|\le \ell\le n}.
$$
Note that any point~$\cat P(K,p,\ell)$ in $\supp(x')$ must satisfy~$\log_p|K|>s$ (otherwise $n+1=N-s \le \ell\le n$). It follows that for every ${y\in \cat{SH}(G)_{p,n}^c}$ in~$\thick_\otimes\{G/K_+ \mid K\in\cat F_{\le p^s}\}$ we have $y\otimes x'=0$ and $[y,x']=0$ for they have disjoint support in~$\cat{SH}(G)_{p,n}^c$.  Hence the same vanishing is true for the corresponding left idempotent~$e_{\le p^s}$ in place of~$y$. Repatriated into~$\SHGp$ this reads $e_{\le p^s}\otimes x\otimes f_{p,n}=0$ and $[e_{\le p^s}, x\otimes f_{p,n}]=0$. From either of them we deduce
\begin{equation}
\label{eq:tate-vanishing}%
t_{\le p^s}(f_{p,n}\otimes x)=t_{\le p^s}(f_{p,n})\otimes x=0.
\end{equation}
Now for every subgroup~$H\le G$ with $|H|=p^t$, we have by the construction of~$x$ that $\cat P(H,p,n-(t-s))\notin Z_{p,N}=\supp(x)$ whereas $\cat P(H,p,n-(t-s)+1)\in Z_{p,N}=\supp(x)$.  Unfolding the definitions, this means that $\phigeomb{H}(x)$ belongs to~$\cat C_{p,n-(t-s)}$ but not to~$\cat C_{p,n-(t-s)+1}$.  In other words, $\phigeomb{H}(x)$ has type exactly~$n-(t-s)$ (in~$\SH^c_{(p)}$) and therefore $\cat C_{p,n-(t-s)}=\thick(\phigeomb{H}(x))$ in~$\SH^c_{(p)}$.  So, to prove~\ref{it:Tate-G-for-fn}, it suffices to prove the vanishing of
\[
\phigeomb{H}(t_{\le p^s}(f_{p,n}))\otimes\phigeomb{H}(x).
\]
But this follows by applying~$\phigeomb{H}$ to \eqref{eq:tate-vanishing}.

\ref{it:Tate-G-for-fn}$\then$\ref{it:log-p-Tate}:
Again via Proposition~\ref{prop:equiv-geom}, it suffices to prove that for every $N>r$, the irreducible closed subset $\adhpt{\cat P(G,p,N-r)}$ in~$\Spc(\SHGcp)$ is exactly
\[
Z_{p,N}=\SET{\cat P(K,p,\ell)}{K\le G,\ 1\le \ell\le \infty\textrm{ such that }\ell+\log_p|K|\ge N}.
\]
We already have one inclusion, $\adhpt{\cat P(G,p,N-r)}\supseteq Z_{p,N}$, by Corollary~\ref{cor:shift-by-s} and we now discuss the reverse inclusion. Since the complement of~$\adhpt{\cat P(G,p,N-r)}$ in~$\Spc(\SHGcp)$ is compact (in fact, even the complement of $Z_{p,N}$ is finite), there exists an object~$z\in\SHGcp$ whose support is exactly
\[
\supp(z)=\adhpt{\cat P(G,p,N-r)}.
\]
Suppose ab absurdo that there exists a ``bad point" $\cat P_0=\cat P(K_0,p,n)$ belonging to $\adhpt{\cat P(G,p,N-r)}$ but with $\cat P_0\notin Z_{p,N}$; the latter reads $n+\log_p|K_0|<N$. We can assume that $n$ is minimal among such ``bad points". This means that we can assume the following:
\begin{equation}
\label{eq:m-minimal}%
\textrm{for all }\ell<n,\textrm{ if }\ell+\log_p|H|< N\textrm{ then }\cat P(H,p,\ell)\notin\adhpt{\cat P(G,p,N-r)}.
\end{equation}
Let $s:=N-n-1$. Since $n<N$, we have $s\ge 0$ and we can consider the family~$\cat F_{\le p^s}$. We claim that the following Tate object vanishes in~$\SHGp$:
\[
t_{\le p^s}(f_{p,n}\otimes z).
\]
To see this, it suffices to prove that $\phigeomb{H}(t_{\le p^s}(f_{p,n})\otimes z)$ vanishes in~$\SHp$ for each~${H\le G}$. When $H\in\cat F_{\le p^s}$ this is automatic from Example~\ref{ex:res-Tate} and Remark~\ref{rem:triv-fams}.
Let us then take $H\le G$ with $|H|=p^t$ and $t>s$.  We are in the situation of hypothesis~\ref{it:Tate-G-for-fn} for $0\le s < t\le r$. (Note we indeed have $n \ge r-s$ because we started with $N>r$.) Hence the vanishing of
\[
\phigeomb{H}(t_{\le p^s}(f_{p,n})\otimes z)=\phigeomb{H}(t_{\le p^s}(f_{p,n}))\otimes \phigeomb{H}(z)
\]
would follow from $\phigeomb{H}(z)\in\cat C_{p,n-(t-s)}$.  Now $s=N-n-1$ by definition, so $n-(t-s)=N-t-1$ and we are left to prove~$\phigeomb{H}(z)\in\cat C_{p,N-t-1}$.  If $N-t-1=0$ then there is nothing to prove.  Otherwise, set $\ell:=N-t-1<N-s-1=n$ and consider the prime $\cat P(H,p,\ell)$. We have $\ell+\log_p|H|=N-t-1+t=N-1<N$ and we can apply~\eqref{eq:m-minimal}, which tells us that $\cat P(H,p,\ell)\notin\adhpt{\cat P(G,p,N-r)}=\supp(z)$. The latter reads $z\in\cat P(H,p,\ell)$ hence $\phigeomb{H}(z)\in\cat C_{p,\ell}=\cat C_{p,N-t-1}$ as desired.

At this stage, we have proved the announced vanishing of the Tate object: $t_{\le p^s}(f_{p,n}\otimes z)=0$ in~$\SHGp$. Tensoring by the dual of~$z$, we get by rigidity
\begin{eqnarray*}
0=t_{\le p^s}(f_{p,n}\otimes z)\otimes z^\vee&=&[f_{\le p^s} \otimes z, \Sigma e_{\le p^s} \otimes z \otimes f_{p,n}] \\
&=&[f_{\le p^s} \otimes z \otimes f_{p,n}, \Sigma e_{\le p^s} \otimes z \otimes f_{p,n}]
\end{eqnarray*}
where the last equality uses Remark~\ref{rem:e-f} again. In the localization $\cat{SH}(G)_{p,n}=f_{p,n}\otimes\SHG$, the last equality implies $\Hom_{\cat{SH}(G)_{p,n}}(f_{\le p^s}\otimes z,\Sigma e_{\le p^s}\otimes z)=0$. Consequently, the triangle $\Delta_{\le p^s}\otimes z$ splits in~$\cat{SH}(G)_{p,n}$, where we thus have:
\[
z\simeq z_1\oplus z_2 \qquadtext{with} z_1=e_{\le p^s}\otimes z \quadtext{and} z_2=f_{\le p^s}\otimes z\,.
\]
Since $z\in\cat{SH}(G)_{p,n}^c$ is compact, so are $z_1$ and $z_2$ and their support must then be contained in~$Y_1:=\SET{\cat P(K,p,n)\in\Spc(\cat{SH}(G)_{p,n}^c)}{|K|\le p^s}$ and $Y_2:=\SET{\cat P(K,p,n)\in\Spc(\cat{SH}(G)_{p,n}^c)}{|K|>p^s}$ respectively. Moreover, the ``bad point" $\cat P(K_0,p,n)$ belongs to~$Y_1$, since $\log_p|K_0|<N-n=s+1$, and $\cat P(G,p,N-r)$ belongs to~$Y_2$. Even more, this point~$\cat P(G,p,N-r)$ is the generic point of~$\supp(z)$, which is irreducible (in~$\Spc(\SHGp)$, hence in the open $\Spc(\cat{SH}(G)_{p,n}^c)$ as well). The non-trivial decomposition $\supp(z)=Y_1\sqcup Y_2$ is therefore absurd. Hence there was no ``bad point" and we have the desired equality $\adhpt{\cat P(G,p,N-r)}=Z_{p,N}$.
\end{proof}

\begin{Cor}
\label{cor:Tate-easy}%
If $G$ is a group of order~$p^r$ which satisfies the $\log_p$-Conjecture~\ref{conj:log-p}, then the \emph{geometric} $G$-fixed points of the Tate construction lowers chromatic degree by~$r$, namely for every $n \ge r$, we have $\phigeomb{G}(t_{G}(\triv(f_{p,n})))\otimes \cat C_{p,n-r}=0$.
\end{Cor}

\begin{proof}
This is the extreme case $s=0$ and $t=r$ in part~\ref{it:Tate-G-for-fn} of Theorem~\ref{thm:equiv-Tate}.
\end{proof}

If one is willing to study the $\log_p$-Conjecture as a whole then \emph{global} geometric fixed points are also enough.

\begin{Cor}
Let $p$ be a prime. The following are equivalent:
\begin{enumerate}[label=\rm(\Alph*)]
\item
\label{it:log-p-Tate-for-all}%
The $\log_p$-Conjecture~\ref{conj:log-p} holds for all $p$-groups.
\item
\label{it:top-Tate-G-for-fn}%
For every group~$G$ of order $p^r$, every $0\le s<r$, and every $n\ge r-s$, we have $\phigeomb{G}(t_{\le p^s}(\triv(f_{p,n})))\otimes \cat C_{p,n-(r-s)}=0$.
\end{enumerate}
\end{Cor}

\begin{proof}
\ref{it:log-p-Tate-for-all}$\then$\ref{it:top-Tate-G-for-fn} is the special case $t=r$ of Theorem~\ref{thm:equiv-Tate}\,\ref{it:Tate-G-for-fn}. The converse \ref{it:top-Tate-G-for-fn}$\then$\ref{it:log-p-Tate-for-all} also follows from Theorem~\ref{thm:equiv-Tate} by induction on~$|G|$. Indeed, the ``missing" cases $\phigeomb{H}(t_{\le p^s}(\triv(f_{p,n})))\otimes \cat C_{p,n-(t-s)}=0$ with $|H|=p^t$ and $t<r$ hold by the induction hypothesis applied to~$H$, using that $\Res^G_H$ preserves $t_{\le p^s}$ (Ex.\,\ref{ex:res-Tate}).
\end{proof}

\smallskip
\section{The classification of tt-ideals in $\SHGc$}
\label{se:classif}
\medskip

As recalled in Remark~\ref{rem:Thom+class}, tt-ideals of~$\SHGc$ are in bijection with the Thomason subsets of~$\Spc(\SHGc)$, \ie arbitrary unions of closed subsets each having quasi-compact complement.

\begin{Prop}
\label{prop:irred-close-quasi}%
Let $Z=\adhpt{\cat P_1}\cup\cdots \cup\adhpt{\cat P_k}$ be a closed subset of~$\Spc(\SHGc)$ (see Prop.\,\ref{prop:finite-union-irred}). We can assume that this is irredundant, that is, $\cat P_i\not\subseteq\cat P_j$ for all $1\le i\neq j\le k$. Let $\cat P_i=\cat P(H_i,p_i,m_i)$ for all $i=1,\ldots, k$. Then the complement of~$Z$ is quasi-compact if and only if all chromatic integers $m_1,\ldots,m_k$ are finite.
\end{Prop}

\begin{proof}
The overall architecture of the proof resembles that of Proposition~\ref{prop:finite-union-irred}. The non-equivariant case ($G=1$) can be found in~\cite[Cor.\,9.5\,(d)]{Balmer10b}.

For a general~$G$, let us show that if all~$m_1,\ldots,m_k$ are finite then $Z$ has quasi-compact complement. As $\SET{\supp(x)}{x\in\SHGc}$ is a closed basis for the topology, we need to prove that whenever $Z=\cap_{i\in I}\supp(x_i)$ for a family $\{x_i\}_{i\in I}$ of objects in~$\SHGc$ there exists a finite subset $J\subseteq I$ such that $Z=\cap_{i\in J}\supp(x_i)$. Since $Z\subseteq\cap_{i\in J}\supp(x_i)$ for any~$J\subseteq I$, we just need to make $J$ big enough so that we get equality. We shall do this by collecting together a finite union of finite subsets of~$I$, one for each subgroup $K\le G$ (up to conjugacy). Indeed, since $\Spc(\SHGc)=\cup_{K\le G}\Img(\varphi^{K,G})$, it suffices to show that we can find a finite subset $J_K\subseteq I$ so that the desired equality holds when intersected with~$\Img(\varphi^{K,G})$:
\begin{equation}
\label{eq:question}%
Z\cap \Img(\varphi^{K,G}) \overset{?}{=} \cap_{i\in J_K} \supp(x_i)\cap \Img(\varphi^{K,G}).
\end{equation}
Since $\varphi^{K,G}=\Spc(\phigeomb{K,G})$ is an inclusion-preserving bijection $\Spc(\SHc)\isoto \Img(\varphi^{K,G})$, with inverse $\Spc(\triv)$ suitably restricted (see Cor.\,\ref{cor:project-H}), we can transport the question of~\eqref{eq:question} into~$\Spc(\SHc)$. Let us study both sides separately.

For $H,K\le G$, $p$ a prime and $1\le m<\infty$ finite, it is a direct consequence of Corollary~\ref{cor:summary} that the intersection of $\adhpt{\cat P(H,p,m)}=\SET{\cat Q}{\cat Q\subseteq \cat P(H,p,m)}$ with $\Img(\varphi^{K,G})=\SET{\cat P(K,q,n)}{q,n}$ is either empty or equals the closure of a single point, namely $\cat P(K,q,n_K)$ in the notation of Corollary~\ref{cor:summary}\,\ref{it:inclusion}; note that this $n_K$ is \emph{finite} since $m$ is. Also note that $\cat P(K,q,n_K)=\varphi^{K,G}(\cat C_{q,n_K})$. Hence
\begin{equation}
\label{eq:irred-K-to-SH}%
(\varphi^{K,G})\inv\big(\adhpt{\cat P(H,p,m)}\cap \Img(\varphi^{K,G})\big)=\adhpt{\cat C_{q,n_K}} \textrm{ or } \varnothing.
\end{equation}
Furthermore, for every object $x\in \SHGc$, we have by a general tt-fact that
\begin{equation}
\label{eq:supp-K-to-SH}%
(\varphi^{K,G})\inv\big(\supp(x)\cap \Img(\varphi^{K,G})\big)=(\varphi^{K,G})\inv(\supp(x))=\supp(\phigeomb{K,G}(x)).
\end{equation}
We can now use~\eqref{eq:irred-K-to-SH} and~\eqref{eq:supp-K-to-SH} to reduce the question of~\eqref{eq:question} to the known non-equivariant case, thus producing the desired finite $J_K\subseteq I$. Hence the result.

Conversely, suppose that one of the $m_i=\infty$. Write $Z=Z'\cup \adhpt{\cat P(H,p,\infty)}$, regrouping the other irreducibles under~$Z'$. Since $\cat P(H,p,\infty)=\cap_{1\le n<\infty}\cat P(H,p,n)$, we have $Z=\cap_{1\le n<\infty}(Z'\cup \adhpt{\cat P(H,p,n)})$. If the complement of~$Z$ is quasi-compact then there is a finite~$n$ such that $Z=Z'\cup\adhpt{\cat P(H,p,n)}$. But then $\cat P(H,p,n)\in Z=Z'\cup\adhpt{\cat P(H,p,\infty)}$ and since $\cat P(H,p,n)\not\subseteq\cat P(H,p,\infty)$, it follows that $\cat P(H,p,n)\in Z'$. However this forces $\cat P(H,p,\infty)\in \adhpt{\cat P(H,p,n)}\subseteq Z'$, or in other words $\cat P(H,p,\infty)$ is contained in one of the other irreducibles of~$Z$, contradicting the assumption that the $\cat P_1,\ldots,\cat P_k$ were irredundant.
\end{proof}

\begin{Cor}
\label{cor:Thom}%
The Thomason subsets of~$\Spc(\SHGc)$ are the arbitrary unions of~$\adhpt{\cat P(H,p,m)}=\SET{\cat Q}{\cat Q\subseteq\cat P(H,p,m)}$ for arbitrary $H\le G$, $p$ prime and \emph{finite} chromatic integer~$1\le m<\infty$.
\end{Cor}

\begin{proof}
This is direct from Proposition~\ref{prop:finite-union-irred} and
Proposition~\ref{prop:irred-close-quasi}.
\end{proof}

\begin{Cor}[Classification of tt-ideals in~$\SHGc$]
\label{cor:classif}%
Let $G$ be a finite group and let $\bbN_{\infty}:=\{0,1,2,\ldots\}\cup\{\infty\}$ with the obvious relation~$<$ (setting $\infty\not<\infty$). Consider all functions $f:\SET{(H,p)}{H\le G\mathrm{\ and\ }p\mathrm{\  prime}}\to \bbN_{\infty}$ and let us say that $f$ is \emph{admissible} if it satisfies the following property:
\begin{enumerate}
\item[(A)]
If $\cat P(K,q,n) \subseteq \cat P(H,p,m)$ and $m > f(H,p)$ then $n > f(K,q)$.
\end{enumerate}
If the $\log_p$-Conjecture~\ref{conj:log-p} is satisfied by every $p$-group subquotient of $G$ (for all primes $p$), then this property can equivalently be expressed without any knowledge of inclusions between equivariant primes as follows:
\begin{enumerate}
\item[(A\!')]
	For every prime~$p$ and for every $p$-subnormal $K\le H$, we have $f(K,p)\le f(H,p)+\log_p[H:K]$.
	Moreover, if $f(H,p)=0$ for any pair $(H,p)$ then $f(H,q) = 0$ for all primes~$q$.
\end{enumerate}
Independently of the $\log_p$-Conjecture, there is a one-to-one correspondence between such admissible functions and Thomason subsets of~$\Spc(\SHGc)$ mapping $f$ to $Y_f:=\SET{\cat P(H,p,m)}{m > f(H,p)}$. Consequently, there is a one-to-one correspondence between admissible functions and tt-ideals in~$\SHGc$ mapping~$f$ to
\[
\cat J_f:=\SET{x\in\SHGc}{\ \phigeomb{H}(x)\in\cat C_{p,f(H,p)}\ ,\ \forall\,H\le G\mathrm{\ and\ }p\textrm{\ such\ that\ } f(H,p) >0}.
\]
\end{Cor}

\begin{proof}
If $Y\subseteq\Spc(\SHGc)$ is a Thomason subset, then $Y=\cup_{H\le G,\ p\textrm{ prime}}X(H,p)$ where $X(H,p):=\SET{\cat P(H,p,n)}{\cat P(H,p,n)\in Y}$. By Corollaries~\ref{cor:summary} and~\ref{cor:Thom}, we see that each non-empty $X(H,p)$ is of the form $\SET{\cat P(H,p,n)}{n_{H,p}\le n\le \infty}$ for some finite $n_{H,p}\ge 1$. The link with our function~$f$ is simply that we set $f(H,p)=\infty$ when $X(H,p)=\varnothing$ and $f(H,p)=n(H,p)-1$ when $X(H,p)$ is non-empty. This is made in such a way that $X(H,p)=\SET{\cat P(H,p,n)}{n>f(H,p)}$ in both cases. The definitions of ``admissible" are translations of Corollary~\ref{cor:summary}.

The classification of tt-ideals then follows from general tt-geometry (see Remark~\ref{rem:Thom+class}) since one easily checks that $\cat J_f=\cat K_{Y_f}$ for $\cat K=\SHGc$.
\end{proof}


\end{document}